\documentclass[11pt,letter]{amsart}

\usepackage{amsmath}
\usepackage{amscd}
\usepackage{amssymb}
\usepackage{amsthm}
\usepackage{amsfonts}
\usepackage[latin1]{inputenc}
\usepackage{mathrsfs}
\usepackage{marvosym}
\usepackage{srcltx}
\usepackage[centering,text={15.5cm,24cm}]{geometry}
\usepackage[dvips]{graphicx}
\usepackage{pstricks}
\usepackage[all]{xy}
\usepackage[colorlinks=true]{hyperref}
\hypersetup{
  colorlinks   = true,          
  urlcolor     = violet,          
  linkcolor    = blue,          
  citecolor   = violet             
}
\newtheorem{theorem}{Theorem}[section]
\newtheorem{lemma}[theorem]{Lemma}
\newtheorem{proposition}[theorem]{Proposition}
\newtheorem{corollary}[theorem]{Corollary}
\newtheorem{conjecture}[theorem]{Conjecture}

\theoremstyle{definition}
\newtheorem{definition}[theorem]{Definition}
\newtheorem{construction}[theorem]{Construction}
\newtheorem{assumption}[theorem]{Assumption}

\newtheorem{example}[theorem]{Example}

\theoremstyle{remark}
\newtheorem{remark}[theorem]{Remark}

\numberwithin{equation}{section}
\numberwithin{figure}{section}

\newcommand{\bary}{\mathrm{bar}}
\newcommand{\length}{\operatorname{length}}
\newcommand{\cover}{\tau^{\mathbf w}_s}
\newcommand{\Cover}{\Delta^{\mathbf w}_{s} } 
\newcommand{\ray}{\frak s}
\newcommand{\defbase}{S^{\mathbf w}}
\newcommand{\defcone}{C^{\mathbf w}}
\newcommand{\effvert}{\widetilde \Gamma^{[0]}}
\newcommand{\vvert}{\tilde \Gamma^{[0]}_{\mathrm{leaf}}}

\newcommand{\doublebar}[1]{\shortstack{$\overline{\hphantom{#1}}
	$\\[-4pt] $\overline{#1}$}}

\newcommand{\bbfamily}{\fontencoding{U}\fontfamily{bbold}\selectfont}
\newcommand{\textbb}[1]{{\bbfamily#1}}

\newcommand {\lfor} {\mbox{\textbb{[}}}
\newcommand {\rfor} {\mbox{\textbb{]}}}

\newcommand{\NN} {\mathbb{N}}
\newcommand{\ZZ} {\mathbb{Z}}
\newcommand{\QQ} {\mathbb{Q}}
\newcommand{\RR} {\mathbb{R}}

\newcommand{\FF} {\mathbb{F}}
\newcommand{\PP} {\mathbb{P}}
\renewcommand{\AA} {\mathbb{A}}
\newcommand{\GG} {\mathbb{G}}

\newcommand {\shD}  {\mathcal{D}}

\newcommand {\shHom} {\mathcal{H}\!\text{\textit{om}}}

\newcommand {\shL}  {\mathcal{L}}

\newcommand {\shQ}  {\mathcal{Q}}

\newcommand {\shW}  {\mathcal{W}}
\newcommand {\shX}  {\mathcal{X}}

\newcommand {\shZ}  {\mathcal{Z}}

\newcommand {\foB}  {\mathfrak{B}}

\newcommand {\foD}  {\mathfrak{D}}

\newcommand {\foT}  {\mathfrak{T}}
\newcommand {\foU}  {\mathfrak{U}}

\newcommand {\foX}  {\mathfrak{X}}

\newcommand {\fob}  {\mathfrak{b}}
\newcommand {\foc}  {\mathfrak{c}}

\newcommand {\foj}  {\mathfrak{j}}

\newcommand {\fop}  {\mathfrak{p}}

\newcommand {\forr}  {\mathfrak{r}}

\newcommand {\fot}  {\mathfrak{t}}
\newcommand {\fou}  {\mathfrak{u}}

\newcommand {\Aut}  {\operatorname{Aut}}

\newcommand {\codim} {\operatorname{codim}}

\newcommand {\conv} {\operatorname{conv}}

\newcommand {\depth} {\operatorname{depth}}

\newcommand {\eps}  {\varepsilon}

\newcommand {\Hom}  {\operatorname{Hom}}

\newcommand {\id}  {\operatorname{id}}
\newcommand {\im}  {\operatorname{im}}

\newcommand {\Int}  {\operatorname{Int}}

\newcommand {\Joints} {\operatorname{Joints}}

\newcommand {\kk} {\Bbbk}

\newcommand {\leafvertices}  {\Gamma^{[0]}_{\mathrm{leaf}}}

\newcommand {\lra}  {\longrightarrow}
\newcommand {\ls}  {\dagger}

\newcommand {\M} {\mathcal{M}}

\renewcommand {\max} {{\operatorname{max}}}

\renewcommand{\O}  {\mathcal{O}}

\newcommand {\ol} {\overline}
\newcommand {\ord}  {\operatorname{ord}}

\renewcommand{\P}  {\mathscr{P}}

\newcommand {\rk} {\operatorname{rk}}

\newcommand {\rootvertex}  {V_{\mathrm{root}}}
\newcommand {\scrR}  {\mathscr{R}}

\newcommand {\scrB}  {\mathscr{B}}
\newcommand {\scrS}  {\mathscr{S}}

\newcommand {\Spec} {\operatorname{Spec}}
\newcommand {\Spf}  {\operatorname{Spf}}

\newcommand {\vmult} {\operatorname{vmult}}

\newcommand {\D} {\mathfrak D}

\newcommand {\X} {\mathfrak X}

\newcommand {\Z} {\mathfrak Z}

\newcommand{\set}[2]{ \left\{ \left. {#1}  \; \right| \;\: {#2} \right\} }

\newcommand{\targetring}{R_{g,\sigma}^k}

\def\mydate{\ifcase\month \or January\or February\or March\or
April\or May\or June\or July\or August\or September\or October\or 
November\or December\fi \space\number\day,\space\number\year}
\long\def\symbolfootnote[#1]#2{\begingroup%
\def\thefootnote{\fnsymbol{footnote}}\footnote[#1]{#2}\endgroup} 

\begin{document}

\title[Tropical Landau-Ginzburg]{A tropical view on
Landau-Ginzburg models}

\author{Michael Carl}
\address{\tiny Alfred-Delp-Str.~2, 73430 Aalen, Germany}
\email{carlmi@gmx.net}

\author{Max Pumperla} 
\address{\tiny IU Internationale Hochschule, Juri-Gagarin-Ring~152, 99084~Erfurt, Germany\\
Anyscale, Inc., 2080 Addison St Ste 7, Berkeley, CA 94704, USA}
\email{max.pumperla@googlemail.com}

\author{Bernd Siebert} 
\address{\tiny Department of Mathematics, The Univ.\ of Texas at Austin,
2515 Speedway, Austin, TX 78712, USA}
\email{siebert@math.utexas.edu}

\begin{abstract}
This paper, largely written in 2009/2010, fits Landau-Ginzburg models into the
mirror symmetry program pursued by the last author jointly with Mark Gross since
2001. This point of view transparently brings in tropical disks of Maslov
index~$2$ via the notion of broken lines, previously introduced
in two dimensions by Mark Gross in his study of mirror symmetry for $\PP^2$.

A major insight is the equivalence of properness of the Landau-Ginzburg
potential with smoothness of the anticanonical divisor on the mirror side. We
obtain proper superpotentials which agree on an open part with those classically
known for toric varieties. Examples include mirror LG models for non-singular
and singular del Pezzo surfaces, Hirzebruch surfaces and some Fano threefolds.
\end{abstract}

\thanks{M.P.\ was supported by the \emph{Studienstiftung des
deutschen Volkes}; research by B.S.\ was supported by NSF grant DMS-1903437.}

\date{\today}
\maketitle

\tableofcontents


\section*{Preface}
This paper, largely written in 2009/2010, investigates the incorporation of
Landau-Ginzburg models into the toric degeneration approach to mirror symmetry
of the last author with Mark Gross \cite{logmirror,affinecomplex}. At the time
we could not answer a key question concerning the existence of the algorithm in
\cite{affinecomplex} in the relevant unbounded case. We also felt that our
construction poses many interesting questions and more should be said. With two
of the authors leaving academia (M.C.~2010, M.P.~2011), the paper was eventually
left in preliminary form on the last author's Hamburg webpage in the state of
September~2010.\footnote{The last named author presented our findings at the
workshop ``Derived Categories, Holomorphic Symplectic Geometry, Birational
Geometry, Deformation Theory'' at IHP/Paris in May 2010 and at VBAC~2010 in
Lisbon in June~2010.} This preliminary version will remain available as an
ancillary file in the arXiv-submission. A variation of the last three sections
partly treating other cases appeared as part of the second author's
doctoral thesis \cite{MaxThesis}.

The key consistency proof of the superpotential in \S3 and \S4 has been
repeatedly used in the construction of generalized theta functions in the
surface and cluster variety cases, notably in \cite{GHK,GHKK}. The question of
existence of a consistent wall structure in the non-compact case eventually
turned out to be best answered trivially by a compactification requirement
(Definition~\ref{Def: compactifiable (B,P)}, Proposition~\ref{Prop: wall
structures exists for compactifiable cases}), or in the context of intrinsic
mirror symmetry \cite{Assoc, CanonicalWalls}. With a recent increase in studies
of the smooth anticanonical divisor case, the case of central interest in this
paper, it now seems the right time to finalize this paper.

To preserve the line of historical developments, we have mostly only done minor
edits for accuracy and clarity. The exceptions are the mentioned
compactification criterion in \S1, Remark~\ref{Rem: algebraizability} on
algebraizability of the superpotential, a new section on the fibers of the
superpotential (\S5), and a corrected treatment of the mirror of Hirzebruch
surfaces taken from \cite[\S5.3]{MaxThesis}. We have also added some references
to newer developments in the introduction and in some footnotes, but this did
not seem the right place to give a comprehensive overview and due credit to all
the wonderful developments that have happened around the topic in the last
decade. We apologize with everybody whose newer contributions are not mentioned
in this paper.

\S5 existed in some form for many years and has been distributed occasionally,
but the LG mirror map (Definition~\ref{Def: LG mirror map} and Theorem~\ref{Thm:
W versus w}) has only rather recently been spelled out in discussions with Helge
Ruddat in a joint project with Michel van Garrel on enumerative period integrals in Landau-Ginzburg models \cite{GRS}.


\section*{Introduction}
Mirror symmetry has been suggested both by mathematicians \cite{Givental} and
physicists \cite{Witten,HoriVafa} to extend from Calabi-Yau varieties to a
correspondence between Fano varieties and Landau-Ginzburg models. Mathematically
a Landau-Ginzburg model is a non-compact K\"ahler manifold with a holomorphic
function, the superpotential. Until recently, the majority of studies confined
themselves to toric cases where the construction of the mirror is immediate. The
one exception we are aware of is the work of Auroux, Katzarkov and Orlov on
mirror symmetry for del Pezzo surfaces \cite{AKO}, where a symplectic mirror is
constructed by a surgery construction. The general Floer-theoretic perspective
for the mirror Landau-Ginzburg model of an SYZ fibered logarithmic Calabi-Yau
manifold has been discussed by Auroux in \cite{auroux1,auroux2}. In the
following we use the phrase log Calabi-Yau to refer to a pair $(X,D)$ of a
complete variety $X$ over a field with a non-zero effective anticanonical
divisor $D\subset X$.\footnote{In \cite{affinecomplex} log Calabi-Yau varieties
were referred to as Calabi-Yau pairs to avoid confusion with the central fiber
of toric degenerations of Calabi-Yau varieties.}

The purpose of this paper is to fit the Fano/Landau-Ginzburg mirror
correspondence into the mirror symmetry program via toric degenerations pursued
by the last author jointly with Mark Gross \cite{logmirror,affinecomplex}. The
program as it stands suggests a non-compact variety as the mirror of a log
Calabi-Yau variety, or rather toric degenerations of these varieties. So the
main new ingredient is the construction of the superpotential.

The key technical idea of \emph{broken lines} (Definition~\ref{def: broken
lines}) for the construction of the superpotential has already appeared in a
different context in the two-dimensional situation in Gross' mirror
correspondence for $\PP^2$ \cite{PP2mirror}. We replace his case-by-case study
of well-definedness with a scattering computation, making it work in all
dimensions.

Our main findings can be summarized as follows.
\begin{enumerate}
\item
From our point of view, the natural data on the Fano side is a toric degeneration
of log Calabi-Yau varieties as defined in \cite{affinecomplex}, Definition~1.8. In
particular, if arising from the localization of an algebraic family, the general
fiber is a pair $(\check X,\check D)$ of a complete variety $\check X$ and a
reduced anticanonical divisor $\check D$. No positivity property is ever used
in our construction apart from effectivity of the anticanonical bundle.
\item
The mirror is a toric degeneration of non-compact Calabi-Yau varieties, together
with a canonically defined regular function on the total space of the
degeneration (Proposition~\ref{Prop: wall structures exists for compactifiable cases}).
\item
The superpotential is proper if and only if the anticanonical divisor $\check D$
on the log Calabi-Yau side is locally irreducible (Theorem~\ref{Thm:
Properness}). These conditions also have clean descriptions on the underlying
tropical models governing the mirror construction from
\cite{logmirror,affinecomplex}. \S3 and \S4 give the all order construction of the superpotential, summarized in Theorem~\ref{Thm: all order W}.
\item
For smooth toric Fano varieties our construction provides a canonical (partial)
compactification of the Hori-Vafa construction \cite{HoriVafa}
(Corollary~\ref{Cor: reproduce Hori-Vafa} in the surface case,
\cite[Thm.\,5.4]{MaxThesis} in all dimensions). But note Remark~\ref{Rem:
algebraizability} concerning the general question of algebraizability of the
superpotential.
\item
The terms in the superpotential can be interpreted in terms of virtual numbers
of tropical disks, at least in dimension two (Proposition~\ref{Prop: Virtual
count}). On the Fano side these conjecturally count holomorphic disks with
boundary on a Lagrangian torus.\footnote{This picture has recently been
confirmed in \cite{CanonicalWalls} in terms of punctured Gromov-Witten
invariants.}
\item
The natural holomorphic parameters occurring in the construction on the Fano
side lie in $H^1(\check X_0,\Theta_{(\check X_0,\check D_0)})$ where
$\Theta_{(\check X_0,\check D_0)}$ is the logarithmic tangent bundle of the
central fiber $(\check X_0,\check D_0)$ in (1). This group rules infinitesimal
deformations of the pair $(\check X_0,\check D_0)$. We have not carefully
analyzed the parameters on the Landau-Ginzburg side and their correspondence to
the K\"ahler parameters on the log Calabi-Yau side.\footnote{The correspondence
is transparent from the intrinsic mirror symmetry perspective
\cite{Assoc,CanonicalWalls}. } Note however that all parameters come from
deformations of the underlying space, our superpotential does not add extra
parameters.
\item
Explicit computations include non-singular and singular del Pezzo surfaces, the
Hirzebruch surfaces $\FF_2$ and $\FF_3$, $\PP^3$ and a singular Fano
threefold (\S7--\S8).
\end{enumerate}
Throughout we work over an algebraically closed field $\kk$ of characteristic
$0$. We use check-adorned symbols $\check X,\check D, \check\X,\check\D, \check
B$ etc.\ for the log Calabi-Yau side, and unadorned symbols for the
Landau-Ginzburg side. For an integral polyhedron $\tau\subset\RR^n$ we denote by
$\Lambda_\tau$ the free abelian group of integral tangent vector fields on
$\tau$.

We would like to thank Denis Auroux, Mark Gross and Helge Ruddat for valuable
discussions. 


\section{Landau-Ginzburg tropical data and scattering diagram}
Throughout the paper we assume familiarity with the basic notions of toric
degenerations from \cite{logmirror}, as reviewed in
\cite[\S\S1.1--1.2]{affinecomplex}. We quickly review the basic picture. Let
$(\check B,\check \P,\check\varphi)$ be the polarized intersection complex or
\emph{cone picture} \cite[Expl.\,1.13]{affinecomplex} associated to a (schematic
or formal) proper polarized toric degeneration $(\check\pi:\check\X\to
T,\check\D)$ of log Calabi-Yau varieties \cite[Defs.\,1.8/1.9]{affinecomplex}.
Here $T$ is the (formal) spectrum of a discrete valuation $\kk$-algebra,
typically $\kk\lfor t\rfor$. Recall that $\check B$ is a topological manifold
with a $\ZZ$-affine structure outside a codimension two cell complex
$\check\Delta\subset \check B$, also called the discriminant locus; $\check\P$
is a decomposition of $\check B$ into integral, convex, but possibly unbounded
polyhedra containing $\check\Delta$ as subcomplex of the first barycentric
subdivision disjoint from vertices and the interiors of maximal cells; and
$\check\varphi$ is a (generally multivalued) strictly convex piecewise linear
function with integral slopes. The irreducible components of the central fiber
$\check X_0\subset\check\X$ are the toric varieties with momentum polytopes the
maximal cells in $\check \P$, and lower dimensional cells describe their
intersections.

Equivalently, one has the discrete Legendre dual data $(B,\P,\varphi)$, referred
to as the dual intersection complex or \emph{fan picture} of the same
degeneration \cite[Expl.\,1.11]{affinecomplex}, or the cone picture of the
mirror via discrete Legendre duality \cite[Constr.\,1.16]{affinecomplex}.

While \cite{logmirror} only treated the case of trivial canonical bundle or
closed $B$, \cite{affinecomplex} gave the straightforward generalization to the
case of interest here of log Calabi-Yau varieties, that is, a variety $\check X$
and an anticanonical divisor $\check D\subseteq \check X$. These correspond to
compact $\check B$ with locally convex boundary $\partial \check B$, with
$\partial\check B\neq\emptyset$ iff $\check D\neq\emptyset$. It holds $\partial
\check B\neq \emptyset$ iff the discrete Legendre-dual $B$ is non-compact. The
proof of \cite[Thm.\,5.4]{logmirror} then still shows that under a primitivity
assumption on the local affine geometry (``simple singularities,
\cite[Def.\,1.60]{logmirror}), the corresponding central fibers $(\check
X_0,\M_{\check X_0})$ of toric degenerations of log Calabi-Yau varieties
$(\check \X\to T,\check \D)$, as a log space, are classified by the cohomology
groups
\[
H^1(B, i_*\Lambda_B\otimes_\ZZ\GG_m(\kk))= H^1(\check
B,\check\imath_*\check\Lambda_{\check B}\otimes_\ZZ\GG_m(\kk))
\]
computed from the affine geometry of $B$ or $\check B$. Here $\Lambda_B$ denotes
the sheaf of integral tangent vectors on the complement of the real codimension
two singular locus $\Delta\subseteq B$, $i:B\setminus\Delta\to B$ is the
inclusion, and $\check\Lambda_B=\shHom(\Lambda_B,\ZZ_B)$. Similar notations
apply to $\check B$. The correspondence works via twisting toric patchings of
standard affine toric models for the degeneration via so-called \emph{open
gluing data} $s$ \cite[Def.\,1.18]{affinecomplex} and showing that any choice of
$s$ gives rise to a log structure on $\check X_0=\check X_0(s)$ over the
standard log point. This log structure is unique up to isomorphisms fixing $\check
X_0$. Unlike in abstract deformation theory, the space of deformations is not
just a torsor over the controlling cohomology group, but a group itself. In
particular, trivial gluing data $s=1$ lead to a distinguished choice of log
Calabi-Yau central fiber $(\check X_0,\M_{\check X_0})$. The log structure also
carries the information of the anticanonical divisor $\check D_0\subseteq \check
X_0$, which hence is suppressed in the notation.

Conversely, \cite[Thm.\,3.1]{affinecomplex} constructs a canonical formal
polarized toric degeneration $\pi:(\check\X\to T,\check\D)$ with given
logarithmic central fiber $(\check X_0,\M_{\check X_0})$, even under a weaker
assumption (\emph{local rigidity}, \cite[Def.\,1.26]{affinecomplex}) than
simplicity. Thus we understand this side of the mirror correspondence rather
well.

The objective of this paper is to give a similarly canonical picture on the
mirror side. The mirrors of Fano varieties are suggested to be so-called
Landau-Ginzburg models (LG-models). Mathematically these are non-compact
algebraic varieties with a holomorphic function, referred to as
\emph{superpotential}, see e.g.\
\cite{HoriVafa,ChoOh,AKO,FOOO1}. Following the general
program laid out in \cite{logmirror} and \cite{affinecomplex}, we construct
LG-models via deformations of a non-compact union of toric varieties. The
superpotential is constructed by canonical extension from the central fiber.

Our starting point in this paper is as follows.

\begin{assumption}
\label{overall assumption}
Let $(B,\P,\varphi)$ be a non-compact, polarized, (integral) tropical manifold
without boundary \cite[Def.1.2]{affinecomplex}. We further assume that
$(B,\P,\varphi)$ comes with a compatible sequence of consistent (wall)
structures $\scrS_k$ as defined in \cite[Defs.\,2.22, 2.28,
2.41]{affinecomplex}, for some choice of open gluing data $s$.
\end{assumption}

For applications in mirror symmetry, $(B,\P,\varphi)$ is the Legendre-dual of a
compact $(\check B,\check \P,\check \varphi)$ with locally convex boundary.
Unfortunately, the algorithmic construction of consistent structures in
\cite{affinecomplex} is problematic at several places in the non-compact case.

The only practical general assumption we are aware of to make the algorithm work
in the non-compact case adds the following convexity requirement at infinity.

\begin{definition}
\label{Def: compactifiable (B,P)}
We call a tropical manifold $(B,\P)$ with or without boundary
\emph{compactifiable} if there exists a compact subset $K\subseteq B$ containing
a neighborhood of the union of all bounded cells of $\P$ and a proper
continuous map $\psi: B\setminus K\to \RR_{\ge 0}$ with the following
properties: (1)~$\psi$ is locally on $B\setminus (K\cup\Delta)$ a convex
function. (2)~Each unbounded cell $\sigma\in\P$ has a finite integral polyhedral
decomposition $\P_\sigma$ such that $\psi|_{\sigma\cap (B\setminus K)}$ is a
convex, integral, piecewise affine function with respect to $\P_\sigma$.
\end{definition}

The existence of $\psi$ in Definition~\ref{Def: compactifiable (B,P)} makes it possible to exhaust $B$ by tropical manifolds as follows. 

\begin{lemma}
\label{Lem: Exhaustion}
Assume that $(B,\P)$ is compactifiable (Definition~\ref{Def: compactifiable
(B,P)}). Then there exists a sequence of compact subsets $\ol B_1\subseteq\ol
B_2\subseteq \ldots\subseteq B$ with (1)~$B=\bigcup_{\nu\ge 1} \ol B_\nu$, and
(2)~$(\ol B_\nu,\ol\P_\nu)$ with $\ol\P_\nu=\big\{\sigma\cap\ol B_\nu\,\big|\,
\sigma\in\P\big\}$ is a tropical manifold in the sense of
\cite[Def.\,1.32]{affinecomplex}. 
\end{lemma}

\begin{proof}
Let $K\subseteq B$ and $\psi:B\setminus K\to \RR_{\ge0}$ be as in
Definition~\ref{Def: compactifiable (B,P)}. Define $\ol B_\nu= K\cup
\psi^{-1}\big([0,\nu]\big)$. Properness of $\psi$ implies $B=\bigcup_\nu \ol
B_\nu$ as claimed in (1). Next observe that all bounded cells of $\P$ are
contained in all $\ol B_\nu$ since they are contained in $K$. For an unbounded
cell $\sigma\in\P$, Definition~\ref{Def: compactifiable (B,P)},(2) implies that
the intersection $\sigma\cap \ol B_\nu$ is a compact convex polyhedron defined
over $\QQ$. The denominators appearing in the vertices of
$\partial(\sigma\cap\ol B_\nu)$ disappear when going over to an appropriate
multiple of $\nu$. Thus up to going over to a subsequence, $\P$ induces an
integral polyhedral decomposition $\ol\P_\nu$ of $\ol B_\nu$.

Convexity at boundary points as required in \cite[Def.\,1.32,(2)]{affinecomplex}
follows from convexity of $\psi$ posited in Definition~\ref{Def: compactifiable
(B,P)},(1), provided $\partial\ol B_\nu\cap K\neq\emptyset$. The last condition
clearly holds for $\nu\gg0$, hence can be achieved for all $\nu$ by relabelling
the $\ol B_\nu$.
\end{proof}

The proof of Lemma~\ref{Lem: Exhaustion} motivates the following definition.

\begin{definition}
\label{Def: truncation}
We call a tropical manifold $(\ol B,\ol \P)$, with compact $\ol B\subset B$
containing all bounded cells of $\P$ and intersecting the interior of each
unbounded cell, a \emph{truncation} of $(B,\P)$.
\end{definition}

With an exhaustion as in Lemma~\ref{Lem: Exhaustion} we can now apply the main
theorem of \cite{affinecomplex} to produce a compatible sequence of consistent wall structures on $(B,\P,\varphi)$.

\begin{proposition}
\label{Prop: wall structures exists for compactifiable cases}
Let $(B,\P,\varphi)$ be a polarized, tropical manifold with $(B,\P)$
compactifiable. Assume further that each $(\ol B_\nu,\ol \P_\nu)$ from
Lemma~\ref{Lem: Exhaustion} describes the intersection complex of a locally
rigid, positive, pre-polarized toric log Calabi-Yau variety $(\ol X_0,\ol B_0)$
\cite[Defs.\,1.4, 1.26, 1.23]{affinecomplex}\footnote{The assumptions are
fulfilled for example if $(B,\P)$ has simple singularities.}. Then there exists
a compatible sequence of consistent (wall) structures $\scrS_k$ on
$(B,\P,\varphi)$.

The formal toric degeneration \cite[Def.\,1.9]{affinecomplex} $\pi:
(\X,\D)\to\Spf\kk\lfor t\rfor$ defined by the $\scrS_k$ according to
\cite[Prop.\,2.42]{affinecomplex} has an open embedding into a proper formal
family $(\ol\X,\ol\D)\to\Spf\kk\lfor t\rfor$, with central fiber $\ol X_0$.

Moreover, assuming $H^1(\ol X_0,\O_{\ol X_0})=H^2(\ol X_0,\O_{\ol X_0})=0$, this
family is algebraizable: There exists a proper flat morphism
$\ol\pi:(\ol\shX,\ol\shD)\to \Spec\kk\lfor t\rfor$, a toric degeneration in the
sense of \cite[Def.1.8]{affinecomplex}, and a divisor $\shZ\subset \ol\shX$ flat
over $\Spec\kk\lfor t\rfor$ such that $\pi$ is isomorphic to the completion
of $\ol\pi|_{\ol\shX\setminus\shZ}$ at the central fiber.
\end{proposition}

\begin{proof}
For each $\nu$, \cite[Thm.\,3.1]{affinecomplex} produces a compatible sequence
$\scrS_{k,\nu}$ of consistent wall structures on $(\ol B_\nu,\ol\P_\nu)$.
Comparing the inductive construction of $\scrS_{k,\nu}$ for fixed $k$, but
taking $\nu\to\infty$ shows that the sets of walls stabilize on any compact
subset of $B$. Note that by convexity no walls emanate from $\partial \ol
B_\nu$. We can therefore take the limit over $\nu$ to define $\scrS_k$ on $B$.
Mutual compatibility and consistency follows by consideration on compact subsets
of $B$, using consistency and compatibility for the wall structure on $(\ol
B_\nu,\ol\P_\nu)$ for $\nu$ sufficiently large.

The compactification $(\ol\X,\ol\D)$ is constructed by observing that for each
$k$ the proper schemes over $\kk[t]/(t^{k+1})$ obtained from $\scrS_{k,\nu}$
stabilize for $\nu\to\infty$. Taking this limit first and then the limit over
$k$ now defines the formal scheme $\ol\X$ proper over $\Spf\kk\lfor t\rfor$ and
containing $(\X,\D)$ as an open dense formal subscheme.

With the additional assumptions, algebraizability follows from the Grothendieck
algebraization theorem as in \cite[Cor.\,1.31]{affinecomplex}.
\end{proof}

\begin{remark}
\label{Rem: compactifying divisor}
The compactifying divisor $\Z\subset\ol\X$ with $\X=\ol\X\setminus\Z$ can be
described by inspection of the polyhedral decomposition $\P_\nu$ of $\ol B_\nu$
for any sufficiently large $\nu$ and the asymptotic behavior of $\scrS_k$. For
each sequence $F=(F_\nu)_\nu$ of mutually parallel maximal flat affine subsets
$F_\nu\subseteq\partial\ol B_\nu$ there exists a maximal closed reduced
subscheme $\Z_F\subseteq\Z$, and $\Z$ is the schematic union of the $\Z_F$. The
restriction of $\pi$ to $\Z_F$ is itself a toric degeneration defined by the
walls of $\scrS_k$ that intersect $F_\nu$ for all $\nu$, with wall functions
obtained by disregarding all monomials not tangent to $F_\nu$.

These statements follow directly from the gluing construction
\cite[\S2.6]{affinecomplex}. See \S\ref{sect: special fiber} for a further
discussion in the context of fibers of the superpotential.
\end{remark}


\section{The superpotential at $t$-order zero}
Recall from \cite[Constr.\,2.7]{affinecomplex} the definition of the rings
$R_{g,\sigma}^k$ used to build $\X$ over $\kk[t]/(t^{k+1})$. These depend on an
inclusion $g:\omega\to \tau$ of cells $\omega,\tau\in\P$ and a maximal reference
cell $\sigma$ containing $\tau$ (hence $\omega$). The ring $R_{g,\sigma}^k$
provides the $k$-th order thickening inside $\X$ of the stratum $X_\tau\subseteq
X_0$ with momentum polytope $\tau$ in an affine open neighbourhood of the
$\omega$-stratum $X_\omega\subseteq X_\tau$. The order is measured with respect
to all irreducible components of $X_0$ containing the $\tau$-stratum. The
reference cell $\sigma$ is necessary to fix the affine chart to work on. The
rings obtained in this way from affine toric models are then localized at all
functions carried by codimension one cells (``\emph{slabs}'') containing
$\tau$.

The details of this construction are rather irrelevant for what follows except
that the ring $R^k_{g,\sigma}$ is a localization of a monomial ring, with each
monomial $z^m$ having an associated underlying tangent vector $\ol
m\in\Lambda_\sigma$, and an order of vanishing $\ord_{\sigma'} z^m$ for each
maximal cell $\sigma'\supseteq\tau$. If $\tau=\sigma$ and $\ord_\sigma z^m=l$
then $z^m$ restricts to $z^{\ol m}t^l$ in the Laurent polynomial ring
$R^k_{\id_\sigma,\sigma}= A[\Lambda_\sigma]$ describing the trivial $k$-th order
deformation of the big cell of the irreducible component $X_\sigma\subseteq X_0$
defined by $\sigma$ over $A=\kk[t]/(t^{k+1})$.

We need the following definition.

\begin{definition}
\label{Def: parallel edges}
We call unbounded edges $\omega,\omega'\in\P$ \emph{parallel} if there exists a
sequence of unbounded edges $\omega=\omega_0,\omega_1,\ldots,\omega_r=\omega'$
and maximal cells $\sigma_1,\ldots,\sigma_r\in\P$ with
$\omega_{i-1},\omega_i\subseteq\sigma_i$ parallel
($\Lambda_{\omega}=\Lambda_{\omega'}$ as sublattices of $\Lambda_\sigma$) and
unbounded in the same direction.

A tropical manifold $(B,\P)$ with all unbounded edges parallel is
called \emph{asymptotically cylindrical}.
\end{definition}

Let now $\sigma\in\P$ be an unbounded maximal cell. For each unbounded edge
$\omega\subseteq \sigma$ there is a unique monomial $z^{m_\omega}\in
R^0_{\id_\sigma,\sigma}$ with $\ord_\sigma (m_\omega)=0$ and
$-\overline{m_\omega}$ a primitive generator of $\Lambda_\omega\subseteq
\Lambda_\sigma$ pointing in the unbounded direction of $\omega$. Denote by
$\scrR(\sigma)$ the set of such monomials $m_\omega$. We identify monomials for
parallel unbounded edges $\omega,\omega'$. So these contribute only one exponent
$m_\omega= m_{\omega'}$ to $\scrR(\sigma)$.

Now at any point of $\partial\sigma$, the tangent vector $-\overline{m_\omega}$
points into $\sigma$. Hence
\begin{equation}
\label{Eqn: W^0(sigma)}
W^0(\sigma):= \sum_{m\in \scrR(\sigma)}
z^m
\end{equation}
extends to a regular function on the component $X_{\sigma}\subseteq X_0$
corresponding to $\sigma$. For bounded $\sigma$ define $W^0(\sigma)=0$. To
simplify the following discussion, from now on we only consider the following
situation.\footnote{This assumption was missing in preliminary versions of this
paper and fixes a subtlety arising with non-trivial gluing data. The correct
treatment of gluing data is given in \cite[\S5.2]{Theta}.}

\begin{assumption}
\label{Ass: trivial gluing data}
One of the following two conditions hold.
\begin{enumerate}
\item
If $\omega,\omega'\in \P$ are parallel edges there exists a maximal cell $\sigma\in\P$ with $\omega\cup\omega'\subseteq \sigma$.
\item
The open gluing data $s$ are trivial.
\end{enumerate}
\end{assumption}

Since by Assumption~\ref{Ass: trivial gluing data} the restrictions of the
$W^0(\sigma)$ to lower dimensional toric strata agree, they define a function
$W^0\in\O(X_0)$. This is our \emph{superpotential at order $0$}. A motivation
for this definition in terms of counts of tropical analogues of holomorphic
disks will be given in Section~\ref{sect:tropical disks}.

One insight in this paper is that in studying LG-models tropically it is
advisable to restrict to asymptotically cylindrical $B$.

\begin{proposition}
\label{properness criterion}
The order zero superpotential $W^0:X_0\to\AA^1$ is proper iff $(B,\P)$ is
asymptotically cylindrical. 
\end{proposition}

\begin{proof}
It suffices to show the claimed equivalence after restriction to a non-compact
irreducible component $X_\sigma\subseteq X_0$, that is, for $W^0(\sigma)$ from
\eqref{Eqn: W^0(sigma)}. If all unbounded edges are parallel, $W^0(\sigma)$ is a
multiple of a monomial with compact zero locus, and hence is proper.

For the converse we show that if $m_{\omega}\neq m_{\omega'}$ for some
$\omega,\omega' \subseteq\sigma$ then $W^0(\sigma)$ is not proper. The idea is
to look at the closure of the zero locus of $W^0(\sigma)$ in an appropriate
toric compactification $X_{\tilde \sigma}\supset X_\sigma$.

Let $\omega_0,\ldots,\omega_r$ be the unbounded edges of $\sigma$ and write
$m_i=m_{\omega_i}$ for their primitive generators. By assumption
$\conv\{0,m_0,\ldots,m_r\}$ has a face not containing $0$ of dimension at least
one. Let $H\subseteq \Lambda_\sigma$ be a supporting affine hyperplane of such a
face. After relabeling we may assume $m_0,\ldots,m_s$ are the vertices of this
face. Note that all $m_i-m_0$ are contained in the affine half-space $H-\RR_{\ge
0} m_0$. Now $\tilde\sigma= \sigma\cap(H-\RR_{\ge 0} m_0)$ is a rational bounded
polytope $\tilde\sigma\subseteq \sigma$ with a single facet $\tau\subset
\tilde\sigma$ not contained in a facet of $\sigma$. Thus the toric variety
$X_{\tilde\sigma}$ with momentum polytope $\tilde\sigma$ contains $X_\sigma$ as
the complement of the toric prime divisor $X_\tau\subset X_{\tilde \sigma}$.
Note that $\Lambda_\tau= H-m_0$ by construction. 

To study the closure of the zero locus of $W^0(\sigma)$ in $X_{\tilde\sigma}$
consider the rational function $z^{-m_0}\cdot W^0(\sigma)$ on
$X_{\tilde\sigma}$. This rational function does not contain $X_\tau$ in its
polar locus, and its restriction to the big cell of $X_\tau$ is
\[
1+\sum_{i=1}^s z^{m_i-m_0}\in \kk[\Lambda_\tau].
\]
In fact, $z^{m_i-m_0}$ for $i>s$ vanishes along $X_\tau$. Since $s\ge 1$ this
Laurent polynomial has a non-empty zero locus. This proves that unless $m_i=m_j$
for all $i,j$ the closure of the zero locus of $W^0(\sigma)$ in
$X_{\tilde\sigma}$ has a non-empty intersection with $X_\tau$, and hence
$W^0(\sigma)$ can not be proper. 
\end{proof}

Thus if one is to study LG models via our degeneration approach, then to obtain
the full picture one has to restrict to asymptotically cylindrical $(B,\P)$.

The interpretation on the mirror side of the condition that $(B,\P)$ be
asymptotically cylindrical brings us to one of the main insights of this paper.
We first formulate the Legendre-dual of Definition~\ref{Def: parallel edges}.

\begin{definition}
A tropical manifold $(\check B,\check \P)$ is said to have \emph{flat boundary}
if $\partial\check B$ is locally flat in the affine structure.
\end{definition}

\begin{theorem}
\label{Thm: Properness}
Let $\X\to \Spf\kk\lfor t\rfor$ and $(\pi:\check\X\to T,\check\D)$ be polarized toric
degenerations with Legendre dual intersection complexes $(B,\P,\varphi)$,
$(\check B,\check\P,\check\varphi)$. Then the following are equivalent.
\begin{enumerate}
\item
The order zero superpotential $W^0:X_0\to\AA^1$ defined in \eqref{Eqn:
W^0(sigma)} is proper.
\item
$(B,\P)$ is asymptotically cylindrical.
\item
$(\check B,\check \P)$ has flat boundary.
\item 
The restriction $\pi|_{\check\D}:\check\D\to T$ of $\pi:\check\X\to T$ is
itself a toric degeneration.
\end{enumerate}
\end{theorem}

\begin{proof}
The equivalence of (1) and (2) is the content of Proposition~\ref{properness criterion}.

The Legendre dual of the condition that $(B,\P)$ is asymptotically cylindrical
says that $\partial \check B\subseteq \check B$ is itself an affine manifold
with singularities to which our program applies. If this is the case then from
the definition of $\check\D\subseteq \check\X$ it follows that $\check\D\to T$
is obtained by restricting the slab functions to $\check D_0\subseteq \check
X_0$ and run our gluing construction \cite[\S2.6,\S2.7]{affinecomplex} on
$\partial\check B$. The result is hence a toric degeneration of Calabi-Yau
varieties.

Conversely, assume that $\rho,\rho'\subseteq\partial \check B$ are two
neighboring $(n-1)$-faces with $\Lambda_\rho\neq\Lambda_{\rho'}$ as subspaces of
$\Lambda_v$ for $v\in\rho\cap\rho'$. In the notation of
\cite[Constr.\,2.7]{affinecomplex}, the toric local model of $\check\X\to T$ is
given by $\kk[P_{\rho\cap\rho',\sigma}]$ for $\sigma\in\P$ a maximal cell
containing $\rho\cap\rho'$, with $\check\D$ locally corresponding to
$\rho\cup\rho'$. It is now easy to see that $\check\D\to T$ is not formally a
smoothing of $\check D_0$ at the generic point of $\check X_{\rho\cap\rho'}$
unless $\partial \check B$ is straight at $\rho\cap\rho'$. Hence $\partial\check
B$ has to be smooth for $\check\D\to T$ to be a toric degeneration.
\end{proof}

Theorem~\ref{properness criterion},(4) motivates the following definition.

\begin{definition}
\label{Def: compact type}
A toric degeneration of log Calabi-Yau varieties $(\pi:\check\X\to T,\check\D)$
is called \emph{of compact type} if $\check\D\to T$ is as well a toric
degeneration of Calabi-Yau varieties.
\end{definition}

As a first example we consider the case of $\PP^2$.

\begin{example}
\label{Expl: PP2}
The standard method to construct the LG-mirror for $\PP^2$ is to start from the
momentum polytope $\Xi= \conv\{(-1,-1),(2,-1),(-1,2)\}$ of $\PP^2$ with its
anticanonical polarization \cite{HoriVafa} . The rays of the corresponding
normal fan associated to this polytope (using inward pointing normal vectors as
in~\cite{affinecomplex}) are generated by $(1,0), (0,1), (-1,-1)$. Calling the
monomials corresponding to the generators of the first two rays $x$ and $y$,
respectively, we obtain the usual (non-proper) Landau-Ginzburg model on the big
torus $(\GG_m(\kk))^2$, the function $x+y+\frac{1}{xy}$.
\begin{figure}
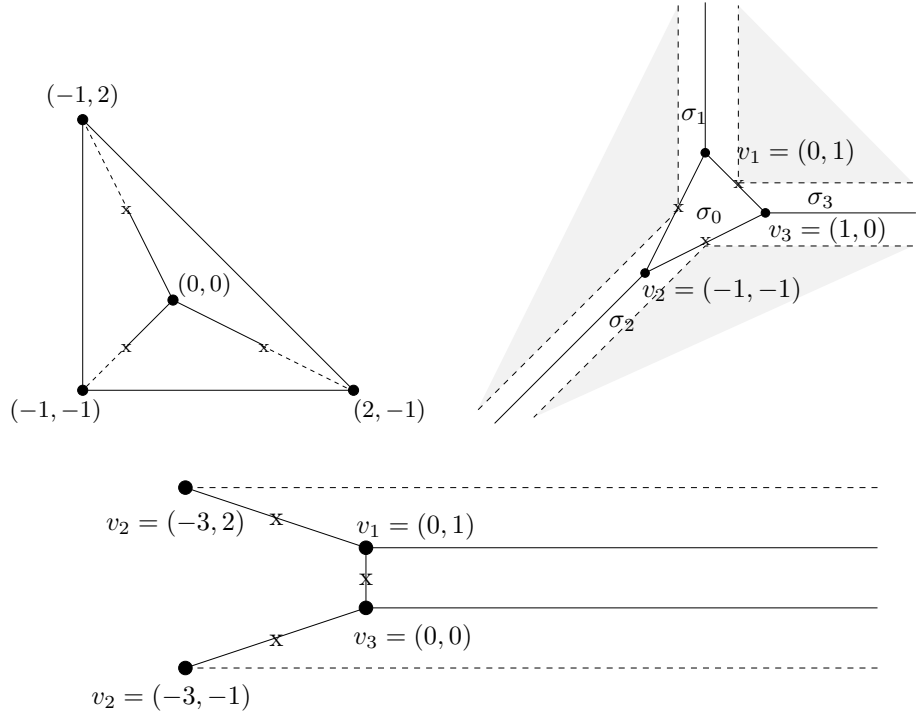

\input{PP2.pspdftex}\hspace{1cm}
\input{mirrorPP2.pspdftex} \vspace{0.5cm}
\input{scattering_mirrorPP2.pspdftex}
\caption{An intersection complex $(\check B,\check \P)$ for $\PP^2$
with straight boundary and its Legendre dual $(B,\varphi)$
for the minimal polarization, with a chart on the complement of the
shaded region and a chart showing the three parallel unbounded
edges.}
\label{fig:PP^2}
\end{figure}

To obtain a proper superpotential we need to make the boundary of the momentum
polytope flat in affine coordinates. To do this we trade the corners with
singular points in the interior. The simplest choice is a decomposition $\check
\P$ of $\check B=\Xi$ into three triangles with three singular points with
simple monodromy, that is, conjugate to $\left(\begin{smallmatrix}1&0\\1&1
\end{smallmatrix}\right)$, as depicted in Figure~\ref{fig:PP^2}. A minimal
choice of the PL-function $\check \varphi$ takes values $0$ at the origin and
$1$ on $\partial \check B$.\footnote{\cite[Expl.\,6.2]{Theta} identifies the
generic fiber of the algebraization $(\check X\to T,\check D)$ of the resulting
toric degeneration with the family of elliptic curves
$g(t)(x_0^3+x_1^3+x_2^3)+x_0 x_1 x_2$ in the trivial deformation of $\PP^2$.
Here $g(t)=t+O(t^2)$ is an analytic change of parameter related to Jacobian
theta functions.} For this choice of $\check \varphi$ the Legendre dual of
$(\check B,\check\P,\check \varphi)$ is shown in Figure~\ref{fig:PP^2} on the
right. Note that the unbounded edges are indeed parallel, so each unbounded edge
comes with copies of the other two unbounded edges parallel at integral
distance~$1$. 

Now let us compute $X_0$ and $W_{\PP^2}^0$. The polyhedral decomposition has one
bounded maximal cell $\sigma_0$ and three unbounded maximal cells
$\sigma_1,\sigma_2,\sigma_3$. The bounded cell is the momentum polytope of the
$X_{\sigma_0}$, a $\ZZ/3$-quotient of $\PP^2$. Each unbounded cell is
affine isomorphic to $[0,1]\times \RR_{\ge0}$, the momentum polytope of
$\PP^1\times\AA^1=: X_{\sigma_i}$, $i=1,2,3$. The $X_{\sigma_i}$ glue together
by torically identifying pairs of $\PP^1$'s and $\AA^1$'s as prescribed by the
polyhedral decomposition to yield $X_0$. By definition, $W_{\PP^2}^0$ vanishes
identically on the compact component $X_{\sigma_0}$. Each of the unbounded
components has two parallel unbounded edges, leading to the pull-back to
$\PP^1\times\AA^1$ of the toric coordinate function of $\AA^1$, say $z_i$ for
the $i$-th copy. Thus $W^0_{\PP^2}|_{X_{\sigma_i}}=z_i$ for $i=1,2,3$ and
$W^0:X_0\to\AA^1$ is proper. These functions are clearly compatible with the
toric gluings.
\end{example}

\begin{remark}
An interesting feature of the degeneration point of view is that the mirror
construction respects the finer data related to the degeneration such as the
monodromy representation of the affine structure. In particular, this poses a
question of uniqueness of the Landau-Ginzburg mirror. For the anticanonical
polarization such as the chosen one in the case of $\PP^2$, the tropical data
$(\check B,\check \P)$ is essentially unique, see Theorem~\ref{Thm: unique} for
a precise statement. For larger polarizations (thus enlarging $\check B$) there
are certainly many more possibilities. For example, as an affine manifold with
singularities one can perturb the location of the singular points transversally
to the invariant directions over the rational numbers and choose an adapted
integral polyhedral decomposition after appropriate rescaling. It is not clear
to us if all $(\check B,\check \P)$ with flat boundary leading to $\PP^2$ can be
obtained by this procedure.
\end{remark}


\section{Scattering of monomials}
\label{sect: scattering of
monomials}
A central tool in \cite{affinecomplex} are scattering diagrams. The purpose of
this section is to study the propagation of monomials through scattering
diagrams. Assume $\scrS_k$ is a structure that is consistent to order $k$ and
let $\foj$ be a \emph{joint} of $\scrS_k$. Recall that a joint for a wall
structure is a codimension two cell of the polyhedral decomposition of $B$ with
codimension one skeleton the union of walls in $\scrS_k$. Thus each joint is the
intersection of the walls that contain the joint. These walls containing $\foj$
define various \emph{scattering diagrams} $\D=(\forr_i, f_\foc)$ that collect
the data carried by $\scrS_k$ relevant around $\foj$. Each joint defines
several, closely related scattering diagrams, one for each choice of vertex
$v\in\sigma_\foj$ and $\omega\in\P$ with $\omega\subseteq \sigma_\foj$ see
\cite[Def.\,3.3, Constr.\,3.4]{affinecomplex}. Here $\sigma_\foj\in\P$ is
the minimal cell containing $\foj$. A scattering diagram is a collection of
half-lines $\RR_{\ge0}\doublebar m$ in the two-dimensional quotient space
$Q_{\foj,\RR}^v=\big(\Lambda_{B,v}/ \Lambda_{\foj,v}\big)_\RR$ along with a
function attached to each half-line. The double-bar notation denotes the image
of an element or subset in $Q_{\foj,\RR}^v$, such as $\doublebar m$,
$\doublebar\sigma$. The functions lie in the ring $\kk[P_x]$ defining the local
model of $\X$ at some $x\in (\foj\cap \Int \omega)\setminus\Delta$. Any
codimension one cell $\rho\in\P$ containing $\foj$ produces one or two
half-lines, the latter in the case $\foj\cap\Int\rho\neq\emptyset$. Half-lines
obtained from codimension one-cells are called \emph{cuts}, all other
half-lines \emph{rays}.

For an exponent $m_0$ with $\overline m_0\in\Lambda_v\setminus \Lambda_\foj$ we
wish to define the scattering of the monomial $z^{m_0}$, which we think of
traveling along the ray $-\RR_{\ge0} \doublebar m_0$ into the origin of
$\shQ_{\foj,\RR}^v\simeq\RR^2$. In a scattering diagram monomials travel along
\emph{trajectories}. These are defined in exactly the same way as rays
\cite[Def.\,3.3]{affinecomplex}, but will have an additive meaning.

\begin{definition}
A \emph{trajectory} in $\shQ_{\foj,\RR}^v$ is a triple $(\fot,m_\fot, a_\fot)$,
where $m_\fot$ is a monomial on a maximal cell $\sigma\ni v$ with $\pm
\doublebar m_\fot\in \doublebar \sigma$ and $m\in P_x$ for all $x\in
\foj\setminus\Delta$, $\fot=\pm \RR_{\ge 0} \doublebar m$, and $a_\fot\in \kk$.
The trajectory is called \emph{incoming} if $\fot=\RR_{\ge 0} \doublebar m$, and
\emph{outgoing} if $\fot=-\RR_{\ge 0} \doublebar m$. By abuse of notation we
often suppress $m_\fot$ and $a_\fot$ when referring to trajectories.
\end{definition}

Here is the generalization of the central existence and uniqueness result for
scattering diagrams \cite[Prop.\,3.9]{affinecomplex} incorporating trajectories.
This result is crucial for the existence proof of the superpotential in
Lemmas~\ref{lem:independence of p} and~\ref{lem:change of chambers}.

\begin{proposition}
\label{prop: trajectory scattering}
Let $\D$ be the scattering diagram defined by $\scrS_k$ for
$\foj\in\Joints(\scrS_k)$, $g:\omega\to\sigma_\foj$, $v\in\omega$. Let
$(\RR_{\ge0}\doublebar m_0,m_0,1)$ be an incoming trajectory and $\sigma\supset
\foj$ a maximal cell with $\doublebar m_0\in \doublebar \sigma$. For $\doublebar
m\in \shQ_{\foj,\RR}^v\setminus \{0\}$ denote by
\[
\theta_{\scriptsize\doublebar m}: R^k_{g,\sigma'} \to R^k_{g,\sigma}
\]
the ring isomorphism defined by $\D$ for a path connecting $-\doublebar m$ to
$\doublebar m_0$, where $\sigma'$ is a maximal cell with $-\doublebar
m\in\doublebar{\sigma'}$.

Then there is a set $\foT$ of outgoing trajectories such that
\begin{equation}
\label{eqn:monomial scattering}
z^{m_0}= \sum_{\fot\in\foT} \theta_{\scriptsize\doublebar m_\fot}
(a_\fot z^{m_\fot})
\end{equation}
holds in $R^k_{g,\sigma}$. Moreover, $\foT$ is unique if $a_{\fot}z^{m_\fot}\neq
0$ in $R^k_{g,\sigma'}$ for all $\fot\in\foT$, and if $m_\fot\neq m_{\fot'}$
whenever $\fot\neq\fot'$.
\end{proposition}

\begin{proof}
The proof is by induction on $l\le k$. We first discuss the case that
$\sigma_\foj$ is a maximal cell, that is, $\codim \sigma_\foj=0$. Then $\D$
has only rays, no cuts. In particular, any $\theta_{\scriptsize\doublebar m}$ is
an automorphism of $R^k_{g,\sigma}$ that is the identity modulo
$I_{g,\sigma}^{>0}$. Thus for $l=0$, \eqref{eqn:monomial scattering} forces one
outgoing trajectory $(-\RR_{\ge 0}\doublebar m_0,m_0,1)$ if
$\ord_{\sigma_\foj}(m_0)=0$ or none otherwise. For the induction step assume
\eqref{eqn:monomial scattering} holds in $R^{l-1}_{g,\sigma}$. Then in
$R^l_{g,\sigma}$ the difference of the two sides of \eqref{eqn:monomial
scattering} is a sum of monomials $a z^m$ with $\ord_{\sigma_\foj} (m)=l$. Since
$l>0$ and since there are no cuts (these represent slabs containing $\foj$), it
holds $\theta_{\scriptsize\doublebar m}(a z^m)=a z^m$. Thus after adding
appropriate trajectories $(-\RR_{\ge0}\doublebar m,m,a)$ with
$\ord_{\sigma_\foj}(m)=l$ to $\foT$, Equation~\eqref{eqn:monomial scattering}
holds in $R^l_{g,\sigma}$. This is the unique minimal choice of $\foT$. This finishes the proof in the case $\codim\sigma_\foj=0$.

Under the presence of cuts we have several rings $R^k_{g,\sigma'}$ for various
maximal cells $\sigma'\supset\foj$. This possibly brings in denominators that
are powers of $f_{\rho,v}$ for cells $\rho\supset \foj$ of codimension one. In
this case we show existence by a perturbation argument. To this end consider
first the simplest scattering diagram in codimension one consisting of only two
cuts $\foc_\pm=(\pm \foc,f_{\rho,v})$ dividing $\shQ$ into two halfplanes
$\doublebar\sigma_\pm$ and with the same attached function. The signs are chosen
in such a way that $\doublebar m_0\in \doublebar \sigma_-$. Let $\theta:
R^k_{g,\sigma_-}\to R^k_{g,\sigma_+}$ be the isomorphism defined by a path from
$\doublebar \sigma_-$ to $\doublebar \sigma_+$ and let
$n\in\Lambda_\rho^\perp\subseteq \Lambda_v^*$ be the primitive integral vector
that is positive on $\sigma_-$. Then $\langle m_0,n \rangle \ge 0$ and
\[
\theta(z^{m_0})= f_{\rho,v}^{\langle  m_0,n\rangle} \cdot z^{m_0}.
\]
Expanding yields the finite sum
\begin{equation}
\label{Eqn: theta(z^m)}
\theta(z^{m_0})= \sum_{\langle m,n\rangle\ge 0} a_m z^m
= \theta\Big(\sum_{\langle m,n\rangle\ge 0} a_m
\theta^{-1}(z^m)\Big)
=\theta\Big(\sum_{\langle m,n\rangle\ge 0}
\theta_{\scriptsize\doublebar m}(a_m z^m)\Big),
\end{equation}
for some $a_m\in\kk$. Note that $\theta^{-1}(z^m)=\theta_{\scriptsize\doublebar
m}(z^m)$ by the definition of $\theta_{\scriptsize\doublebar m}$. Now
\eqref{Eqn: theta(z^m)} equals $\theta$ applied to \eqref{eqn:monomial
scattering} for the set of trajectories
\[
\foT:=\big\{(-\RR_{\ge0}\doublebar m,m,a_m)\,\big|\,\langle
m,n\rangle\ge 0, a_mz^m\neq 0 \big\}.
\]
Hence existence is clear in this case.

In the general case we work with
perturbed trajectories as suggested by Figure~\ref{fig:perturbed scattering}.
\begin{figure}
\input{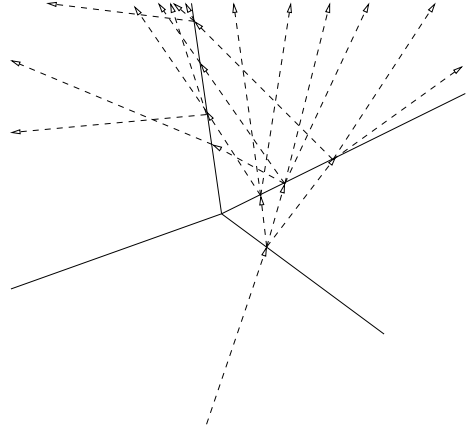}
\caption{Scattering diagram with perturbed trajectories (cuts
and rays solid, perturbed trajectories dashed).}
\label{fig:perturbed scattering}
\end{figure}
More precisely, a perturbed trajectory is a trajectory with the origin shifted.
There is one unbounded perturbed incoming trajectory, a translation of
$(-\RR_{\ge 0}\doublebar m_0,m_0,1)$, and a number of perturbed outgoing
trajectories, each the result of scattering of other trajectories with rays or
cuts. At each intersection point of a trajectory with a ray or cut, the incoming
and outgoing trajectories at this point fulfill an equation analogous to
\eqref{eqn:monomial scattering}. Similar to \cite[Constr.\,4.5]{affinecomplex}
with our additive trajectories replacing the multiplicative
$s$-rays\footnote{For technical reasons $s$-rays were not asked to be piecewise
affine. In the present situation we insist in piecewise affine objects.}, there
is then an asymptotic scattering diagram with trajectories obtained by taking
the limit $\lambda\to 0$ of rescaling the whole diagram by $\lambda\in\RR_{>0}$.
Any choice of perturbed incoming trajectory determines a unique minimal
scattering diagram with perturbed trajectories. Moreover, for a generic choice
of perturbed incoming trajectory the intersection points of trajectories with
rays or cuts are pairwise disjoint, and they are in particular different from
the origin. Hence the perturbed diagram can be constructed uniquely by induction
on $l\le k$. Taking the associated asymptotic scattering diagram with
trajectories now establishes the existence in the general case.

Next we show uniqueness for $\codim\sigma_\foj>0$. For $\codim\sigma_\foj=2$ any
monomial $z^m$ in $\kk[\Lambda_{\sigma_\foj}]$ fulfills
$\ord_{\sigma_\foj}(m)=0$, and all $\theta_{\scriptsize\doublebar m}$ extend to
$\kk(\Lambda_{ \sigma_\foj})$-algebra automorphisms of
$R^k_{g,\sigma}\otimes_{\kk[\Lambda_{\sigma_\foj}]} \kk(\Lambda_{\sigma_\foj})$.
Hence we can deduce uniqueness by induction on $l\le k$ as in codimension~$0$ by
taking the factors $a$ of trajectories to be polynomials with coefficients in
$\kk(\Lambda_{\sigma_\foj})$. Thus we combine all trajectories $\fot$ with the
same $\doublebar m_\fot$ and the same $\ord_{\sigma_\foj} (m_\fot)$. It is clear
that such generalized trajectories can be split uniquely into proper
trajectories with all $m$ distinct, showing uniqueness in this case.

Finally, for uniqueness in codimension one we can not argue just with
$\ord_{\sigma_\foj}$ because there are monomials $z^m$ with
$\ord_{\sigma_\foj}(m)=0$ but $\doublebar m\neq0$. Instead we look closer at the
effect of adding trajectories. By induction it suffices to study the insertion
of trajectories $(-\RR_{\ge 0}\doublebar m, m,a)$ with $\ord_{\sigma_\foj}
(m)=l$ for each $m$ and such that \eqref{eqn:monomial scattering} continues to
hold. By working with a perturbed scattering diagram as in Figure~4.1 of
\cite{affinecomplex} and the asociated asymptotic scattering diagram as in
\cite[\S4.3]{affinecomplex}, it suffices to consider the case of only two cuts
as already considered above. In this case we have
\[
0= \sum_m \theta_{\scriptsize\doublebar m} (a_m z^m)
\;=\;\sum_{i=0}^l \sum_{\langle m,n\rangle=i}
\theta_{\scriptsize\doublebar m} (a_m z^m)\\
=\sum_{i=0}^l  f_{\rho,v}^{-i}\sum_{\langle m,n\rangle=i} a_m z^m.
\]
Since all monomials in $f_{\rho,v}$ have vanishing $\ord_{\sigma_{\foj}}$ and
only monomials $z^m$ with the same value of $\langle m,n\rangle$ can cancel,
this equation implies
\[
f_{\rho,v}^{-i}\sum_{\langle m,n\rangle=i} a_m z^m=0
\]
holds for each $i$. Multiplying by $f_{\rho,v}^i$ thus shows $\sum_{\langle
m,n\rangle=i} a_m z^m=0$ in $R_{g,\sigma}^l$, and hence $a_m=0$ for all $m$.
This proves uniqueness also in codimension one.
\end{proof}


\section{The superpotential via broken lines}
\label{Ch: Broken lines}
The easiest way to define the superpotential in full generality is by the method
of broken lines. Broken lines have been introduced by Mark Gross for $\dim B=2$
in his work on mirror symmetry for $\PP^2$ \cite{PP2mirror}. We assume we are
given a locally finite scattering diagram $\scrS_k$ for the non-compact
intersection complex $(B,\P,\varphi)$ of a polarized LG-model that is consistent
to order $k$, as given by Assumption~\ref{overall assumption}. The notion of
broken lines is based on the transport of monomials by changing chambers of
$\scrS_k$. Recall from \cite[Def.\,2.22]{affinecomplex} that a chamber is the
closure of a connected component of $B\setminus|\scrS_k|$.

\begin{definition}
\label{def: monomial transport}
Let $\fou,\fou'$ be neighboring chambers of $\scrS_k$, that is,
$\dim(\fou\cap\fou') =n-1$. Let $a z^m$ be a monomial defined at all points of
$\fou\cap\fou'$ and assume that $\overline m$ points from $\fou'$ to $\fou$. Let
$\tau:= \sigma_\fou\cap \sigma_{\fou'}$ and
\[
\theta: R^k_{\id_\tau,\sigma_\fou}\to R^k_{\id_\tau,\sigma_{\fou'}}
\]
be the gluing isomorphism changing chambers. Then if
\begin{equation}
\label{eqn: transport}
\theta(az^m)=\sum_i a_iz^{m_i}
\end{equation}
we call any summand $a_i z^{m_i}$ with $\ord_{\sigma_{\fou'}}(m_i)\le k$ a
\emph{result of transport of $az^m$} from $\fou$ to $\fou'$.
\end{definition}

Note that since the change of chamber isomorphisms commute with changing strata,
the monomials $a_iz^{m_i}$ in Definition~\ref{def: monomial transport} are
defined at all points of $\fou\cap\fou'$. 

\begin{definition}
\label{def: broken lines}(Cf.\ \cite[Def.\,4.9]{PP2mirror})
A \emph{broken line} for $\scrS_k$ is a proper continuous maps
\[
\beta: (-\infty,0]\to B
\]
with image disjoint from any joint of $\scrS_k$, along with a sequence
$-\infty=t_0<t_1<\ldots< t_{r-1}\le t_r=0$ for some $r\ge 0$ with $\beta(t_i)\in
|\scrS_k|$, and for $i=1,\ldots,r$ monomials $a_i z^{m_i}$ defined at all points
of $\beta([t_{i-1},t_i])$ (for $i=1$, $\beta((-\infty,t_1])$), subject to the
following conditions.
\begin{enumerate}
\item
$\beta|_{(t_{i-1},t_i)}$ is a non-constant affine map with image disjoint from
$|\scrS_k|$, hence contained in the interior of a unique chamber $\fou_i$ of
$\scrS_k$, and $\beta'(t)=-\overline m_i$ for all $t\in (t_{i-1},t_i)$.
Moreover, if $t_r=t_{r-1}$ then $\fou_r\neq \fou_{r-1}$.
\item
$a_1=1$ and\footnote{The normalization condition $a_1=1$ has to be modified if
there are parallel unbounded edges not contained in one cell and the gluing data
are not trivial, see \cite[\S5.2]{Theta}.} there exists a (necessarily
unbounded) $\omega\in\P^{[1]}$ with $\overline m_1\in\Lambda_\omega$ primitive
and $\ord_\omega(m_1)=0$.
\item
For each $i=1,\ldots,r-1$ the monomial $a_{i+1} z^{m_{i+1}}$ is a result of
transport of $a_i z^{m_i}$ from $\fou_i$ to $\fou_{i+1}$ (Definition~\ref{def:
monomial transport}).
\end{enumerate}
The \emph{type} of $\beta$ is the tuple of all $\fou_i$ and $m_i$. By abuse of
notation we suppress the data $t_i,a_i, m_i$ when talking about broken lines,
but introduce the notation
\[
a_\beta:= a_r,\quad m_\beta:=m_r.
\]
For $p\in B$ the set of broken lines $\beta$ with $\beta(0)=p$ is denoted
$\foB(p)$.
\end{definition}

\begin{remark}
\label{rem: broken lines}
1)\ \ If $(B,\P)$ is asymptotically cylindrical (Definition~\ref{Def: parallel
edges}) then in Definition~\ref{def: broken lines} the existence of a one-cell
$\omega\in\P$ with $\overline m_1\in\Lambda_\omega$ in (2) follows from
(1).\\[1ex]
2)\ \ A broken line $\beta$ is determined uniquely by specifying its endpoint
$\beta(0)$ and its type. In fact, the coefficients $a_i$ are determine
inductively from $a_1=1$ by Equation~\eqref{eqn: transport}. 
\end{remark}

According to Remark~\ref{rem: broken lines},(2) the map $\beta\mapsto \beta(0)$
identifies the space of broken lines of a fixed type with a subset of $\fou_r$.
This subset is the interior of a polyhedron:

\begin{proposition}
\label{prop: broken line moduli polyhedral}
For each type $(\fou_i,m_i)$ of broken lines there is an integral, closed,
convex polyhedron $\Xi$, of dimension $n$ if non-empty, and an affine immersion
\[
\Phi: \Xi\lra \fou_r,
\]
so that $\Phi\big (\Int\Xi\big)$ is the set of endpoints $\beta(0)$ of broken
lines $\beta$ of the given type.
\end{proposition}

\begin{proof}
This is an exercise in polyhedral geometry left to the reader. For the statement
on dimensions it is important that broken lines are disjoint from joints.
\end{proof}

\begin{remark}
\label{degenerate broken lines}
A point $p\in \Phi(\partial \Xi)$ still has a meaning as an endpoint of a
piecewise affine map $\beta:(-\infty,0]\to B$ together with data $t_i$ and $a_i
z^{m_i}$, defining a \emph{degenerate broken line}. For $\beta$ not to be a
broken line, $\im(\beta)$ has to intersect a joint.

Indeed, if $\beta$ violates Definition~\ref{def: broken lines},(1) then
$t_{i-1}=t_i$ or there exists $t\in (t_{i-1},t_i)$ with $\beta(t)\in|\scrS_k|$.
In the first case denote by $\fop_j\in\scrS$ the wall with
$\beta'(t_j)\in\fop_j$ for all $\beta'\in \Phi(\Int\Xi)$. Then
$\beta(t_{i-1})=\beta(t_i)$ lies in $\fop_{i-1}\cap\fop_i$, hence in a joint.
In the second case convexity of the chambers implies that
$\beta([t_{i-1},t])\subset\partial \fou$ or $\beta([t,t_i])\subset\partial
\fou$. Thus $\beta$ maps a whole open interval to $|\scrS_k|$. The break
point $t_{i-1}$ or $t_i$ contained in this interval again intersects two
walls, hence is contained in a joint. The other conditions in the definition of
broken lines are closed.

The set of endpoints $\beta(0)$ of degenerate broken lines of a given type is
the $(n-1)$-dimensional polyhedral subset $\Phi(\partial\Xi)\subseteq \fou_r$.
The set of degenerate broken lines \emph{not transverse} to some joint of
$\scrS_k$ is polyhedral of smaller dimension.
\end{remark}

Any finite structure $\scrS_k$ involves only finitely many slabs and walls, and
each polynomial coming with each slab or wall carries only finitely many
monomials. Hence broken lines for $|\scrS_k|$ exist only for finitely many
types. The following definition is therefore meaningful.

\begin{definition}
\label{def: general points}
A point $p\in B$ is called \emph{general} (for the given structure $\scrS_k$) if
it is not contained in $\Phi(\partial \Xi)$, for any $\Phi$ as in
Proposition~\ref{prop: broken line moduli polyhedral}.
\end{definition}

Recall from \cite[\S2.6]{affinecomplex} that $\scrS_k$ defines a $k$-th order
deformation of $X_0$ by gluing the sheaf of rings defined by
$R^k_{g,\sigma_\fou}$, with $g:\omega\to\tau$ and $\fou$ a chamber of $\scrS_k$
with $\omega\cap\fou\neq\emptyset$, $\tau\subseteq\sigma_\fou$. Given a general
$p\in \fou$ we can now define the \emph{superpotential} up to order $k$ locally
as an element of $R^k_{g,\sigma_\fou}$ by
\begin{equation}
\label{Eqn: W^k}
W^k_{g,\fou}(p):=\sum_{\beta\in \foB(p)} a_\beta z^{m_\beta}.
\end{equation}
The existence of a canonical extension $W^k$ of $W^0$ to $X_k$ follows once we
check that (i)~$W^k_{g,\fou}(p)$ is independent of the choice of a general
$p\in\fou$ and (ii)~the $W^k_{g,\fou}(p)$ are compatible with changing strata or
chambers \cite[Constr.\,2.24]{affinecomplex}. This is the content of the
following two lemmas.

\begin{lemma}
\label{lem:independence of p}
Let $\fou$ be a chamber of $\scrS_k$ and $g:\omega\to\tau$ with $\omega\cap\fou
\neq\emptyset$, $\tau\subseteq\sigma_\fou$. Then $W^k_{g,\fou}(p)$ is
independent of the choice of $p\in\fou$.
\end{lemma}

\begin{proof}
By Proposition~\ref{prop: broken line moduli polyhedral} the set $A \subseteq
\fou$ of non-general points is a finite union of nowhere dense polyhedra.
Moreover, since all $\Phi$ in Proposition~\ref{prop: broken line moduli
polyhedral} are local affine isomorphisms, for each path $\gamma: [0,1]\to
\fou\setminus A$ and broken line $\beta_0$ with $\beta_0(0)= \gamma(0)$ there
exists a unique family $\beta_s$ of broken lines with endpoints
$\beta_s(0)=\gamma(s)$ and with the same type as $\beta_0$. Hence
$W^k_{g,\fou}(p)$ is locally constant on $\fou\setminus A$.

To pass between the different connected components of $\fou\setminus A$,
consider the set $A'\subseteq A$ of endpoints of degenerate broken lines that
are not transverse to the joints of $\scrS_k$. More precisely, for each type of
broken line, the set of endpoints of broken lines intersecting a given joint
defines a polyhedral subset of $\fou$ of dimension at most $n-1$. Then $A'$ is
the union of the $(n-2)$-cells of these polyhedral subsets, for any joint and
any type of broken line. Since $\dim A'= n-2$, we conclude that $\fou\setminus
A'$ is path-connected. It thus suffices to study the following situation. Let
$\gamma:[-1,1]\to \Int \fou\setminus A'$ be an affine map with $\gamma(0)$ the
only point of intersection with $A$. Let $\overline\beta_0: (-\infty,0]\to B$ be
the underlying map of a degenerate broken line with endpoint $\gamma(0)$. The
point is that $\overline\beta_0$ may arise as a limit of several different types
of broken lines with endpoints $\gamma(s)$ for $s\neq0$. The lemma follows once
we show that the contributions to $W^k_{g,\fou} \big(\gamma(s)\big)$ of such
broken lines for $s<0$ and for $s>0$ coincide. Note we do not claim a bijection
between the sets of broken lines for $s<0$ and for $s>0$, which in fact need not
be true.

For later use let $V\subset \Int\fou$ be a local affine hyperplane intersecting
$\ol\beta_0$ only in $\ol\beta_0(0)=\gamma(0)$ and containing $\im(\gamma)$.
Note that $V$ is also transverse to $A$. Thus by the unique continuation
statement at the beginning of the proof, each broken line $\beta$ with endpoint
$\beta(0)\in V$ fits into a unique family of broken lines of the same type and
with endpoint any other point of the connected component of $\beta(0)$ in
$V\setminus A$.

In particular, since $\gamma^{-1}(A)=\{0\}$ any broken line $\beta$ with
endpoint $\gamma(s_0)$ for $s_0\neq 0$ extends uniquely to a family of broken
lines $\beta_s$ for $s\in [-1,0)$ or $s\in(0,1]$. Thus $\beta$ has a unique
limit $\lim\beta:=\lim_{s\to 0} \beta_s$, a possibly degenerate broken line. For
$s\neq 0$ denote by $\foB_s$ the space of broken lines $\beta$ with endpoint
$\gamma(s)$ and such that the map underlying $\lim\beta$ equals
$\overline\beta_0$. Since $\overline{\beta_0}$ is the underlying map of a
degenerate broken line, $\foB_s\neq \emptyset$ for some sufficiently small $s$,
hence also for all $s$ of the same sign, by unique continuation. Possibly by
changing signs in the domain of $\gamma$ we may thus assume $\foB_s\neq
\emptyset$ for $s<0$. We have to show
\begin{equation}
\label{eqn: equality of contributions}
\sum_{\beta\in\foB_{-1}} a_\beta z^{m_{\beta}}
= \sum_{\beta\in\foB_{1}} a_\beta z^{m_{\beta}}.
\end{equation}
Denote by $\foT_s$ the set of types of broken lines in $\foB_s$. Obviously,
$\foT_s$ only depends on the sign of $s$.

The central observation is the following. Let $J\subset B$ be the union of the
joints of $\scrS_k$ intersected by $\im \overline\beta_0$. Let $x:=
\overline\beta_0(t)$ for $t\ll 0$ be a point far off to $-\infty$. Thus $x$ lies
in one of the unbounded cells of $\P$ and $\overline\beta_0$ is asymptotically
parallel to an unbounded edge and does not cross a wall for $t'<t$. Let
$U\subseteq B$ be a local affine hyperplane intersecting $\overline{\beta_0}$
transversally at $x$. By transversality with $J$ the images of the degenerate
broken lines of types contained in any $\foB_s$ lie in a local affine hyperplane
$H\subset U$ ($\dim H=n-2$). Note that each joint that $\ol\beta_0$ meets
defines such a local hyperplane $H\subset U$, and if $\ol\beta_0$ meets several
joints, the hyperplanes agree locally near $x$ since $\gamma(0)\not\in A'$.
Moreover, the images of degenerate broken lines arising as the limit of broken
lines of type in $\foT_{-1}\cup\foT_1$ separate a contractible neighbourhood of
$\im\ol\beta_0$ in $B$ into two connected components. It follows that broken
lines of type in $\foT_s$ for $s<0$ intersect $U$ only on one side of $H$, and
for $s>0$ only on the other.

Now let $\breve\beta_s\in\foB_s$ be one family of broken lines, say for $s<0$.
Denote by $\fot$ the type of $\breve\beta_s$. Thus $\breve\beta_s$ is the unique
broken line $\beta$ of type $\fot$ with endpoint $\beta(0)=\gamma(s)$. Since
$\lim\breve\beta_s= \ol\beta_0$, the broken line $\breve\beta_s$ does not pass a
wall before hitting $U$, thus is a straight line. Denote by $x_s\in U\setminus
H$ the point of intersection of $\breve\beta_s$ with $U$. The $x_s$ vary affine
linearly with $s$ with $\lim_{s\to0} x_s= x$, hence define an affine line segment
in $U$. This line segment can be viewed as a lift of $\gamma([-1,0])\subset\fou$
to a line segment in $U$ via broken lines of type $\fot$ and their limits.

If $\beta_s$ is another family of broken lines in $\foB_s$ for $s<0$, of type
$\fot'$, then by the same argument $\beta_s$ hits $U$ in another point $x'_s$ in
the same connected component of $U\setminus H$. Moreover, there is a unique such
$\beta'_s$ with endpoint $\beta'_s(0)\in V$, where $V\subset\Int \fou$ is the
local affine hyperplane chosen above. In particular, for each $s$ there is a
unique broken line $\beta'_s$ of type $\fot'$ passing through $x_s$. In other
words, each broken line $\beta_s\in\foB_s$, $s<0$, deforms uniquely to a broken
line $\beta'_s$ with endpoint on $V$ and passing through $x_s$.

The set of $\beta'_s$ obtained in this way can alternatively be constructed as
follows. For $s<0$ all broken lines in $\foB_s$ have the same first chamber
$\fou_1$ and monomial $z^{m_1}$. Now start with the straight broken line ending
at $x_s$ and of type $(\fou_1,m_1)$. This broken line can be continued until it
hits a wall or slab, where it splits into several broken lines, one for each
summand in \eqref{eqn: transport}. Iterating this process leads to the infinite
set of all broken lines with asymptotic given by $m_1$ and running through
$x_s$. The $\beta'_s$ are the subset of the considered types, that is, with the
unique deformation for $s\to 0$ having underlying map $\overline\beta_0$ and
endpoint on $V$.

From this point of view it is clear that at each joint $\foj$ intersected by
$\im\overline \beta_0$ the $\beta'_s$ compute a scattering of monomials as
considered in \S\ref{sect: scattering of monomials}. In fact, the union of the
$\beta'_s$ with the same incoming part $az^m$ near $\foj$ induce a scattering
diagram with perturbed trajectories as considered in the proof of
Proposition~\ref{prop: trajectory scattering}. Thus the corresponding sum of
monomials leaving a neighborhood of $\foj$ can be read off from the right-hand
side of \eqref{eqn:monomial scattering} in this proposition, applied to the
incoming trajectory $(\RR_{\ge0}\doublebar m,m,a)$.

We conclude that $\sum_{\beta\in\foB_s} a_\beta z^{m_\beta}$ for $s<0$ computes
the transport of $z^{m_1}$ along $\overline \beta_0$. This transport is defined
by applying \eqref{eqn:monomial scattering} instead of \eqref{eqn: transport} at
each joint intersected transversally by $\overline\beta_0$. The same argument holds for $s>0$,
thus proving \eqref{eqn: equality of contributions}.
\end{proof}

\begin{remark}
\label{rem:generalized broken lines}
The proof of Lemma~\ref{lem:independence of p} really shows that the scattering
of monomials introduced in \S\ref{sect: scattering of monomials} allows to
replace the condition that broken lines have image disjoint from joints by
transversality with joints. In the following we refer to these as
\emph{generalized broken lines}. The set of degenerate broken lines with
endpoint $p$ is denoted $\ol\foB(p)$.
\end{remark}

\noindent
By Lemma~\ref{lem:independence of p} and Remark~\ref{rem:generalized broken
lines} we are now entitled to define
\begin{equation}
\label{Eqn: W from degenerate broken lines}
W^k_{g,\fou}:= W^k_{g,\fou}(p)= \sum_{\beta\in \ol\foB(p)} a_\beta z^{m_\beta},
\end{equation}
for any choice of $p\in \fou\setminus A'$, $A'$ the set of endpoints in $\fou$
of degenerate broken lines \emph{not} transverse to all joints of $\scrS_k$.

\begin{lemma}
\label{lem:change of chambers}
The $W^k_{g,\fou}$ are compatible with changing strata and changing chambers.
\end{lemma}

\begin{proof}
Compatibility with changing strata follows trivially from the definitions. As
for changing from a chamber $\fou$ to a neighboring chamber $\fou'$ ($\dim
\fou\cap \fou'=n-1$) the argument is similar to the proof
Lemma~\ref{lem:independence of p}. Let $g:\omega\to\tau$ be such that
$\omega\cap\fou \cap\fou' \neq\emptyset$, $\tau\subseteq
\sigma_\fou\cap\sigma_{\fou'}$ and
\[
	\theta: R^k_{g,\fou}\lra R^k_{g,\fou'}
\]
be the corresponding change of chamber isomorphism \cite[\S2.4]{affinecomplex}.
We have to show $\theta \big(W^k_{g,\fou}\big) = W^k_{g,\fou'}$.

Let $A\subseteq \fou\cup\fou'$ be the set of endpoints of degenerate broken
lines. Consider a path $\gamma:[-1,1]\to \fou\cap\fou'$ connecting general
points $\gamma(-1)\in\fou$, $\gamma(1)\in \fou'$ and with
$\gamma^{-1}(\fou\cap\fou')=\{0\}$. We may also assume that $\gamma(s)\in A$ at
most for $s=0$, and that any degenerate broken line with endpoint $\gamma(0)$ is
transverse to joints. For $s\neq 0$ we then consider the space $\foB_s$ of
broken lines $\beta_s$ with endpoint $\gamma(s)$ and with deformation for $s\to
0$ a fixed underlying map of a degenerate broken line $\overline\beta_0$. By
transversality of $\overline\beta_0$ with the set of joints the limits of
families $\beta_s$, $s\to 0$, group into generalized broken lines
(Remark~\ref{rem:generalized broken lines}). Each such generalized broken line
$\beta$ has as endpoint $p_0:= \gamma(0)$, but viewed as an element either of
$\fou$ or of $\fou'$. We call this chamber the \emph{reference chamber} of
$\beta$. Generalized broken lines with reference chambers $\fou$ and $\fou'$
contribute to $W^k_{g,\fou}$ and $W^k_{g,\fou'}$, respectively. Moreover,
$m_\beta$ is either tangent to $\fou\cap\fou'$ or points properly into $\fou$ or
into $\fou'$. We claim that $\theta$ maps each of the three types of
contributions to $W^k_{g,\fou}$ to the three types of contributions to
$W^k_{g,\fou'}$. Then $\theta \big(W^k_{g,\fou}\big) = W^k_{g,\fou'}$ and the
proof is finished.

Let us first consider the set of degenerate broken lines $\beta$ with
$m_\beta$ tangent to $\fou\cap\fou'$. Then changing the reference
chamber from $\fou$ to $\fou'$ defines a bijection between the
considered generalized broken lines with endpoint $p_0$ and reference
cell $\fou$ and those with reference cell $\fou'$. Note that in
this case $\beta$ has to intersect a joint, so this statement
already involves the arguments from the proof of
Lemma~\ref{lem:independence of p}. Because $\theta(a_\beta
z^{m_\beta})= a_\beta z^{m_\beta}$ this proves the claim in this case. 

Next assume $m_\beta$ points from $\fou\cap\fou'$ into the interior
of $\fou$. This means that $\beta$ approaches $p_0$ from the
interior of $\fou$. If we want to change the reference chamber to
$\fou'$ we need to introduce one more point $t_{r+1}:=t_r=0$ and
chamber $\fou_{r+1}:=\fou'$. According to Equation~\eqref{eqn:
transport} in Definition~\ref{def: monomial transport} the possible
monomials $a_{r+1}z^{m_{r+1}}$ are given by the summands in
$\theta(a z^m)= \sum_i a_{r+1,i} z^{m_{r+1,i}}$. Thus for each
summand we obtain one generalized broken line with reference cell
$\fou'$. Clearly, this is exactly what is needed for compatibility
with $\theta$ of the respective contributions to the local
superpotentials.

By symmetry the same argument works for generalized broken lines
$\beta$ with reference cell $\fou'$ and $m_\beta$ pointing from
$\fou\cap\fou'$ into the interior of $\fou'$, and $\theta^{-1}$
replacing $\theta$. Inverting $\theta$ means that a number of
generalized broken lines with reference cell $\fou_{r+1}=\fou$ and
two points $t_{r+1}=t_r$ (and necessarily $\fou_r=\fou'$), one for
each summand of $\theta^{-1}(a_\beta z^{m_\beta})$, combine into a
single generalized broken line with reference cell $\fou_r=\fou'$.
This process is again compatible with applying $\theta$ to the
respective contributions to the local superpotentials. This finishes
the proof of the claim, which was left to complete the proof of the
lemma.
\end{proof}

Summarizing the construction and Lemmas~\ref{lem:independence of p}
and~\ref{lem:change of chambers}, we now state the unique
existence of the superpotential $W$ on canonical toric degenerations:

\begin{theorem}
\label{Thm: all order W}
Let $\pi:\X\to \Spf\kk\lfor t\rfor$ be the canonical toric degeneration given
by the compatible system of wall structures in Assumption~\ref{overall
assumption}. Then there exists a unique formal function $W:\X\to \AA^1$ that
modulo $t^{k+1}$ agrees with the expressions \eqref{Eqn: W^k} and~\eqref{Eqn: W
from degenerate broken lines} at each point $p$.
\qed
\end{theorem}

Having defined the superpotential $W$ as a regular function on the formal scheme
$\X$, a natural question concerns algebraizablity of $W$ assuming
$\X\to\Spf\kk\lfor t\rfor$ is algebraizable. This generally appears to be a
difficult question, but sometimes more can be said by methods going beyond the
scope of this paper, as detailed in the following remark.

\begin{remark}
\label{Rem: algebraizability}
Given a toric degeneration, we have defined the superpotential $W^k$ locally at
$p\in B$ in the interior of a chamber $\fou$ as an expression in the Laurent
polynomial ring $A_k[\Lambda_\sigma]$, $A_k=\kk[t]/(t^{k+1})$, with
$\sigma=\sigma_\fou$ the maximal cell containing $\fou$. Increasing $k$ may make
$\fou$ smaller, but if we choose $p$ sufficiently transcendental we can achieve
that $p$ never hits a wall of any order. Then the all order potential $W:=
\varprojlim W^k$ still has an expression in the ring
\[
\kk[\Lambda_\sigma]\lfor t\rfor = \varprojlim_k A_k[\Lambda_\sigma].
\]
To understand the dependence on $p$ note that by the construction of $\X$ as a
colimit, a point $p\in B$ contained in the interior of a chamber for each $k$
defines an open embedding of the formal algebraic torus
\begin{equation}
\label{Eqn: formal algebraic torus}
\hat \GG_m(\Lambda_\sigma)= \Spf \kk[\Lambda_\sigma]\lfor t\rfor
\end{equation}
into $\X$. Here $\sigma\in \P$ is the maximal cell containing $p$. The
expression $W(p)=\varprojlim W^k(p)$ computes the pull-back of $W$ to this open
formal subscheme. But the embedding $\hat \GG_m(\Lambda_\sigma)\to \X$ depends
on $p$. Expressions for $W(p)$ for different choices of $p$ inside the same
maximal cell $\sigma$ are related by a (typically infinite) series of wall
crossing automorphisms. If $p$ moves to a $p'$ in a neighboring maximal cell
then one needs to use the codimension one relations $XY=f_\rho t^e$ from the
proof of \cite[Lem.\,2.34]{affinecomplex} to see that a similar wall crossing
transformation involving quotients by slab functions relate $W(p)$ and $W(p')$.

In any case, for each general $p$ we obtain an expression for $W$ in $\kk
[\Lambda_\sigma]\lfor t\rfor$ rather than in the algebraic subring $\kk\lfor
t\rfor [\Lambda_\sigma]$.

However, as we will see in the examples in the last three section,
sometimes $\X$ is algebraizable and there exists a point $p\in B$ such that only
finitely many broken lines have endpoint $p$. Then $W(p)$ lies in the subring of
finite type
\[
\kk[\Lambda_\sigma][t] \subset \kk\lfor t\rfor[\Lambda_\sigma] \subset \kk[\Lambda_\sigma]\lfor t \rfor,
\]
hence is even a Laurent polynomial with algebraic coefficients in $t$. It is
then tempting to believe that this algebraic expression describes a lift of $W$
to an algebraization of $\X$. But this may not be the case: Formal local
representability by an algebraic expression neither means that $W$ lifts to an
algebraization, nor that such a lift could locally be represented with
polynomial coefficients in $t$ rather than with coefficients that are formal
power series. In such a situation we therefore say that $W$ is \emph{ostensibly
algebraic}.

To conclude algebraizability of $W$, one rather has to write $W$ as a finite sum
of quotients of sections of the polarizing line bundle $\O_\X(1)$. This can
indeed sometimes be achieved by using \emph{generalized theta functions},
defined analogously to $W$ via sums of broken lines, see
\cite{thetasurvey,Theta}. This method, which is beyond the scope of this paper,
appears to work for example in all cases derived from reflexive polytopes via
Construction~\ref{Const: Reflexive polytope degeneration}.
\end{remark}


\section{Fibers of the superpotential over $0,1,\infty$, and the LG mirror map}
\label{sect: special fiber}

Given a finite scattering diagram $\scrS_k$ for a non-compact $(B,\P,\varphi)$
consistent to order $k$ as in Assumption~\ref{overall assumption}, we have
constructed in the last section the superpotential
\begin{equation}
W:\X\lra \AA^1.
\end{equation}
as a morphism of formal schemes.

In this section we discuss some general features of $W$. Denote by
$A\subset|\X|=X_0$ the union of complete irreducible components of $X_0$,
that is, the union of the irreducible components $X_\sigma\subset X_0$ defined
by the bounded maximal cells $\sigma\in\P$.

We start with $W^{-1}(0)$, viewed as a Cartier divisor on the ringed
space $\X$. To compute the order of $W$ along $X_\sigma\subset\X$ we define the following notion.

\begin{definition}
For a maximal cell $\sigma\in\P$, $\depth\sigma$ is the minimum
of $\ord_\sigma m_\beta$ for all broken lines $\beta$ with endpoint in $\Int\sigma$.
\end{definition}

\begin{proposition}
\label{Prop: 0-fiber of W}
The multiplicity of the Cartier divisor $W^{-1}(0)\subset \X$ on the irreducible
component $X_\sigma\subset X_0$ given by a maximal cell $\sigma\in\P$ equals
$\depth\sigma$.
\end{proposition}

\begin{proof}
This is obvious from the definition of $W_k$ in \eqref{Eqn: W^k}.
\end{proof}

If $(B,\P)$ is asymptotically cylindrical, properness of $W|_{X_0}$
(Proposition~\ref{properness criterion}) implies that if $\sigma\in\P$ is a
maximal cell with $W^{-1}(0)\cap X_\sigma\neq \emptyset$ then $\sigma\in\P$ is
bounded. In other words, $W^{-1}(0)\cap X_0\subseteq A$. Otherwise not much can
generally be said about $W^{-1}(0)$.
\medskip

We next turn to the behavior of $W$ over $\infty$. This discussion makes only
sense in the case that $\X\to\Spf\kk\lfor t\rfor$ is compactifiable, that is, if
it extends to a proper family $\ol\X\to\Spf\kk\lfor t\rfor$. But even in this
case, $W$ may not be a formal meromorphic function
\cite[Tag~01X1]{stacks-project} near $\infty$. In other words, $W$ may have
essential singularities near a compactifiying formal closed subscheme. We
therefore assume $(B,\P)$ compactifiable and $W$ to extend to a meromorphic
function on the corresponding partial completion $\ol\X$ of $\X$. A sufficient
condition is if $W$ itself is algebraizable, as discussed in Remark~\ref{Rem:
algebraizability}.

Another problem is that the meromorphic lift may have a locus of indeterminacy
containing components of the added divisor at $\infty$ on the central fiber.
Here is a purely toric example.

\begin{example}
Let $(B,\P)$ be given by the complete fan in $\RR^2$ with rays generated by
$(1,0)$, $(0,1)$, $(0,-1)$ and $(-1,k)$ for some $k\ge 2$, and $\varphi$ the
convex PL function with value $1$ at all the ray generators. The discriminant
locus is empty. With trivial gluing data and trivial wall structure we obtain
the completion at $t=0$ of a toric threefold $\shX$ over $\AA^1=\Spec\kk[t]$.
Letting $x,y$ be the monomial functions defined by $(1,0)$ and $(0,1)$, the
superpotential equals
\[
W=x+y+ty^{-1}+tx^{-1}y^k
\]
on the big torus. Using $\varphi$ for the truncation, we obtain a
compactification $\ol\shX$ of $\shX$ by intersecting $B$ with the $4$-gon $\ol
B$ with vertices $(k,0)$, $(0,k)$, $(0,-k)$ and $(-1,k)$. Now express $W$ in the
neighborhood
\[
\Spec\kk[u,v^{\pm 1},t]\subset \ol\shX,\quad y=u^{-1}, x=u^{-1}v,
\]
of the divisor $\shD_\rho\subset \partial\shX$ defined by the face $\rho\subset
\ol B$ with vertices $(k,0),(0,k)$:
\[
W= vu^{-1}+u^{-1}+tu+tu^{-k+1}v^{-1}= \frac{(1+v)u^{k-2}+tu^k+tv^{-1}}{u^{k-1}}.
\]
This rational function on $\ol\shX$ has locus of indeterminacy $u=t=0$.
\end{example}

In the asymptotically cylindrical case a meromorphic extension of $W$ has empty
locus of indeterminacy. We therefore now restrict ourselves to the following
situation. Let $(B,\P)$ be compactifiable and asymptotically cylindrical and let
$(\ol B_\nu,\ol\P_\nu)$ be an associated sequence of truncations obtained by
Lemma~\ref{Lem: Exhaustion}. Denote by $\ol\X\to \Spf\kk\lfor t\rfor$ the
corresponding formal toric degeneration constructed in Proposition~\ref{Prop:
wall structures exists for compactifiable cases} from the associated sequence
$(\ol B_\nu,\ol \P_\nu)$ of truncations from Lemma~\ref{Lem: Exhaustion}, and
$\shZ\subset\ol\shX$ the compactifying reduced divisor. Assume further that $W$
lifts to a meromorphic function $\shW$ on an algebraization $\ol\shX$ of
$\ol\X$.

For a sequence $F=(F_\nu)$ of maximal flat affine subspaces of $\partial \ol
B_\nu$ denote by $\shZ_F\subset\ol\shZ$ the irreducible component defined in
Remark~\ref{Rem: compactifying divisor}. For such an $F$ define a multiplicity
$\mu_F\in\NN\setminus\{0\}$ as follows. Let $\rho\subseteq F$ be an
$(n-1)$-cell, $\sigma\in\P$ the unique maximal cell containing $\rho$ and
$\xi\in\Lambda_\sigma$ the primitive generator of $\Lambda_\omega$ for an
unbounded $1$-cell $\omega$ pointing in the unbounded direction. Then
\begin{equation}
\label{Eqn: mu_F}
\mu_F= \big[ \Lambda_\sigma: \Lambda_\rho+ \ZZ\cdot\xi\big].
\end{equation}
Note that $\mu_F$ depends not on the choice of $\rho\subset F$.

\begin{proposition}
\label{Prop: infty-fiber of W}
Let $\shW$ be a meromorphic lift of the superpotential $W$ to an algebraization
$\ol\shX$ of the partial compactification $\ol\X$ of $\X$ from
Proposition~\ref{Prop: wall structures exists for compactifiable cases} of the
compactifiable, asymptotically cylindrical $(B,\P,\varphi)$. Then $\shW$ defines
a morphism of schemes $\shW:\ol\shX\to\PP^1$, and it holds
\[
\ol\shW^{-1}(\infty)= \sum_F \mu_F\cdot \shZ_F.
\]
The sum is over all sequences of parallel maximal flat affine subspaces of $\partial B_\nu$ and $\mu_F$ is defined in \eqref{Eqn: mu_F}.
\end{proposition}

\begin{proof}
Since $W$ defines a proper map to $\AA^1$ the locus of indeterminacy of $\shW$
is empty. Thus $\shW$ defines a morphism to $\PP^1$. Let $\sigma\in\P$ be an
unbounded maximal cell and $X_\sigma\subseteq X_0$ the corresponding irreducible
component of the central fiber. Then $\shW|_{X_\sigma} =W|_{X_\sigma}\neq0$, so
the multiplicity of a prime divisor $\shZ_F\subseteq \shW^{-1}(\infty)$ can be
computed after restriction to the central fiber $X_0$, that is, from $W^0$, and
$X_\sigma$ is a component of the reduced divisor $\shZ\cap X_0$. Moreover, in
the coordinate ring $\kk[\Lambda_\sigma]$ of the big torus of $X_\sigma$ we have
$W^0= z^\xi$. Thus the coefficient of $\shZ_F\cap X_0$ in
$(W^0)^{-1}(0)$ equals
\[
\ord_{D_\rho}W^0= \ord_{D_\rho}z^\xi = - [ \Lambda_\sigma: \Lambda_\rho+ \Lambda_{\omega_i}],
\]
by a standard fact in toric geometry.
\end{proof}
\medskip

The most interesting result in this section concerns a tropical description of
the special fiber $W^{-1}(1)$ in the asymptotically cylindrical case. We will
see that it is described by the asymptotic behaviors of $(B,\P,\varphi)$ and
$\scrS_k$. Moreover, $W^{-1}(1)$ is canonically the mirror toric degeneration of
the anticanonical divisor $\check\D\to T$ from Theorem~\ref{Thm: Properness}, up
to a tropically interesting mirror map reparametrizing the codomain $\AA^1$ of
$W$.

\begin{construction} (\emph{Asymptotic tropical manifold and asymptotic
scattering diagram.})
\label{Constr: B_infinity}
Assume that $(B,\P)$ is asymptotically cylindrical. Denote by $K\subseteq B$ the
compact subset defined by the union of bounded cells of $\P$. Then there exists
a non-zero integral vector field $\xi\in\Gamma(B\setminus K, i_*\Lambda)$ that
is parallel to all unbounded $1$-cells of $\P$. We fix $\xi$ uniquely by
requiring it to be primitive (indivisible) and pointing in the unbounded
direction. The integral curves of $\xi$ generate an equivalence relation $\sim$
on $B\setminus K$. Define the \emph{asymptotic tropical manifold} $B_\infty$
associated to $B$ as the quotient $(B\setminus K)/\!\sim$.

An explicit description of $B_\infty$ runs as follows. Let $\sigma\in\P$ be an
unbounded cell. Let $\bar\sigma$ be the convex hull of the vertices. Since $B$
is asymptotically cylindrical it holds
\[
\sigma= \bar\sigma+\RR_{\ge 0}\cdot\xi,
\]
as an equation of subsets of $\Lambda_{\sigma,\RR}$. Then define $\sigma_\infty$
as the image of $\sigma$ or $\bar\sigma$ under the canonical quotient
\begin{equation}
\label{Eqn: asymptotic monomial map}
\Lambda_{\sigma,\RR} \to \Lambda_{\sigma,\RR}/\RR_{\ge 0}\cdot\xi.
\end{equation}
Clearly, this construction is compatible with the inclusion of faces. Taking the
colimit of the $\sigma_\infty$ defines $B_\infty$ along with a polyhedral
decomposition $\P_\infty$. The charts for the affine structure at vertices of
$B_\infty$ are induced by the charts of $B$ along unbounded $1$-cells of $\P$.
Note that unbounded $1$-cells of $\P$ are disjoint from $\Delta$. It is also
clear that the strictly convex PL function $\varphi$ on $(B,\P)$ induces a
strictly convex PL function $\varphi_\infty$ on $(B_\infty,\P_\infty)$. We have
thus constructed a polarized tropical manifold
$(B_\infty,\P_\infty,\varphi_\infty)$ of dimension $\dim B-1$, the
\emph{asymptotic tropical manifold} of the asymptotically cylindrical
$(B,\P,\varphi)$.

A monomial at a point $x$ on an unbounded component of $\tau\setminus\Delta$ for
$\tau \in\P$ induces a monomial at the image $x_\infty$ of the corresponding
cell $\tau_\infty =\tau/{\RR_{\ge0} \xi}$. Since $\scrS_k$ is finite there exist
a compact subset $K'\subseteq B$ such that only unbounded slabs or walls
intersect $B\setminus K'$. We call these slabs or walls
\emph{asymptotic}. In the present asymptotically cylindrical case $\xi$ is then
tangent to the suppport of any such asymptotic slab or wall. Thus for
any asymptotic slab $\fob$ or wall $\fop$ in $\scrS_k$ the image under
the quotient map $B\setminus K'\to B_\infty$ defines the support of a slab or
wall in $(B_\infty,\P_\infty,\varphi_\infty)$. The associated slab function or
exponent is defined via the projection of monomials \eqref{Eqn: asymptotic
monomial map}. Define by $\scrS_k^\infty$ the structure on $B_\infty$ obtained
in this way.
\qed
\end{construction}

To further relate the wall structures $\scrS_k$ on $B$ and $\scrS_k^\infty$ on
$B_\infty$, only the rings $R^k_{g,\sigma}$ are relevant with $g:\omega\to \tau$
an inclusion of unbounded cells. Denote by $g_\infty:\omega_\infty\to
\tau_\infty$ the induced inclusion of cells in $\P_\infty$. Taking a splitting
of the inclusion $\ZZ\cdot\xi\subseteq \Lambda_\sigma$ provides a
(non-canonical) isomorphism
\begin{equation}
\label{Eqn: w}
\big(R^k_{g,\sigma}\big)_{z^\xi} \simeq  R^k_{g_\infty,\sigma_\infty}[w,w^{-1}],
\end{equation}
where by abuse of notation $\xi$ denotes the unique monomial $m$ of order $0$
with $\ol m=\xi\in\Lambda_\sigma$. The isomorphism identifies $z^\xi$
with $w$.

Mapping $w$ to $1$ now induces the canonical isomorphism of quotients
\begin{equation}
\label{eqn: quotient hom}
\big(R^k_{g,\sigma}\big)_{z^\xi} /(z^\xi-1)\simeq R^k_{g_\infty,\sigma_\infty}.
\end{equation}
Note this isomorphism is compatible with the map of monomials discussed in
Construction~\ref{Constr: B_infinity} and does not depend on choices. In
particular, there is a well-defined formal function $w$ on $\X\setminus A$
locally given by $w$ in \eqref{Eqn: w}. From~\eqref{Eqn: w} it is also obvious
that consistency of $\scrS_k$ implies consistency of $\scrS_k^\infty$.

\begin{proposition}
\label{Prop: w^(-1)(1)}
Let $(\pi:\X\to\Spf\kk\lfor t\rfor, W)$ be the Landau-Ginzburg model with an
asymptotically cylindrical intersection complex $(B,\P)$. Then the composition
$w^{-1}(1)\to \X\stackrel{\pi}{\to} \Spf\kk\lfor t\rfor$ of the inclusion
followed by $\pi$ is canonically isomorphic to the toric degeneration defined by
the compatible system of wall structures $\scrS_k^\infty$ on
$(B_\infty,\P_\infty,\varphi_\infty)$.
\end{proposition}

\begin{proof}
It is enough to check the statement for a fixed finite order $k$. The key
observation is that the asymptotic vector field $\xi$ is tangent to all
unbounded walls on $B$. Now for fixed $k$ the complement $U\subset B$ of the
compact subset $K'\subset B$ in Construction~\ref{Constr: B_infinity} only
intersects unbounded walls. Moreover, gluing the rings associated to chambers in
$U$ is enough to describe $X_k\setminus A$. The statement now readily follows
from the construction of $\scrS_k^\infty$ and \eqref{eqn: quotient hom}.
\end{proof}

\begin{remark}
\label{Rem: parallel asymptotic monomials}
Monomials in unbounded walls or slabs proportional to $z^\xi$ map to elements
of the base ring $\kk\lfor t\rfor$ under the homomorphism~\ref{eqn: quotient
hom}. Such constant terms do not appear in the wall structures constructed
either by the algorithm in \cite{affinecomplex} or in the canonical wall
structure of \cite{CanonicalWalls}. Thus $\scrS_\infty$ is a more general wall
structure than previously constructed that also involves \emph{undirectional
walls}, that is, walls $\fop$ with $f_\fop\in\kk\lfor t\rfor^\times$. This fact
has been overlooked in earlier versions of this paper and affects the mirror
statements for $w^{-1}(1)$.

The corrected statement is given in Proposition~\ref{Prop: mirror statement for
w^{-1}(0)} below.
\end{remark}

To understand the influence of undirectional walls, we observe a close
relationship to gluing data. 

\begin{remark}
\label{Rem: Unidirectional walls from gluing data}
Consider a wall structure $\scrS$ on an affine manifold with singularities $B$
with all walls and slabs undirectional. To emphasize the constant nature of the
slab and wall functions, we now write $c_\fop, c_\fob\in \kk\lfor t\rfor^\times$
rather than $f_\fop, f_\fob$. Let $\foj$ be a zero-codimensional joint, that is,
intersecting the interior of a maximal cell $\sigma$. Then the automorphism
$\theta_\fop$ of $\kk\lfor t\rfor[\Lambda_\sigma]$ associated to a wall $\fop$
containing $\foj$ equals
\begin{equation}
\label{Eqn: Undirectional wall automorphism}
\theta_\fop: z^m\longmapsto
\big\langle c_\fop\otimes n_\fop,m\big\rangle\cdot z^m,
\end{equation}
with $n_\fop\in\check\Lambda_\sigma=\Hom(\Lambda_\sigma,\ZZ)$ the primitive
normal vector spanning $\Lambda_\fop^\perp$, with sign depending on the
direction of wall crosssing. Now all such automorphisms $\theta_\fop$ with
$\foj\subset\fop$ commute. Moreover, their product is trivial iff
\begin{equation}
\label{Eqn: balancing in codim 0}
\prod_\fop c_\fop \otimes n_\fop=1\otimes 0\in
\kk\lfor t\rfor^\times\otimes_\ZZ\check\Lambda_\sigma.
\end{equation}
Note that in the tensor product the first factor is written multiplicatively,
the second additively, so $1\otimes 0$ is the unit in this abelian group. 

We observe that the consistency condition \eqref{Eqn: balancing in codim 0} for
$\scrS$ at $\foj$ is the cocycle condition at $\foj$ for the \emph{tropical
$1$-cocycle} on $B$ supported on $|\scrS|$ that assigns $c_\fop\otimes n_\fop$,
$c_\fob\otimes n_\fop$ to the elements of $\scrS$. This motivates to view a
consistent, undirectional wall structure as a tropical 1-cocycle, with the
cocycle condition reflected by consistency in all codimensions. Let us denote
the group of tropical $1$-cocyles by $C^{n-1}_{\mathrm{trop}}(B)$.

Next note that the ring homomorphisms defined by undirectional walls have the
same form as in applying open gluing data, see e.g.\ \cite[(5.2)]{Theta}. In the
simple singularity case, the group of equivalence classes of lifted open gluing
data obeying a similar local consistency condition is given by the cohomology
group $H^1(B,\iota_*\check\Lambda\otimes\kk^\times)$ \cite[Prop.\,4.25, Def.\,5.1,
Lem.\,5.5]{logmirror}. Now there is an obvious map
\[
C^{n-1}_{\mathrm{trop}}(B)\lra
H_{n-1}(B_0,\check\Lambda\otimes \kk\lfor t\rfor^\times),
\]
which by consistency factors over $H_{n-1}(B,\iota_*\check\Lambda\otimes
\kk\lfor t\rfor^\times)$. Moreover, by \cite[Thm.\,1]{Ruddat} we have an
isomorphism
\[
H_{n-1}(B,\iota_*\check\Lambda\otimes \kk\lfor t\rfor^\times)
\simeq H^1(B,\iota_*\check\Lambda\otimes \kk\lfor t\rfor^\times).
\]
Thus we obtain a homomorphism
\[
C^{n-1}_{\mathrm{trop}}(B)\lra
H^1(B,\iota_*\check\Lambda\otimes \kk\lfor t\rfor^\times.
\]
This map associates to the tropical $1$-cocycle given by the undirectional
wall structure $\scrS$ certain open gluing data that we denote $s_\scrS$.

Now one can run the algorithm in \cite{affinecomplex} in two ways. First as
usual with a cocycle representative of the gluing data $s_\scrS$, leading to a
wall structure $\scrS'$. Second, starting with an initial wall structure that
takes into account the reduction modulo $t$ of $\scrS$, interpreted as gluing
data; then run the algorithm with the undirectional walls inserted at each order
to obtain a wall structure $\scrS''$. Consistency of the undirectional wall
structure $\scrS$ should be a necessary and sufficient condition for this to
work. We conjecture that the two families $\foX',\foX''\to T$ obtained from
$\scrS',\scrS''$ are isomorphic.

One can compare $\foX',\foX''$ by choosing a general point in each cell as a
reference point and relate the diagrams of schemes in
\cite[\S\,2.6]{affinecomplex} by sequences of wall crossing automorphisms. To
prove that this procedure induces an isomorphism of diagrams would require to
carefully analyze the scattering algorithm on a neighborhood of the interior of
each cell of $\P$, including the difference of the presence of undirectional
walls versus the associated gluing data.

For the application to our asymptotic wall structure $\scrS_\infty$ on
$B_\infty$, one takes for the walls of the undirectional wall structure all
undirectional walls of the wall structure on the asymptotically
cylindrical $B$. Applying \cite[Prop.\,3.10,(2)]{affinecomplex} one can prove
consistency at codimension $0$ joints order by order. For the slabs one observes
that undirectional monomials in an unbounded slab $\fob$ are never of order $0$.
Hence as in \cite[Thm.\,5.2]{affinecomplex} we can factor
\[
f_\fob= \bar f_\fob\cdot f_\fob^\parallel
\]
with $f_\fob^\parallel\in \kk\lfor t\rfor[z^{\pm\xi}]^\times$ and $\bar f_\fob$
having a product decomposition with no factors in $\kk\lfor t\rfor[z^{\pm\xi}]$.
We use $f_\fob^\parallel$ to define the slab of the undirectional wall structure
on $B_\infty$ induced by $\fob$. A very careful analysis of the algorithm, which
we have not carried out, should now show that this undirectional wall structure
is consistent.
\qed
\end{remark}

We finally relate the mirror of $\check\foD\to T$ to the LG-fiber $w^{-1}(1)$.
This makes precise and proves a conjecture of Auroux in our setup
\cite[Conj.\,7.4]{auroux1}. 

\begin{proposition}
\label{Prop: mirror statement for w^{-1}(0)}
If $\X\to\Spf\kk\lfor t\rfor$ is mirror dual to $(\check\X\to T,\check\D)$ then
$w^{-1}(1)\to \Spf\kk\lfor t\rfor$ is the mirror family to $\check\foD\to T$
twisted by an undirectional wall structure as discussed in Remarks~\ref{Rem:
parallel asymptotic monomials} and~\ref{Rem: Unidirectional walls from gluing
data}.\footnote{In the interpretation of wall structures via punctured
invariants \cite{CanonicalWalls}, undirectional walls count punctured
Gromov-Witten invariants with one positive contact order with $\check\foD$,
hence relate to traditional log Gromov-Witten invariants of
$(\check\X,\check\foD)$. Further details will appear in \cite{GRS}.
}
\end{proposition}

Most interestingly, the function $w$ in Proposition~\ref{Prop: w^(-1)(1)} and
\eqref{Eqn: w} is closely related to the superpotential $W$. To explain this
relation, note that for any fixed finite order $k$ we may restrict to the
complement $U\subset B$ of a compact set such that each broken line $\beta$
ending in $U$ is parallel to the asymptotic vector field $\xi$. In other words,
there exists $c\in\ZZ\setminus\{0\}$ with $m_\beta= c\cdot\xi$. There are two
types of such broken lines, depending on the sign of $c$. If $c>0$ then $\beta$
is (part of) a broken line not intersecting any wall, and hence $c=1$ and the
monomial carried by $\beta$ equals $z^\xi$. Otherwise $\beta$ is a broken line
that returned from entering the compact set due to some non-trivial interaction
with walls outside of $U$. We call these broken lines $\beta$ and the
corresponding monomial $m_\beta$ at the root vertex \emph{outgoing}. In this
case we can write $a_\beta z^{m_\beta} = a_\beta t^l z^{-d\xi}$ for some $l>0$
and $d=-c>0$. Summing over all such broken lines with general endpoint $p\in U$
leads to a polynomial with coefficients $N_{d,l}\in\kk$:
\begin{equation}
\label{Eqn: N(d,l)}
h_k=\sum_{l=1}^k\sum_{d>0} N_{d,l}t^l z^{-d\xi}\in \kk[t,z^{-\xi}].
\end{equation}

\begin{lemma}
The coefficients $N_{d,l}\in\kk$ in \eqref{Eqn: N(d,l)} do not depend on the
choices of $U$, $p\in U$ or $k\ge l$.
\end{lemma}

\begin{proof}
Observe first that $h_k+w$ equals $W^k_{g,\fou}(p)$ from \eqref{Eqn: W^k}, for
$\fou$ any unbounded chamber of $\scrS_k$ containing the chosen endpoint
$p\in\fou$ of broken lines in~\eqref{Eqn: N(d,l)}, and $g:\omega\to\tau$ any
inclusion of cells relevant to $\fou$, that is, with $\omega\cap\fou\neq
\emptyset$, $\tau\subseteq\sigma_\fou$. Since $\xi$ is tangent to all walls
intersecting $U$, this expression is unchanged under wall crossing
automorphisms. The statement now follows from the independence of
$W^k_{g,\fou}(p)$ of the choice of $p\in \fou$ (Lemma~\ref{lem:independence of
p}) and the compatibility of $W^k_{g,\fou}$ with changing strata and chambers
(Lemma~\ref{lem:change of chambers}).

The statement on the independence of $k\ge l$ follows since a broken line
$\beta$ interacting with a term in a wall of order $>k$ has an outgoing monomial
$m_\beta$ of order $>k$.
\end{proof}

We are now in position to define the map relating $w$ and $W$ as the
automorphism $\Phi$ of the formal algebraic torus 
\[
\hat\GG_m=\GG_m\times \Spf\kk\lfor t\rfor
=\Spf\!\big( \kk[u^{\pm1}]\hat\otimes_\kk\kk\lfor t\rfor\big)
= \Spf \kk[u^{\pm1}]\lfor
t\rfor
\]
over $\kk\lfor t\rfor$ defined by
\begin{equation}
\label{Def: mirror map}
\Phi^\sharp(u)= u+\sum_{l>0} \Big(\sum_{d>0} N_{d,l} u^{-d}\Big) t^l=
u\left(1+\sum_{l>0} \Big(\sum_{d>0} N_{d,l} u^{-d-1}\Big) t^l\right).
\end{equation}
Note that for each fixed $l$ there are only finitely many broken lines
contributing to the coefficient of $t^l$ in \eqref{Eqn: N(d,l)}. Hence
$\Phi^\sharp(u)\in \kk[u^{\pm1}]\lfor t\rfor$ as required to define $\Phi$. Note
also that $\Phi$ induces the identity on $\GG_m\subset \GG_m\times\Spf\kk\lfor
t\rfor$, the reduction modulo $t$.

\begin{definition}
\label{Def: LG mirror map}
We call the automorphism $\Phi$ of $\GG_m\times\Spf \kk\lfor t\rfor$ the
\emph{Landau-Ginzburg (LG-) mirror map}.
\end{definition}

We emphasize that $\Phi$ is not in general induced by an automorphism of the scheme
\[
(\GG_m)_{\Spec {\kk\lfor t\rfor}}= \Spec \kk\lfor t\rfor [u^{\pm1}].
\]

For the following mirror map statement in the asymptotically cylindrical case we
consider both $W|_{\X\setminus A}$ and $w$ as maps to $\hat\GG_m=
\GG_m\times\Spf\kk\lfor t\rfor$.

\begin{theorem}
\label{Thm: W versus w}
In the situation of Proposition~\ref{Prop: w^(-1)(1)} it holds
$W|_{\X\setminus A}= \Phi\circ w$.
\end{theorem}

\begin{proof}
The statement follows by observing that modulo $t^{k+1}$,
\[
(w^\sharp\circ\Phi^\sharp)(u)= w+\sum_{l>0}
\Big(\sum_{d>0} N_{d,l} w^{-d}\Big) t^l
\]
equals $W^k_{g,\fou}=h_k+w$ with $h_k$ as in \eqref{Eqn: N(d,l)}, for all unbounded chambers $\fou$.
\end{proof}

Theorem~\ref{Thm: W versus w} together with Proposition~\ref{Prop: mirror statement for w^{-1}(0)} yields the following.

\begin{corollary}
\label{Cor: fiber of W over 1 is mirror to check D}
Let $\X\to\Spf\kk\lfor t\rfor$ be mirror dual to $(\check\X\to T,\check\D)$.
Then $(\Phi^{-1}\circ W)^{-1}(1)\to \Spf\kk\lfor t\rfor$ is the mirror
family of $\check\D\to T$ twisted by the undirectional wall structure discussed
in Remarks~\ref{Rem: parallel asymptotic monomials} and~\ref{Rem: Unidirectional
walls from gluing data}.
\qed
\end{corollary}


\section{Wall structures and broken lines via tropical disks}
\label{sect:tropical disks}

We now aim for an alternative construction of the potential $W$ in terms of
tropical disks.\footnote{ The interpretation of wall structures and the
superpotential in terms of tropical disks of Maslov indices $0$ and $2$ in this
section has a more speculative and inconclusive nature than the rest of the
paper. Recent advances in our understanding of wall structures in the context of
intrinsic mirror symmetry \cite{CanonicalWalls} makes it now feasible to develop
the picture given in this section in full generality. This section is therefore
included with only minor changes from the original version, but should be read
with some caution.}

\subsection{Tropical disks}
\label{subsec:disks}
Our definition of tropical disks depends only on the integral affine geometry of
$B$ and not on its polyhedral decomposition $\P$. As usual let $i:
B\setminus\Delta \to B$ denote the inclusion for $\Delta$` the singular locus of
the integral affine structure, and let $\Lambda_B$ be the sheaf of integral
tangent vectors. We restrict to the asymptotically cylindrical case
(Definition~\ref{Def: parallel edges}). Without reference to $\P$ we require
that $B$ is non-compact and for some orientable compact subset $K\subset B$,
$\Gamma(B\setminus K, i_*\Lambda)$ has rank one. Then there exists a unique
primitive integral affine vector field $\xi$ on $B\setminus K$ pointing away
from $K$. We assume the semiflow of $\xi$ is complete and call its integral
curves the \emph{asymptotic rays}.

\begin{definition}
\label{def:disk}	
Let $\Gamma$ be a tree with root vertex $\rootvertex$. Denote by $\Gamma^{[1]},
\Gamma^{[0]}, \leafvertices$ the set of edges, vertices, and leaf vertices
(univalent vertices different from the root vertex), respectively. We allow
unbounded edges, that is, edges adjacent to only one vertex, defining a subset
$\Gamma_\infty^{[1]}\subseteq \Gamma^{[1]}$. Let $w: \Gamma^{[1]} \to
\mathbb{N}\setminus \{0\}$ be a weight function.

Let $x\in B\setminus \Delta$. A \emph{tropical disk bounded by $x$} is a proper,
locally injective, continuous map
\[
h: \big( |\Gamma|, \{\rootvertex\} \big) \lra \big(B,\{x\} \big)
\]
with the following properties.
\begin{enumerate}
\item
$h^{-1}(\Delta)= \leafvertices$.
\item
For every edge $E\in \Gamma^{[1]}$ the image $h(E\setminus\partial E)$ is a
locally closed integral affine submanifold of $B\setminus \Delta$ of dimension
one.
\item
If $V\in\Gamma^{[0]}$ there is a primitive integral vector $m\in
\Lambda_{B,h(V)}$ extending to a local vector field tangent to $h(E)$ and
pointing away from $h(V)$. Define the \emph{tangent vector of $h$ at $V$ along
$E$} as $\overline m_{V,E}:= w(E)\cdot m$.
\item 
For every $V\in \Gamma^{[0]}\setminus \leafvertices$ the following
\emph{balancing condition} holds:
\[
\sum_{\{E\in \Gamma^{[1]}| \,V\in E\} } \overline m_{V,E}= 0.
\]
\item 
The image of an unbounded edge is an asymptotic ray.
\end{enumerate}
Two disks $h: |\Gamma| \to B$, $h': |\Gamma'| \to B$ are identified if
$h=h'\circ \phi$ for a homeomorphism $\phi: |\Gamma|\to |\Gamma'|$ respecting
the weights.

The \emph{Maslov index} of $h$ is defined as $\displaystyle \mu(h):=
2\sum_{E\in\Gamma^{[1]}_\infty} w(E)$.
\qed
\end{definition}

Note that for a tropical disk $h^*(i_*\Lambda_B)$ is a trivial local system. In
particular, there is a unique parallel transport of tangent vectors along $h$.

\begin{example} \label{triangle}
\begin{figure}[h!] 
\input{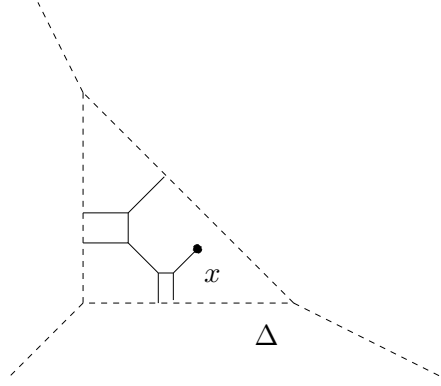}
\caption{A tropical Maslov index zero disk bounding $x$ belonging to a moduli
space of dimension $5$. The dashed lines indicate a part of the discriminant
locus.}
\label{fig:moduli}
\end{figure}
Suppose $\dim B=3$ and $\Delta$ bounds an affine $2$-simplex $\sigma$ with
$T_x\sigma$ contained in the image of $i_*\Lambda_{B,x}$ for all $x\in \Delta$.
Such a situation occurs in the mirror toric degenerations of local Calabi-Yau
threefolds, for example in the mirror of $K_{\mathbb P^2}$
\cite[Expl.\,5.2]{Invitation}. Then any point $x\in \sigma \setminus \Delta$
bounds a family of tropical Maslov index zero disks of arbitrary dimension, as
illustrated in Figure~\ref{fig:moduli}.
\end{example}

So far, our definition of tropical disks only depends on $|\Gamma|$ and not on
its underlying graph $\Gamma$. A distinguished choice of $\Gamma$ is by assuming
that there are no divalent vertices. At an interior vertex $V\in\Gamma^{[0]}$
(that is, neither the root vertex nor a leaf vertex) the rays $\RR_{\ge 0}\cdot
\overline m_{E,V}$ of adjacent edges $E$ define a fan $\Sigma_{h,V}$ in
$i_*\Lambda_{B,h(V)} \otimes_\ZZ\RR$. Denote by $\Sigma_{h,V}^0$ the parallel
transport along $h$ of $\Sigma_{h,V}$ to $\rootvertex$. The \emph{type} of $h$
consists of the weighted graph $(\Gamma,w)$ along with the $\Sigma_{h,V}^0$, $V
\in \Gamma^{[0]}\setminus \Gamma^{[0]}_\infty$. For $x\in B\setminus \Delta$ and
$\overline m\in \Lambda_{B,x}$ denote by $\M_\mu(\overline m)$ the moduli space
of tropical disks of Maslov index $\mu$ and root tangent vector $\overline m$.
It comes with a natural stratification by type: A stratum consists of disks of
fixed type, and the boundary of a stratum is reached when the image of an
interior edge contracts to a vertex of higher valency.
\smallskip

From now on assume $B$ is equipped with a compatible polyhedral structure $\P$
as defined in \cite[\S1.3]{logmirror}. It is then natural to adapt $\Gamma$ to
$\P$ by appending Definition~\ref{def:disk} as follows:
\begin{list}{(6)}
\item
For any $E\in\Gamma^{[1]}$ there exists $\tau\in \P$ with $h(\Int E)\subseteq
\Int(\tau)$, and if $V\in E$ is a divalent vertex then
$h(V)\subseteq\partial\tau$.
\end{list}
\noindent 
In other words, we insert divalent vertices precisely at those points of
$|\Gamma|$ where $h$ changes cells of $\P$ locally. Note however that we still
consider the stratification on $\M_\mu(\overline m)$ defined with all divalent
vertices removed.

\begin{example}
\begin{figure}[h]
\input{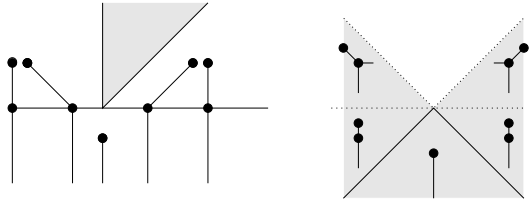}
\caption{Disks near $\Delta$ (left) and their moduli cell complex (right).}
\label{fig:cells}
\end{figure}
As it stands, the type does not define a good stratification of the moduli space
of tropical disks. For each vertex $V\in\Gamma$ mapping to a codimension one
cell $\rho\in\P$ we also need to specify the connected component of $\rho
\setminus \Delta$ containing $h(V)$ (that is, specify a reference vertex
$v\in\rho$). This is illustrated in Figure~\ref{fig:cells}. Here the dotted
lines in the right picture correspond to generalized tropical disks, fulfilling
all but (1) in Definition~\ref{def:disk}.
\end{example}

Tropical disks are closely related to broken lines as follows. We place
ourselves in the context of \S\ref{Ch: Broken lines}. In particular, we assume
given a structure $\scrS_k$ that is consistent to order $k$.

\begin{lemma}[Disk completion]\footnote{Cf.\ \cite[Prop.4.13]{CanonicalWalls} for a refined treatment in a slightly different setup.}
\label{line-disk} 
As a map, any broken line is the restriction of a Maslov index two disk $h:
|\Gamma| \to B$ to the smallest connected subset of $|\Gamma|$ containing the
root vertex and the (unique) unbounded leaf. The restriction of $h$ to the
closure of the complement of this subset consists of Maslov index zero disks.
\end{lemma}

\begin{proof} 
We continue to use the terminology of \cite{affinecomplex}. First we show that
any projected exponent $\overline m$ at a point $p$ of a wall or slab in
$\scrS_k$ is the root tangent vector of a Maslov index zero disk $h$ rooted in
$h(\rootvertex)=p$. This is true for $\scrS_0$ since by simplicity the exponents
of a slab function $f_{\rho,v}$ are root tangent vectors of Maslov index zero
disks with only one edge. Assume inductively this holds as well for $\scrS_{l}$,
$0\le l \le k$, and show the claim for walls in $\scrS_{l+1}\setminus \scrS_{l}$
arising from scattering. We must show that the exponents of the outgoing rays
are generated by those of the incoming rays or cuts. But if there existed an
additional exponent, it would be preserved by any product with log automorphisms
attached to the rays or cuts, as up to higher orders the latter are
multiplications by polynomials with non trivial constant term. This contradicts
consistency.

In particular, if $p=\beta(t_i)$ is a break point of a broken line $\beta$ then
$t_i$ can be turned into a balanced trivalent vertex by attaching a Maslov index
zero disk $h$ with root tangent vector $\overline m$ equal to the projected
exponent taken from the unique wall or slab containing $p$.
\end{proof}

\noindent
We call any tropical disk as in the lemma a \emph{disk completion} of the broken
line. The disk completion is in general not unique due to the following reasons:
\begin{enumerate}
\item
First, Example~\ref{triangle} shows that tropical Maslov index zero disks may
come in families of arbitrarily high dimension.
\item
Even if the moduli space of tropical Maslov index zero disks is of expected
finite dimension, there may be joints with different incoming root tangent
vectors.
\item
There may exist several Maslov index zero disks with the same root
tangent vector, for example a closed geodesic of different winding numbers.
\end{enumerate}

We now take care of these issues.


\subsection{Virtual tropical disks}
Example~\ref{triangle} illustrates that for $\dim B\geq 3$, a tropical disk
whose image is contained in a union of slabs leads to an unbounded dimension of
the moduli space of tropical Maslov index $2\mu$ disks. In order to get
enumerative invariants which recover broken lines we need a virtual count of
tropical disks. Throughout we assume $B$ is oriented.
\smallskip

Suppose $\Delta$ is straightened as in \cite[Rem.\,1.49]{logmirror}, that is,
$\Delta$ defines a subcomplex $\Delta^\bullet$ of the barycentric refinement of
the polyhedral decomposition $\P$ of $B$. Note that the simplicial structure of
$\Delta^\bullet$ refines the natural stratification of $\Delta$ given by local
monodromy type. Let $\Delta_\max$ denote the set of maximal cells of $\Delta
^\bullet$ together with an orientation, chosen once and for all. Each $\tau \in
\Delta_\max$ is contained in a unique $(n-1)$-cell $\rho\in\P$. Then monodromy
along a small loop about $\tau$ defines a monodromy transvection vector
$m_\tau\in\Lambda_\rho$, where the signs are fixed by the orientations via some
sign convention. In view of the orientations of $\tau$ and $B$ we can then also
choose a maximal cell $\sigma_\tau\supset \tau$ unambiguously.

For each $\tau \in \Delta_\max$ let $w_\tau$ be the choice of a partition of
$|w_\tau|\in \mathbb N$ (with $w_\tau = \emptyset$ for $|w_\tau|=0$). To
separate leaves of tropical disks we will now locally replace $\Delta$ by a
branched cover. We can then consider deformations of a disk $h$ whose leaves end
on that cover instead of $\Delta$, with weights and directions prescribed by the
partitions $\mathbf{w}:=(w_\tau|\tau \in \Delta_\max)$.
\medskip

\paragraph{\emph{Deformations of $\Delta$}}
We first define a deformation of the barycentric refinement $\Delta_\bary$ of
$\Delta$ as a polyhedral subset of $B$. For each $\tau \in \Delta_\max$, denote
by $\ray_\tau \subseteq \sigma_\tau$ the 1-cell connecting the barycenter
$b_\tau$ of $\tau$ to the barycenter of $\sigma_\tau$. Note that
$\Lambda_{\ray_\tau} \otimes_\ZZ \RR$ intersects $i_*\Lambda_{b_\tau}
\otimes_\ZZ \RR= \mathrm{span}(\Lambda_\tau,m_\tau)$ transversally. Moving the
barycenter of the barycentric refinement $\tau_\bary$ of $\tau$ along
$\ray_\tau$ while fixing $\partial \tau_\bary$ now defines a piecewise linear
deformation $\tau_s$ of $\tau$ over $s\in \ray_\tau$ as a polyhedral subset of
$\sigma_\tau$. Thus we obtain a deformation $\{ \Delta_s |s \in S\}$ of $
\Delta$ over the cone $S:=\prod_\tau \ray_\tau$. It is trivial as deformation of
cell complexes, as parallel transport in direction $\ray_\tau$ in each cell
$\sigma_\tau$ induces an isomorphism of cell complexes $\Delta_\bary \cong
\Delta_s$.

For an infinitesimal point of view let $i_\tau:\tau\to\sigma$ be the inclusion.
Consider the preimage of the deformation of $\tau\subseteq\Delta$ under the
natural inclusion $\sigma_\tau \hookrightarrow i_\tau^*T\sigma_\tau$. For
$s=(s_1,\ldots, s_{\length w_\tau})\in \ray_\tau^{\length w_\tau}$ with $s_i$
pairwise different,
\[
\cover:= \bigcup_{k=1}^{\length w_\tau}
\tau_{s_k} \subseteq i_\tau^*T\sigma_\tau
\]
is then a $\length w_\tau$-fold branched cover of $\tau$, via the natural
projection
\[
\pi: i_\tau^*T\sigma_\tau \to \tau.
\]

Note that $\cover = \emptyset$ if $|w_\tau|=0$ and $\cover \subseteq
\tau_{s'}^{\mathbf w'}$ if $\length w_\tau \leq \length w_\tau'$ and if the
entries of $s$ agree with the first $\length w_\tau$ entries of $s'$. We
make $\tau_s^{\mathbf w}$ into a weighted cell complex by equipping each cell of
$\cover$ with the weight defined by the partition $w_\tau$. Finally, set
$\defbase:= \prod_\tau \ray_\tau^{\length w_\tau}$ and $\Cover:= \bigcup_\tau
\tau^{\mathbf w}_{s_\tau}$, where $s \in \defbase$. We still call $\Cover$ a
deformation of $\Delta$, as for $\epsilon \to 0$, $\tilde \Delta_{\epsilon
s}^{\mathbf w}$ converges to $\Delta$ as \emph{weighted} complexes in an obvious
sense.
\medskip
 
\paragraph{\emph{Deformations of tropical disks}} 
 
We now want to define a virtual tropical disk as an infinitesimal deformation
$\tilde h$ of a tropical disk $h$ such that the leaves of $\tilde h$ end on
$\Cover$ as prescribed by $\mathbf w$.

The idea is that for small $\epsilon>0$ and suitable environments $ U_v\subseteq
T_vB$ of $0$, $v=h(V)$ the image of an internal vertex, the rescaled exponential
map $\exp\big|_{\bigcup_v \left(\frac1\epsilon U_v\right)} \circ \epsilon
\id_{T(B\setminus \Delta)}$ maps the union of the tropical curves $\tilde h_V$
in the following Definition~\ref{Def: virtual tropical disk} onto the image of a
tropical disk $\tilde h_\epsilon: \tilde \Gamma \to B$ with leaves emanating
from $ \Delta_{\epsilon s}^{\mathbf w} $. By choosing $\epsilon>0$ sufficiently
small, the image of $\tilde h_\epsilon$ is contained in an arbitrary small
neighborhood of the image of $h$. Thus $\tilde h_\epsilon$ indeed defines a
deformation of $h$. Conversely, for $\epsilon$ sufficiently small, $\tilde
h_\epsilon$ determines $\tilde h$ uniquely. Hence in order to simplify language
and visualization, we may and will identify a virtual curve $\tilde h$ with its
``images" $\tilde h_\epsilon$ in $B$ for small $\epsilon>0$.

\begin{definition}
\label{Def: virtual tropical disk}
Let $h: |\Gamma|\to B$ be a tropical disk not intersecting $|\Delta^{[\dim
B-3]}|$. A \emph{virtual tropical disk $\tilde h$ of intersection type $\mathbf
w$ deforming $h$} consists of:
\begin{enumerate}
\item
For each interior vertex $V\in\Gamma^{[0]}$, a possibly disconnected genus zero
ordinary tropical curve $\tilde h_V: \tilde \Gamma_V \to T_vB$ with respect to
the fan $\Sigma_{h,V}$. This means that $\tilde \Gamma_v$ is a possibly
disconnected graph with simply connected components and without di- and
univalent vertices, the map $\tilde h_V$ satisfies conditions (2)--(4) of
Definition \ref{def:disk}, while instead of (5) the unbounded edges are parallel
displacements of rays of $\Sigma_{h,V}$.
\item
A cover $\tilde h_E$ of each edge $E$ of $h$ by weighted parallel sections of
the normal bundle $h|_E^*TB/Th(E)$. For each edge $E$ adjacent to an interior
vertex $V$, we require that the inclusion defines a weight-respecting bijection
between the cosets of $\tilde h_E^{-1}(V)$ over $V$ and rays of $\tilde h_V$ in
direction $E$. Moreover, the intersection defines a weight-preserving bijection
between the cosets of $\{ \tilde h_E^{-1}(V) \, | \, h( V)\in \tau, E\ni V \}$
over the leaf vertices in $\tau$ and branches of $\cover$.
\item
A virtual root position, that is, a point $\tilde h_{\rootvertex}
(\widetilde\rootvertex)$ in $T_{h(\rootvertex)}B$ such that $\tilde
h_{\rootvertex}(\widetilde\rootvertex) + Th(E) = \tilde h_E^{-1}({\rootvertex})$, where $E$ is
the root edge.
\end{enumerate}
\end{definition}

We denote by $\M_{\mu}(\Cover,h)$ the moduli space of virtual Maslov index $\mu$
disks of intersection type $\mathbf w$ deforming $h$.

In order to exclude the phenomenons in Example~\ref{triangle}, we now restrict
to sufficiently general tropical disks. For such tropical disks a local
deformation of the constraints on $\tilde h(\tilde \Gamma^{[0]})$ lifts to a
local deformation of $\tilde h$ preserving the type. Formally, we define:

\begin{definition}
Let $ s\in \defbase$ and $\mu \in\{0,1\}$. A virtual tropical disk $\tilde h \in
\M_{2\mu}(\Cover,h)$ is \emph{sufficiently general} if:
\begin{enumerate}
\item
$\tilde h$ has no internal vertices of valency higher than three,
\item
all intersections of $\tilde h$ with the codimension one cells of $\P$ are
transverse intersections at divalent vertices outside $|\P^{[\dim B-2]}|$,
\item 
there exists a subspace $L\subseteq T_{h(\rootvertex)}B$ of dimension
$\max\{1-\mu,0\}$ and an open cone $\defcone \subseteq \defbase$ containing $s$
such that the natural map 
\begin{equation}
\label{eq:constraintmap}
\pi \times \mathrm{ev}_{\widetilde \rootvertex}: \quad 
\bigcup_{s\in \defcone}\M_\mu(\Cover,h) \lra \defcone \times
\big(T_{h(\rootvertex)}B\big)/ L
\end{equation}
is open.
\end{enumerate}
$\Cover$ is in \emph{general position} if for all Maslov index zero disks $h$
the complement of the set $\M_0(\Cover,h)^{gen}$ of sufficiently general disks
in $\M_0(\Cover, h)$ is nowhere dense.
\end{definition}

\begin{lemma} \label{dimModuli}
Given $\mathbf w$, the space of non-general position deformations of $\Delta$ is
nowhere dense in $S^{\mathbf w}$. For general position, $\M_{2\mu}(\Cover,h)$ is
of expected dimension $\dim B+\mu-1$.
\end{lemma}

\begin{proof}(Sketch)
Consider a generalized class of tropical disks by forgetting the leaf
constraints, allowing edge contractions and replacing condition~(6) by the
assumption that the graph contains no divalent vertices. Fix a type with a
trivalent graph $\Gamma$. Then any tropical disk of the given type is determined
by the position of the root vertex $x$ and the length of the $N:=| \Gamma^{[1]}
| -\mu = 2| \leafvertices|-1 -\mu$ bounded edges. This shows that the inverse
map restricts to an open embedding $ \bigcup_{s\in \defbase} \M_\mu(\Cover)\to
B\times \mathbb R^{N}_{\geq 0}$ in obvious identifications dictated by $\Gamma$.
The statement now follows from the observation that the map
\eqref{eq:constraintmap} expressed in the induced affine structure on
$\M_\mu(\Cover,h)$ is piecewise linear, and any violation of stability defines a
subset of a finite union of hyperplanes. In particular, the dimension statement
follows by noting that the positions of the leaf vertices define constraints of
codimension $2 (| \leafvertices | - \mu)$. 
\end{proof}

\begin{remark}
\label{foj}
A stratum of $\M_0(\overline m)$ admits a natural affine structure. Hence a disk
$h\in \M_0(\overline m)^{[k]}$ belonging to a $k$-dimensional stratum naturally
comes with the $k$-dimensional subspace of induced infinitesimal vertex
deformations
\[
\foj_V(h)^{[k]}:=T_h \mathrm{ev}_{V}(T_h\M_0^{[k]}(\overline m))
\subseteq T_{h(V)}B.
\]
Likewise, infinitesimal deformations of a sufficiently general Maslov index zero
disk $\tilde h$ give rise to \emph{virtual joints}, that is, the codimension two
subsets defined by restricting the deformation family to the vertices. Such
virtual joints converge to some codimension two space $\foj_V(\tilde h)\in
T_{h(V)}B$, as $s \to 0 \in\defcone$. Moreover, if the limiting disk $h$ of
$\tilde h$ belongs to a $(\dim B- 2)$-stratum of $\M_0(\overline m)$ such that
\eqref{eq:constraintmap} extends to an open map at $0\in \partial \defcone$,
then $ \foj_V(h)^{[\dim B-2]} = \foj_V(\tilde h)$. This may be used to define
stability for tropical disks.
\end{remark}


\subsection{Structures via virtual tropical disks}
We now relate the counting of virtual Maslov index zero disks with the
structures of \cite{affinecomplex}. Let $\scrB$ be the set of closures of
connected components of the codimension one cells of $\P$ when $\Delta$ is
removed. For $\fob\in \scrB$ contained in $\rho\in \P^{[n-1]}$ and $v\in\rho$
a vertex contained in $\fob$ denote by $f_\fob:= f_{\rho,v}$ the order zero
slab function attached to $\fob \in \scrB$ via the log structure. Then $f_\fob
\in \kk[C_\fob]$ where $C_\fob$ is the monoid generated by one of the two
primitive invariant $\tau$-transverse vectors $\pm m_{\tau}$ for each positively
oriented $\tau \in \Delta^{[max]}$ with $\overline \fob \cap \tau \not =
\emptyset$. 

Let $k \in \mathbb{N}$.  Define the order $k$ scattering parameter
ring by
\[
R^{k}:= \kk[t_{ \tau} \, | \,  \tau  \in 
\Delta_\max ] / \mathcal I_{k}, \quad \mathcal I _{k}:=  (t_{
\tau}^{k+1}|  \tau \in  \Delta_\max),
\]
and let $\widehat R$ be its completion as $k\to \infty$. 

As $f_\fob$ has a non-trivial constant term, we can take its logarithm as in
\cite{gps}
\begin{equation}
\label{logslab}
\log f_{\fob} = \sum_{\overline m\in C_{\fob} } 
\length(\overline m) a_{\fob, \overline m} z^{\overline m} \in
\kk\lfor C_\fob\rfor,
\end{equation}
defining virtual multiplicities $a_{\fob, \overline m}\in \kk$. We consider
$\log f_{\fob}$ as an element of $\kk\lfor C_\fob\rfor \otimes _k\widehat R$ via
the completion of the inclusions
\[
\iota_k:  \kk[C_\fob]\to \kk[C_{\fob}]\otimes_\kk R^{k}, \quad
z^{m_{\tau}} \mapsto z^{m_{\tau }} t_{\tau}.
\]

\begin{definition} 
Attach the following numbers to a sufficeintly general virtual tropical disk
$\tilde h \in \M_\mu(\Cover,h)^{gen} $:
\begin{enumerate}
\item
The virtual multiplicity of a vertex $V\in\effvert$ of $\tilde h $ is
\[
\vmult_V(\tilde h):=\left\{ \begin{array}{ll} a_{\fob, \overline m}   
& \text{if $V$ is univalent, $\pi (\tilde h(V)) \in \fob  $}\\
s(\overline m) &  \text{if $V$ is divalent } \\
\big|\overline m \wedge \overline m'\big|_{\foj_V(h)}
& \text{if $V$ is trivalent} \end{array}\right.
\]
where $\overline m$ denotes the tangent vector of $\tilde h$ at $V$ in the
direction leading to the root, $\overline m'$ the tangent vector of a different
edge of $\tilde h$ at $V$, $a_{\fob,\overline m}$ the coefficients in
\eqref{logslab}, $s: \Lambda_{\tilde h(V)}B \to k^\times$ the change of stratum
function at $\tilde h(V)$ coming from the gluing data, and the last expression
the quotient density on $T_{\tilde h(V)}B/\foj_V(\tilde h)$ induced from the
natural density on $B\setminus \Delta$. Explicitly, letting $\overline j_1,
\ldots,\overline j_{n-2}$ be generators of $ \foj_V(\tilde h) \cap
\Lambda_{B,h(V)} $ (cf. Remark~\ref{foj}) then
\[
\big|\overline m \wedge \overline m'\big|_{\foj_V(h)}
:= |\overline m\wedge \overline m' \wedge \overline j_1 \wedge
\ldots \wedge \overline j_{n-2}|.
\]
\item 
The virtual multiplicity $\vmult(\tilde h)$ of $\tilde h$ is the total product
\[
\vmult(\tilde h):=\frac{1}{|\Aut(\mathbf w)|} \cdot
\prod_{V \in \effvert } \vmult_{V}(\tilde h).
\]
Here $\mathbf w$ is the intersection type of $\tilde h$, and $\Aut(\mathbf
w)$ is the product of the automorphisms\footnote{that is, the number of
permutations of the entries of $w_\tau$ that do not change the partition } of
the partitions $w_\tau$ over all $\tau \in \Delta_{[\max]}$.
\item 
The $t$-order of $\tilde h$ is the sum of the changes in the $t$-order of the
tangent vectors $\overline m_V$ at divalent vertices $V$ under changing the
adjacent maximal strata $\sigma^\pm_V$, that is
\[
\ord \tilde h := \sum _{\substack{V\in \Gamma^{[1]}:\\  \tilde h(V)
\in \sigma^+_V\cap \sigma^-_V \in \P^{[n-1]} } } \left|\left<
d\varphi|_{\sigma^-_V} - d\varphi|_{\sigma^+_V}, \overline
m_V\right>\right|.
\]
\end{enumerate}
\end{definition}

\begin{remark}
\label{t-order}
The $t$-order may be considered as a combinatorial analogue of the
symplectic area of a holomorphic disk.
\end{remark}

Note that the virtual multiplicity of a sufficiently general tropical disk
depends only on its type. Moreover, we have:

\begin{lemma}
The virtual multiplicity of a (type of) sufficiently general tropical disk
$\tilde h$ of intersection type $\mathbf w$ deforming a Maslov index zero disk
$h$ is independent of the choice $s\in \defbase$ of the general position
deformation $\Delta_s$.
\end{lemma}

\begin{proof}
We only give a very rough sketch here: If the type only changes by the number of
divalent vertices, the claim follows immediately from the definition of
$\varphi$ as a continuous and piecewise linear function. In dimension two, the
result then reduces to a standard one, cf.~\cite{gat}. In higher dimension, the
only remaining instable \emph{hyper}planes consist of disks with a four-valent
vertex. Here the independence of their stable deformations reduces to the
dimension two case, as the virtual multiplicity is invariant under splitting
each edge of $\Gamma$, acting with the stabilizer $SL(n,\mathbb Z)_{\foj_V(h)}$
on each fan $\Sigma_{h,V}$, and regluing formally.
\end{proof}

We are now ready for our central definitions: Denote by $\# \mathcal M_0(\mathbf
w,\overline m,\ell)^{gen}$ the number of types of sufficiently general virtual
tropical disks with intersection type $\mathbf w$ and $t$-order $\ell$ which
deform a tropical Maslov index zero disk with root tangent vector $\overline
m\in \Lambda_{B\setminus \Delta, x}$, counted with virtual multiplicity. More
generally, for $\mu\in \{0,2\}$ we can define $\# \mathcal M_\mu(\mathbf
w,\overline m,\ell)^{gen} $ by counting the corresponding disks themselves, but
specifying the virtual root as follows: The virtual root is $0\in T_xB$ if $\mu
=2$, and belongs to a line in $T_{h(\rootvertex)}B$ transverse to
$\foj_{\rootvertex}(\tilde h)$ if $\mu = 0$.

Let $\sigma \in \P_\max$, and $ P_{v,\sigma}$ the associated monoid at $v\in
\sigma$ which is determined by $\varphi$ as in
\cite[Constr.\,2.7]{affinecomplex}. Define the \emph{counting function} to
order $k$ in $x\in \sigma$ by
\begin{equation}
\label{countfct}
\log f_{\sigma, x}:= \sum_
{\substack{\overline m \in \Lambda_{x,B}, \\
\ell \leq k} }
\sum_ {\mathbf w}   \length (\overline m)\#\mathcal M_0
(\mathbf w, \overline m, \ell)^{gen}z^{\overline m}t^{\ell}
\prod_{ \tau} t_{ \tau}^{|w_{ \tau}| },
\end{equation}
which is an element of the ring $\kk[P_{v,\sigma}]\otimes_\kk R^k$.

\begin{conjecture}
\label{conject}
For each $k \in \mathbb N$, the counting polynomial \eqref{countfct} modulo
$(t^{k+1})$ stabilizes in $\mathbf w$ and then maps to the rings $\targetring$
via $t_\tau \mapsto 1$. The sets
\begin{equation}
\label{wall}
\fop^k[x]:= \overline{  \set{ y\in \sigma \setminus \partial \sigma }
{\log f_{\sigma,y}  =\log f_{\sigma,x} \not = 0 \in \targetring}},
\end{equation}
are either empty or define polyhedral subsets of codimension at least one. Up to
refinement and adding trivial walls, the $\fop^k[x]$ together with their
functions $\log f_{\sigma,y}$ reproduce the consistent wall structure $\scrS_k$
constructed in \cite{affinecomplex}.
\end{conjecture}

\begin{remark}
Note that general position of $\Delta$ is not essential as long as we obtain the
same virtual counts.
\end{remark}

\begin{proposition}
The conjecture is true for $\dim B = 2$.
\end{proposition}

\begin{proof} (Sketch)
It is sufficient to show the following two claims: \paragraph{\underline{Claim
1}: Over $R^k$, the counting monomials arise via scattering.} This can be proved
by first decomposing $\exp f_{\sigma,x} $ into products of binomials and then
proceeding inductively by applying \cite[Thm.\,2.7]{gps} to each joint.
Alternatively, one can adapt their proof directly:

Consider the thickening
\[
\iota_\tau: \; \kk[t_\tau ]/t^{k+1} \lra \frac{\kk[u_{\tau i} | 1\leq i
\leq k, ]}{(u_{\tau i}^2 | 1\leq i \leq k )  },\quad t_{\tau}
\longmapsto \sum_{i=1}^{k} u_{\tau i}.
\]
inducing a thickening $\bigotimes_{ \tau} \iota_{ \tau}: R^k\to \tilde R^k$ of
the scattering parameter ring. Then consider virtual tropical disks with respect
to the $2^k$-fold branched covering of $\Delta$ whose branches $\tau_s^J$ over
$\tau$ are labeled by $J \subseteq \{1,...,k\}$. Denote $u_{\tau J}:=
\prod_{i\in J}u_{\tau i}$. We say such a disk is special if it has the
following additional properties: The weight of a leaf is $|J|$ if it emanates
from $\tau_s^J$, and $u_{\tau J} u_{\tau J'} \not = 0$ whenever there are leaf
vertices in $\tau_s^J$ and $\tau_s^{J'}$. We can now attach the following
function to the root tangent vector $-\overline m_{\tilde h}$ of such disks:
\begin{equation}
\label{f_h}
f_{\tilde h} :=   1 + \length(m_{\tilde h}) \vmult'(\tilde h) \cdot
z^{\overline m_{\tilde h}}t^{\ord \tilde h}   \hspace{-1.5em}
\prod_{\tau^J_s \cap \tilde h(\vvert) \not = \emptyset} \hspace {-1em}  |J|!
\cdot u_{\tau J}.
\end{equation}
where $\vmult'$ equals $\vmult$ without the combinatorial factor $|\Aut(\mathbf w)|$.

Now the order~0 terms appearing in the thickening of the exponential of
\eqref{countfct} are precisely the $f_{\tilde h}$ of those special disks $\tilde
h$ that contain only one edge. The others indeed arise from scattering, meaning
the following: Whenever the root vertices of two special disks $\tilde h, \tilde
h'$ map to the same point $p$ with transverse root leaves, there at most two
ways to extend both disks beyond $p$ locally: Either glue them to a single
tropical disk, which is possible only if $f_{\tilde h}f_{\tilde h'}\not = 0$, or
enlarge the root leaves such that $p$ stays a point of intersection, which is
always possible.

The functions attached to the two old and the three new roots then define a
consistent scattering diagram, that is the counterclockwise product of the
automorphisms 
\[
z^{\overline m} \mapsto z^{\overline m}
f_{\tilde h}^{|\overline m \wedge \overline m_{\tilde h}|}
\]
equals one. This is the content of \cite[Lem.\,1.9]{gps}, to which we refer for
details. Now the proof of \cite[Thm.\,2.7]{gps} shows that the sum
$\sum_{\tilde h} \log f_{\tilde h}$ over all special disks with $k$-intersection
type $\mathbf w$, root tangent vector $-\overline m$ and $t$-order $\ell$ equals
the thickening of the corresponding monomial in \eqref{countfct}.

\paragraph{\underline{Claim 2}: The counting functions
\eqref{countfct} can be lifted.}

We must show that the scattering diagrams at each joint $\foj_V(\tilde h)$
produce liftable monomials:

In case of a codimension zero joint this follows from the observation that each
incoming non constant ray monomial has $t$-order greater than zero. Hence
working modulo $(t^\ell)$ implies working up to a finite $k$-order. In case of
codimension one joints, by assumption there is only one non constant monomial of
zero $t$-order present in each scattering diagram, namely that given by the log
structure. In this case, we can apply \cite{kansas}. Finally, there are no
codimension two joints by assumption. \smallskip From both claims it follows
that the gluing functions of both constructions must indeed coincide, as by our
assumption on $\Delta$ both rely on scattering only, and scattering is unique up
to equivalence.
\end{proof}


\subsection{Virtual counts of tropical Maslov index two disks}

Assume now $(B,\P,\varphi)$ fulfills Conjecture~\ref{conject}, for example $\dim
B=2$.

\begin{proposition}
\label{Prop: Virtual count}
The coefficient $a_\beta $ of the last monomial $z^{m_\beta}$ of a general
broken line $\beta$ is the virtual number of tropical Maslov index two disks
with root tangent vector $m_\beta$ which complete $\beta$ as in
Lemma~\ref{line-disk} and whose $t$-order equals the total change in the
$t$-order of the exponents along $\beta$.
\end{proposition}

\begin{proof}
Let $az^m, a'z^{m'}$ be the functions attached to the edges adjacent to a fixed
break point $\beta(t)$ of $\beta$. Let $h$ be a virtual Maslov index zero disk
bounded by $\beta(t)$ and with root tangent vector the required difference
$\overline m- \overline m'$. Define the completed multiplicity of $h$ as the
virtual multiplicity of $h$ times that of the break point. By definition, $a'/a$
is a coefficient in the exponential of $a_\fob:=|\overline m \wedge \overline
m'| \log f_{\fob}$ for the function $f_\fob$ belonging to the wall or slab with
tangent space containing $\overline m-\overline m'$. By formula
\eqref{countfct}, $a_\fob$ equals the virtual number of completing virtual
Maslov index zero disks with root tangent leaf $\overline m - \overline m'$ and
$t$-order $\ell$, completed by the break point multiplicity. Hence the required
coefficient in $\exp a_\fob$ is given by counting {\em disconnected} virtual
tropical disks of total $t$-order $\ell$ and total root tangent leaf $\overline
m$ with completed multiplicity.
\end{proof}


\section{Toric degenerations of del Pezzo surfaces and their mirrors}
\label{sect:delpezzo}
In this section we will compare superpotentials for mirrors of toric
degenerations of del Pezzo surfaces, using broken lines and tropical Maslov
index two disks. Recall that apart from $\PP^1 \times \PP^1$ all non-singular
del Pezzo surfaces $dP_k$ can be obtained by blowing up $\PP^2$ in $0 \leq k
\leq 8$ points. Note that $dP_k$ for $k\ge 5$ is not unique up to isomorphism
but has a $2(k-4)$-dimensional moduli space. For the anticanonical bundle to be
ample the blown-up points need to be in sufficiently general position. This
means that no three points are collinear, no six points lie on a conic and no
eight points lie on an irreducible cubic which has a double point at one of the
points. However, rather than ampleness of $-K_X$ the existence of certain toric
degenerations is central to our approach. For example, our point of view
naturally includes the case $k=9$.


\subsection{Toric del Pezzo surfaces}
Up to lattice isomorphism there are exactly five toric del Pezzo surfaces
$X(\Sigma)$ whose fans $\Sigma$ are depicted in Figure~\ref{fig:toric_fans},
namely $\PP^2$ blown up torically in at most three distinct points and $\PP^1
\times \PP^1$. To construct distinguished superpotentials for these surfaces we
consider the following class of toric degenerations. For the definition note
that by the Grothendieck algebraization theorem a toric degeneration
$(\check\X\to T,\check\D)$ with $\check\D$ relatively ample can be algebraized.
By abuse of notation we write $(\check\X_\eta,\check\D_\eta)$ for the generic
fiber of an algebraization, and also simply speak of the \emph{generic fiber of
$(\check\X\to T,\check\D)$}.

\begin{figure}
\input{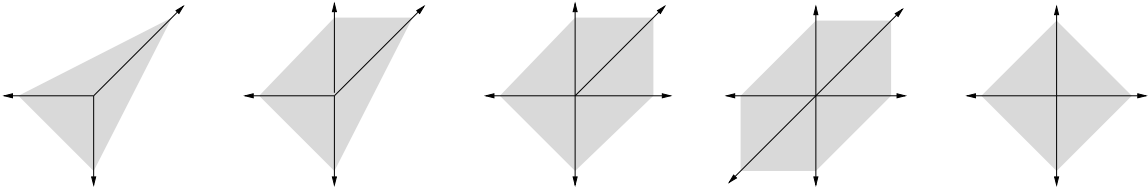}
\caption{Fans of the five toric del Pezzo surfaces}
\label{fig:toric_fans}
\end{figure}
\begin{definition}
\label{Def: toric degen dP_k}
A \emph{distinguished toric degeneration of Fano varieties} is a toric
degeneration $(\check\X \to T, \check\D)$ of compact type (Definition~\ref{Def:
compact type}) with $ \check\D$ relatively ample over $T$ and with generic fiber
$ \check\D_{\eta}\subseteq \check \X_{\eta}$ of an algebraization an
anticanonical divisor in a Gorenstein surface.

The \emph{associated intersection complex} $( \check B, \check\P, \check\varphi)$ is the intersection complex of $( \check\X \to T,  \check\D)$ polarized by $\O_{ \check\X}( \check\D)$.
\end{definition}

The point of this definition is both the irreducibility of the anticanonical
divisor and the fact that this divisor extends to a polarization on the central
fiber.

If the generic fiber $\check\X_\eta$ is a surface then it is a $dP_k$ for some
$k$, together with a smooth anticanonical divisor. 

Starting from a reflexive polytope, there is a canonical construction of the
intersection complex of a distinguished toric degeneration as follows.

\begin{construction}
\label{Const: Reflexive polytope degeneration}
Let $\Xi$ be a reflexive polytope and $v_0\in \Xi$ the unique interior integral
vertex. Define the polyhedral decomposition $\check \P$ of $\check B=\Xi$ with
maximal cells the convex hulls of the facets of $\Xi$ and $v_0$. The affine
chart at $v_0$ is the one defined by the affine structure of $\Xi$. At any other
vertex define the affine structure by the unique chart compatible with the
affine structure of the adjacent maximal cells and making $\partial \check B$
totally geodesic. This works because by reflexivity for any vertex $v$ the
integral tangent vectors of any adjacent facet together with $v-v_0$ generate
the full lattice. Moreover, $(\check B,\check \P)$ has a natural polarization
by defining $\check\varphi(v_0)=0$ and $\check\varphi(v)=1$ for
each other vertex $v$.

The tropical manifold $(\check B,\check \P)$ obtained in this way does not
generally have simple singularities. Using a similar computation from \cite{GBB}
on the Legendre dual side one can show, however, that $(\check B,\check\P)$ has
simple singularities iff the normal fan $\Sigma$ of $\Xi$ is elementary
simplicial, meaning that each cone is the cone over an elementary simplex. In
dimensions two and three this is equivalent to requiring the toric variety
$\check X(\Xi)$ with momentum polytope $\Xi$ to be smooth. See
\cite[Constr.\,5.2]{MaxThesis} for more details.

Assuming $(\check B,\check \P)$ has locally rigid singularities (e.g.\ simple
singularities or in dimension two) so we can run \cite{affinecomplex}, or there
is a consistent compatible sequence of wall structures on $(\check B,\check
\P,\check\varphi)$ by other means, we obtain a distinguished, anticanonically
polarized toric degeneration $\check\X\to\Spf\kk\lfor t\rfor$ together with an
irreducible anticanonical divisor $\check\D\subset\check\X$.

The generic fiber $\check\X_\eta$ of this toric degeneration may not be
isomorphic to the toric variety $\check X(\Xi)$. But by introducing an
additional parameter $s$ scaling the non-constant coefficients of the slab
functions one can produce a two-parameter family with $\check\X\to\Spf\kk\lfor
t\rfor$ the restriction to $s=1$ and the constant family with fiber $\check
X(\Xi)$ for $t\neq0$ the restriction to $s=0$.

The discrete Legendre transform $(B,\P,\varphi)$ has a unique bounded cell
$\sigma_0$, isomorphic to the dual polytope of $\Xi$. Up to the addition of a
global affine function, the dual polarizing function $\varphi$ is the
unique piecewise affine function changing slope by one along the unbounded
facets.\footnote{In the present case $\check\varphi$ is single-valued.}
\end{construction}

\begin{remark}
Alternatively, one can use an MPCP resolution \cite[Thm.\,2.2.24]{Bat} of
$\check X(\Xi)$ defined by a simplicial subdivision of $\Sigma$ to split the
discriminant locus of $(\check B,\check\P$ into simple singularities. On the
Legendre-dual side the subdivision is given by writing the bounded maximal cell
$\sigma_0$ as a union of elementary simplices. The resolution process leads to
the introduction of more K\"ahler parameters on the log Calabi-Yau side, hence
more complex parameters on the Landau-Ginzburg side, reflected in the choice of
$\varphi$. See \cite{MaxThesis} for some discussions in this direction.
\end{remark}

\begin{example}
Specializing to del Pezzo surfaces, we start from the momentum polytopes of the
five non-singular toric Fano surfaces. The result of the construction is
depicted in Figure~\ref{fig: toric_del_pezzo}, which shows a chart in the
complement of the dotted segments. 
\begin{figure}
\input{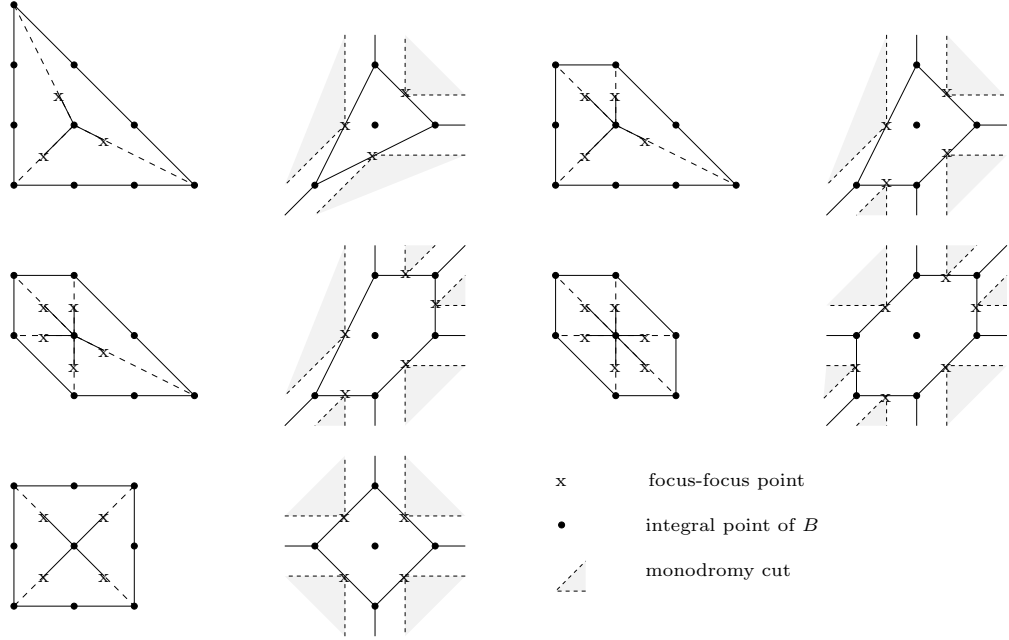}
\caption{The intersection complexes $(\check B,\check\P)$ of the five
distinguished toric degenerations of del Pezzo surfaces with simple
singularities and their Legendre duals $(B,\P)$.}
\label{fig: toric_del_pezzo}
\end{figure}
Note that the discrete Legendre transform $(B,\P,\varphi)$, also depicted in
Figure~\ref{fig: toric_del_pezzo}, indeed has parallel outgoing rays.
\end{example}

Conversely, in dimension~$2$ we have the following uniqueness result.

\begin{theorem}
\label{Thm: unique}
Let $(\pi:\check \X\to T,\check \D)$ be a distinguished toric degeneration of
del Pezzo surfaces. Then the associated intersection complex $(\check B,
\check\P)$ is a subdivision of the star subdivision of a reflexive polygon
$\Xi$, with the edges containing a singular point precisely those connecting the
interior integral point to a vertex of $\Xi$.

If furthermore the toric degeneration has simple singularities
\cite[\S1.5]{logmirror}, then $(\check B,\check\P)$ is isomorphic to an integral
subdivision of one of the cases listed in Figure~\ref{fig: toric_del_pezzo}.
\end{theorem}

\begin{proof}
Let $(\pi:\check \X \to T,\check\D)$ be the given toric degeneration. Thus the
generic fiber $\check \X_\eta$ is isomorphic to a del Pezzo surface $dP_k$ over
$\eta$ for some $0\le k\le 8$, or to $\PP^1\times\PP^1$. By assumption
$\partial\check B$ is locally straight in the affine structure.

First we determine the number of integral points of $\check B$. Let
$\shL=\O_{\check\X}(\check\D)$ be the polarizing line bundle on $\check\X$. By
assumption
\begin{equation}
\label{H^0 dP_k}
h^0(\check \X_\eta,\shL|_{\check\X_\eta})=
h^0(dP_k,-K_{dP_k})=
\begin{cases}
10-k,&\check\X_\eta\simeq dP_k\\
9,&\check\X_\eta\simeq \PP^1\times\PP^1.
\end{cases}
\end{equation}
Let $t\in \O_{T,0}$ be a uniformizing parameter and $\check X_n:= \Spec \big(
\kk[t]/(t^{n+1})\big) \times_T \check\X$ the $n$-th order neighborhood of
$\check X_0:= \pi^{-1}(0)$ in $\check \X$. Denote by $\shL_n=\shL|_{\check
X_n}$.

Then for any $n$ there is an exact sequence of sheaves on $\check X_0$,
\[
0\lra \O_{\check X_0}\lra \shL_{n+1}\lra \shL_n\lra 0.
\]
By the analogue of \cite[Thm.\,4.2]{PartII} for log Calabi-Yau varieties, we
know
\[
h^1(\check X_0, \O_{\check X_0}) =
h^1(\check \X_\eta, \O_{\check\X_\eta})= 0.
\]
Thus the long exact cohomology sequence induces a surjection $H^0(\check
X_0,\shL_{n+1}) \twoheadrightarrow H^0(\check X_0,\shL_n)$ for each $n$. By the
theorem on formal functions and cohomology and base change \cite[Thms.\,11.1 \&
12.11]{Hartshorne}, we thus conclude that $\pi_*\shL$ is locally free, with
fiber over $0$ isomorphic to $H^0(\check X_0,\shL_0)$. In view of \eqref{H^0
dP_k} we thus conclude
\[
h^0(\check X_0,\shL_0)=
\begin{cases}
10-k,&\check\X_\eta\simeq dP_k\\
9,&\check\X_\eta\simeq \PP^1\times\PP^1.
\end{cases}
\]
Now on a toric variety the dimension of the space of sections of a polarizing
line bundle equals the number of integral points of the momentum polytope. Since
$\check X_0$ is a union of toric varieties, each integral point $x\in \check B$
provides a monomial section of $\shL_0$ on any irreducible component $\check
X_\sigma\subseteq \check X_0$ with $\sigma\in \P$ containing $x$. These provide
a basis of sections of $H^0 (\check X_0,\shL_0)$.\footnote{This also follows by
the description of $(\check X_0,\shL_0)$ by a homogeneous coordinate ring in
\cite[Def.\,2.4]{logmirror}.} Hence $\check B$ has $10-k$ integral points.

An analogous argument shows that the number of integral points of $\partial
\check B$ equals
\[
h^0(\check D_0,\shL_0)= h^0(\check\D_\eta,\shL_\eta),
\]
which by Riemann-Roch equals $K_{dP_k}^2 = 9 -k$ or $K_{\PP^1\times\PP^1}^2=8$.
In either case we thus have a unique integral interior point $v_0\in \check B$.
In particular, $\check B$ has the topology of a disk, and each interior
edge connects $v_0$ to an integral point of $\partial \check B$.

Viewed in the chart at the interior integral point, $\check B$ is therefore a
reflexive polygon $\Xi$. Moreover, since $\partial\check B$ is locally straight,
each of the interior edges with endpoint a vertex of $\Xi$ has to contain a
singular point. None of the other interior edges can contain a singular point
for otherwise $\partial\check B$ would not be straight in the affine structure
(and $\check B$ would not even be locally convex on the boundary).

If the singularities are simple, one finds that $\partial\check B$ is only
locally straight in the five cases shown in Figure~\ref{fig: toric_del_pezzo},
up to adding some edges connecting $v_0$ to $\partial \check B$ without singular
point. The remaining 11~cases are discussed in \S\ref{Subsect: Singular Fano
surfaces} below.
\end{proof}

\begin{remark}
1)\ The proof shows that the five types can be distinguished by $\dim H^0(\check
\X_\eta, \shL_\eta)$, except for $\PP^1\times\PP^1$ and $dP_1$. Alternatively,
by Proposition~\ref{Prop: Number of singular points} below one could use
$H^1(\check \X_\eta, \Omega^1_{\check\X_\eta})$.\\[1ex]
2)\ For each $(\check B,\check \P)$ there is a discrete set of choices of $\check \varphi$, which determines the local toric models of $\check\X \to\check\D$. This reflects the fact that the base of (log smooth) deformations of the central fiber $\check X_0^\ls$ as a log space over the standard log point $\kk^\ls$ is higher dimensional. In fact, let $r$ be the number of vertices on $\partial \check B$. Then taking a representative of $\check \varphi$ that vanishes on one maximal cell, $\check\varphi$ is defined by the value at $r-2$ vertices on $\partial \check B$. Convexity then defines a submonoid $Q\subseteq \NN^{r-2}$ with the property that $\Hom(Q,\NN)$ is isomorphic to the space of (not necessarily strictly) convex, piecewise affine functions on $(\check B,\check \P)$ modulo global affine functions. Running the construction of \cite{affinecomplex} with parameters then produces a log smooth deformation over the completion at the origin of $\Spec \kk[Q]$ with central fiber $(\check X_0,\check D_0)$. For the minimal polyhedral decompositions of Figure~\ref{fig: toric_del_pezzo} with $r=l$ the number of singular point we have $\rk Q=l-2$, which by Remark~\ref{Rem: H^1(B,Lambda)},2 below agrees with the dimension of the space $H^1(\check X_0, \Theta_{\check X_0^\ls/\kk^\ls})$ of infinitesimal log smooth deformations of $X_0^\ls/ \kk^\ls$. One can show that in this case the constructed deformation is in fact semi-universal.\footnote{\cite[Thm.\,4.4]{RS} proves semi-universality for all simple toric degenerations of compact type.}
\qed
\end{remark}

The technical tool to compute superpotentials in the toric del Pezzo cases and
in related examples in finitely many steps is the following lemma, suggested to
us by Mark Gross. It greatly reduces the number of broken lines to be
considered, especially in asymptotically cylindrical situations and with a
finite structure on the bounded cells. 

\begin{lemma}
\label{Lem: Broken lines are rays}
Let $\scrS$ be a structure for a non-compact, polarized tropical manifold
$(B,\P,\varphi)$ that is consistent to all orders. We assume that there is a
subdivision $\P'$ of a subcomplex of $\P$ with vertices disjoint from $\Delta$
and with the following properties.
\begin{enumerate}
\item
Each $\sigma\in \P'$ is affine isomorphic to $\rho\times\RR_{\ge0}$ for some
bounded face $\rho\subseteq \sigma$.
\item
$B\setminus \Int(|\P'|)$ is compact and locally convex at the vertices (this
makes sense in an affine chart).
\item
If $m$ is an exponent of a monomial of a wall or unbounded slab intersecting
some $\sigma\in\P'$, $\sigma=\rho+\RR_{\ge0}\overline m_{\sigma}$, then
$-\overline m\in\Lambda_\rho+\RR_{> 0}\cdot \overline m_{\sigma}$.
\end{enumerate}

Then the first break point $t_1$ of a broken line $\beta$ with
$\im(\beta)\not\subseteq |\P'|$ can only happen after leaving $\Int|\P'|$, that
is,
\[
t_1\ge \inf\big\{t\in (-\infty,0]\,\big|\, \beta(t)\not\in |\P'|\big\}.
\]
\end{lemma}

\begin{proof}
Assume $\beta(t_1)\in \sigma\setminus\rho$ for some $\sigma=
\rho+\RR_{\ge0}\overline m_{\sigma} \in\P'$. Then $\beta|_{(-\infty,t_1]}$ is an
affine map with derivative $-\overline m_{\sigma}$, and $\beta(t_1)$ lies on a
wall. By the assumption on exponents of walls on $\sigma$, the result of
nontrivial scattering at time $t_1$ only leads to exponents $m_2$ with
$-\overline m_2\in \Lambda_\rho+\RR_{\ge0} \overline m_{\sigma}$, the outward
pointing half-space. In particular, the next break point can not lie on $\rho$.
Going by induction one sees that any further break point in $\sigma$ preserves
the condition that $\beta'$ does not point inward. Moreover, by the convexity
assumption, this condition is also preserved when moving to a neighboring cell
in $\P'$. Thus $\im(\beta)\subseteq |\P'|$.
\end{proof}

\begin{proposition}
\label{Prop: Simple Hori-Vafa cases}
Let $(B,\P)$ be the dual intersection complex of a distinguished toric
degeneration of del Pezzo surfaces with simple singularities and let
$\sigma_0\subseteq B$ be the bounded cell. Then there is a neighborhood $U$ of
the interior vertex $v_0\in \sigma_0$ such that for any $p\in U$ there is a
canonical bijection between broken lines with endpoint $p$ and rays of $\Sigma$,
the fan over the proper faces of $\sigma_0$.
\end{proposition}

\begin{proof}
We can embed $\Sigma$ in the tangent space at $v_0$ by extending the unbounded
edges to $v_0$ in the chart shown in Figure~\ref{fig:toric_fans}. Each ray of
$\Sigma$ can then be interpreted as the image of a unique degenerate broken
line. Because each such degenerate broken line has a positive distance from the
shaded regions in Figure~\ref{fig: toric_del_pezzo} they can be moved with small
perturbations of $p$ to deform to a proper broken line. Conversely, by
inspection of the five cases, the result of non-trivial scattering at
$\partial\sigma_0$ leads to a broken line not entering $\Int(\sigma_0)$.
There are no walls entering $\Int\sigma_0$, so by Lemma~\ref{Lem: Broken
lines are rays} any such broken line can bend at most at the intersection with
$\partial\sigma_0$.
\end{proof}

\begin{corollary}
\label{Cor: reproduce Hori-Vafa}
Let $(\X\to\Spf\kk\lfor t\rfor,W)$ be the Landau-Ginzburg model mirror to a
distinguished toric degeneration of del Pezzo surfaces with simple
singularities. Then there is an open subset $U\simeq \Spf \kk[x^{\pm1},y^{\pm1}]
\lfor t\rfor \subseteq \X$ such that $W|_U$ equals the usual Hori-Vafa
monomial sum times $t$.
\qed
\end{corollary}

\begin{remark}
1)\ For other than the anticanonical polarization of the del Pezzo surface, the
terms in the superpotential receive different powers of $t$, just as in the
Hori-Vafa proposal.\\[1ex]
2)\ Analogous arguments work for Landau-Ginzburg mirrors of smooth toric Fano
varieties of any dimension \cite[Thm.\,5.4]{MaxThesis}.
\end{remark}

\begin{figure}
\input{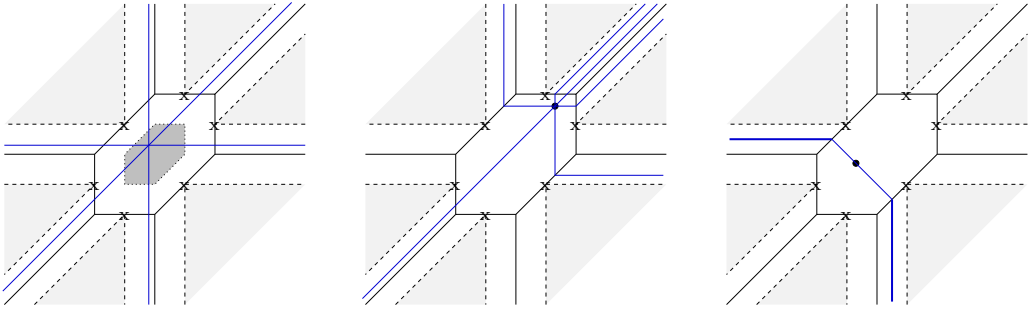}
\caption{Tropical disks in the mirror of the distinguished base of
$dP_3$ indicating the invariance under change of root vertex.}
\label{fig:example_dp3}
\end{figure}

\begin{example} \label{ex: dP_3}
Let us study the mirror of the distinguished toric degeneration of $dP_3$ with
the minimal polarization $\check\varphi$ explicitly. In
Figure~\ref{fig:example_dp3} the first two pictures show all Maslov index two
tropical disks, respectively broken lines using disk completion
(Lemma~\ref{line-disk}), for two choices of root vertex. Moving the root
vertex within the shaded open hexagon yields the same result, that is, none of
the six broken lines has a break point. In contrast, moving the root vertex
inside $\sigma_0$ out of the shaded hexagon leads to some bent broken lines,
but the set of root tangent vectors always remains $(1,0),(1,1),(0,1),
(-1,0),(-1,-1)$ and $(0,-1)$, all with coefficient~$1$. This illustrates the
invariance under the change of root vertex proved in Lemma~\ref{lem:independence
of p}.

The potential in the chart for the bounded cell $\sigma_0$ is thus computed as
\[
W_{dP_3}(\sigma_0) = \Big(x + y + xy + \frac{1}{x} + \frac{1}{y} +
\frac{1}{xy}\Big) \cdot t,
\]
which for $t\neq 0$ has six critical points.

The picture on the right shows two tropical disks with weight two unbounded
leaves. These do not contribute to the superpotential.

An analogous picture arises for the other four distinguished del
Pezzo degenerations.
\end{example}

Morally speaking the last example shows that in toric situations ray
generators of the fan are sufficient to compute the superpotential,
but really they should be seen as special cases of tropical disks
or broken lines.


\subsection{Non-toric del Pezzo surfaces}
In this section we consider del Pezzo surfaces $dP_k$ for $k\ge
4$, referred to as higher del Pezzo surfaces. Let us first determine
the topology of $B$ and the number of singular points of the
affine structure.

We need the following statement for Fano varieties with smooth anticanonical divisor

\begin{lemma}
\label{Lem: No log holomorphic 1-form}
Let $X$ be a smooth Fano variety and $D\subset X$ a smooth
anticanonical divisor. Then $H^0(X,\Omega_X(\log D))=0$.
\end{lemma}

\begin{proof}
It is a folklore result in Hodge theory that the connecting homomorphism of the residue sequence
\[
0\lra \Omega_X\lra \Omega_X(\log D)\stackrel{\text{res}}{\lra} \O_D\lra 0
\]
maps $1\in H^0(D,\O_D)$ to $c_1(\O_X(D))\in H^{1,1}(X)= H^1(X,\Omega_X)$. By
ampleness of $D$ this class is nonzero. Thus $H^0(X,\Omega_X(\log D))\simeq
H^0(X,\Omega_X)$.

The claim now follows from the Kodaira vanishing theorem:
\[
H^0(X,\Omega_X)\simeq H^n(X,\O_X)
= H^n(X,K_X\otimes K_X^{-1})=0.
\qedhere
\]
\end{proof}
\smallskip

\begin{proposition}
\label{Prop: Number of singular points}
Let $(B,\P)$ be the dual intersection complex of a distinguished toric
degeneration $(\pi:\check\X\to T, \check\D)$ of del Pezzo surfaces with simple
singularities (Definition~\ref{Def: toric degen dP_k}). In particular, we assume
the generic fiber $\check\X_\eta$ is a proper surface with $\check\D_\eta$ a
smooth anticanonical divisor.

Then $B$ is homeomorphic to $\RR^2$, and the affine structure has $l=\dim
H^1(\check\X_\eta, \Omega^1_{\check\X_\eta})+2$ singular points.
\end{proposition}

\begin{proof}
Since the relative logarithmic dualizing sheaf $\omega_{\check\X/ \check\D}
(-\log \check\D)$ is trivial, the generalization of \cite[Thm.\,2.39]{logmirror}
to the case of log Calabi-Yau varieties shows that $B$ is orientable. By the
classification of surfaces with effective anticanonical divisor we know
$H^i(\check\X_\eta, \O_{\check\X_\eta})=0$, $i=1,2$. As in the proof of
Theorem~\ref{Thm: unique} this implies $H^1(\check X_0,\O_{\check X_0})=0$. Thus
by the log Calabi-Yau analogue of \cite[Prop.\,2.37]{logmirror},
\[
H^1(B,\kk)= H^1(\check X_0,\O_{\check X_0})=0.
\]
In particular, $B$ has the topology of $\RR^2$.

As for the number of singular points, the generalization of \cite[Thms.\,3.21 \&
4.2]{PartII} to the case of log Calabi-Yau varieties \cite[Thm.\,3.5 \&
Thm.\,3.9]{Tsoi} shows that $\dim H^1(\check\X_\eta, \Omega^1_{\check \X_\eta})$
is related to an affine Hodge group:\footnote{The proof of
\cite[Thm.\,4.2]{PartII} has a gap fixed in \cite[Thm.\,1.10]{FFR}.}
\begin{equation}
\label{Eqn: H^1 affine}
\dim_{\kk((t))} H^1(\check\X_\eta, \Omega^1_{\check\X_\eta})
=\dim_\RR H^1(B,i_*\check\Lambda \otimes_{\ZZ} \RR).
\end{equation}
To compute $H^1(B,i_*\check\Lambda \otimes_{\ZZ} \RR)$ we choose the following
\v Cech cover of $B=\RR^2$. Let $\ell_1,\ldots,\ell_l$ be disjoint real
half-lines emanating from the singular points $p_1,\ldots,p_l$. Define
$U_0=\RR^2\setminus \bigcup_{i=1}^l \ell_i$, and $U_i=B_\eps(p_i)$ for $\eps$
sufficiently small to achieve $U_i\cap \ell_j=\emptyset$ unless $i=j$. Then
$\foU:= \big\{U_0,U_1,\ldots,U_l\big\}$ is a Leray covering of $B$ for
$i_*\check\Lambda_\RR:= i_*\check\Lambda\otimes_\ZZ \RR$, cf.\
\cite[Lem.\,5.5]{logmirror}. The terms in the \v Cech complex are
\[
C^0(\foU,i_*\check\Lambda_\RR)  =  \RR^2 \times 
\prod_{i=1}^{l} \RR,\quad
C^1(\foU,i_*\check\Lambda_\RR) = \prod_{i=1}^{l} \RR^2,\quad
C^k(\foU,i_*\check\Lambda_\RR) = 0 \text{ for $k\ge 2$.}
\]
The analogue of \eqref{Eqn: H^1 affine} for degree~0 cohomology groups shows
that the kernel of the \v Cech differential $C^0(\foU,i_*\check\Lambda_\RR) \to
C^1 (\foU,i_*\check\Lambda_\RR)$ computes $H^0(\X_\eta,\Omega_{\X_\eta}(\log
\D_\eta))$. This latter group vanishes by Lemma~\ref{Lem: No log holomorphic
1-form}. Hence
\[
\dim H^1(B,i_*\check\Lambda_\RR)= 2l - (l + 2) = l-2
\]
determines the number $l$ of focus-focus points as claimed.
\end{proof}

\begin{remark}
\label{Rem: H^1(B,Lambda)}
1)\ From the analysis in Proposition~\ref{Thm: unique} and
Proposition~\ref{Prop: Number of singular points} it is clear that for del Pezzo
surfaces $dP_k$ with $k\ge4$ the anticanonical polarization is too small to
extend over a toric degeneration with simple singularities. The associated
tropical manifold would simply not have enough integral points to admit the
required number of singular points.\\[1ex]
2)\ Essentially the same argument also computes the dimension of the
space of infinitesimal deformations:
\[
h^1(\check \X_\eta, \Theta_{\check\X_\eta}(\log \check\D_\eta))=
h^1(\check X_0,\Theta_{\check X_0^\ls/\kk^\ls})=
h^1(B, i_*\Lambda_\RR)= l-2.
\]
\vspace{-7ex}

\qed
\end{remark}

It is easy to write down toric degenerations of non-toric del Pezzo
surfaces, since they can be represented as hypersurfaces or complete
intersections in weighted projective spaces, as for example done for
$dP_6$ in~\cite[Expl.\,4.4]{Invitation}. The most natural toric
degenerations in this setup have central fiber part of the toric
boundary divisor of the ambient space. But because this construction
gives nodal $\D_\eta$ such toric degenerations are never
distinguished. To obtain proper superpotentials we therefore need a
different approach.

\begin{construction}
\label{con: nontoric} 
Start from the intersection complex $(\check B, \check\P)$ of the distinguished
toric degeneration of $dP_3$ depicted as a hexagon in Figure~\ref{fig:
toric_del_pezzo}. The six focus-focus points in the interior of the bounded two
cell make the boundary $\rho$ straight. There is no space to introduce more
singular points of the affine structure because all interior edges already
contain a singular point. To get around this, polarize by $-2 \cdot K_{dP_3}$
and adapt $\P$ in the obvious way, see Figure~\ref{fig:nontoric_del_pezzo}. This
scales the affine manifold $B$ by two, but keeps the singular points fixed. The
new boundary now has $12$ integral points and the union $\gamma$ of edges
neither intersecting the central vertex nor $\partial B$ is a geodesic. We can
then introduce new singular points on the boundary of the interior hexagon as
visualized in Figure~\ref{fig:nontoric_del_pezzo}. Moreover, let $\check
\varphi$ be unchanged on the interior cells and change slope by one when passing
to a cell intersecting $\partial B$. Plugging in up to five singular points,
the Hodge numbers from Proposition~\ref{Prop: Number of singular points}
show that the toric degenerations obtained from the tropical data are in
fact toric degenerations of $dP_k$, $4\le k\le 8$.
\qed
\end{construction}

\begin{figure}
\input{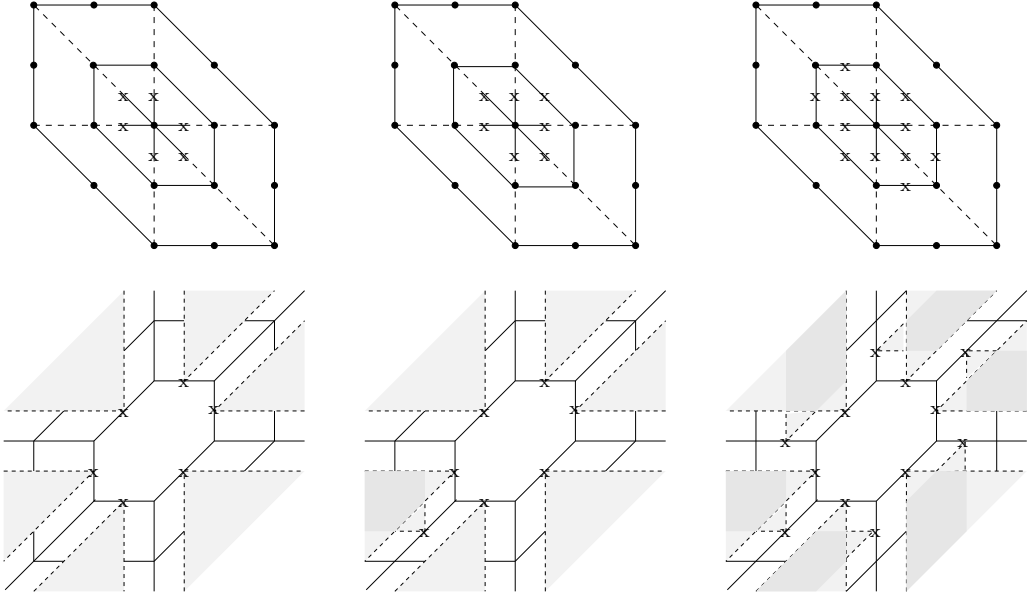}
\caption{Straight boundary models for higher del Pezzo surfaces
obtained by changing affine data for $dP_3$ and their
Legendre duals.} 
\label{fig:nontoric_del_pezzo}
\end{figure}

Unlike in the anticanonically polarized case, the models constructed in this
way are not unique. The geodesic $\gamma$ is divided into six segments by $\P$,
and the choice on which of these segments we place the singular points, modulo
the $\ZZ/6$-rotational symmetry, results in non-isomorphic models. We will see
in Example~\ref{ex: dp5} how this choice influences tropical curve counts.

Although there are other ways to define distinguished models for higher del
Pezzo surfaces, for example by choosing another polarization, in this way we can
extend the unique toric models most easily, since all tropical disks and
broken lines we studied before arise in these models without any change.

\begin{figure} 
\input{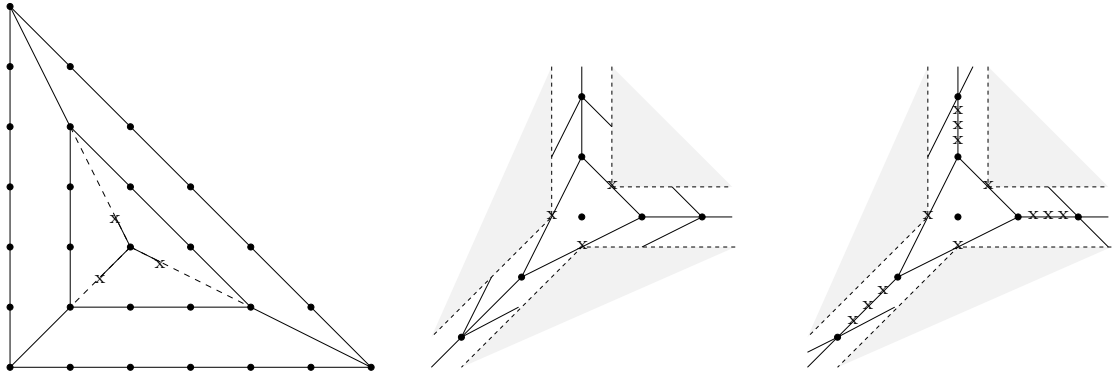}
\caption{An alternative base for higher del Pezzo surfaces and their
mirror.}
\label{fig: AKO}
\end{figure}

\begin{remark}
Note that introducing six new points, for instance as in the rightmost picture
in Figure~\ref{fig:nontoric_del_pezzo}, corresponds to a blow up of $\PP^2$ in
nine points, which is not Fano anymore, but from our point of view still has a
Landau-Ginzburg mirror.

From a different point of view this has already been
noted in~\cite{AKO}, where the authors construct a compactification of the
Hori-Vafa mirror as a symplectic Lefschetz fibration as follows. Start with the
standard potential $x+y+\frac{1}{xy}$ for $\PP^2$ and compactify by a divisor at
infinity consisting of nine rational curves. Then by a deformation argument it
is possible to push $k$ of those rational curves to the finite part and
decompactify to obtain a potential for $dP_k$, including $k=9$.

We can reproduce this result from our point of view by starting with $\PP^2$
rather than with $dP_3$, as illustrated in Figure~\ref{fig: AKO}. Moving
rational curves from infinity to the finite part is analogous to introducing new
focus-focus points. In the present case one may put three focus-focus points on
each unbounded ray of $(B,\P)$ until the respective Legendre dual vertex becomes
straight. Figure~\ref{fig: AKO} on the right shows nine such points
(corresponding to the case $k=9$ above), and any additional singular point would
result in a concave boundary. This can be seen as an affine-geometrical
explanation for why the compactification constructed by the authors
in~\cite{AKO} has exactly nine irreducible components.

Note that it is possible to introduce more singular points when passing to
larger polarizations, but in this way we will not end up with degenerations of
del Pezzo surfaces.
\end{remark}

In order to determine the superpotential, we depicted in
Figure~\ref{fig:nontoric_del_pezzo} an appropriate chart of the relevant
$(B,\P)$. When two regions to be removed overlap we shade them darker to
indicate the non-trivial transformation there.

\begin{figure}
\input{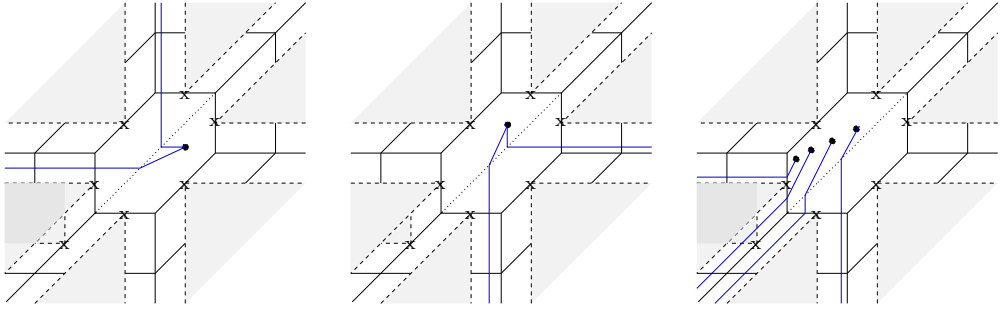}
\caption{Mirror base to $dP_4$ showing new broken lines contributing
to $W_{dP_4}$, indicating the wall crossing phenomenon and
the invariance under change of endpoint within a chamber.}
\label{fig:higher_dp4}
\end{figure}

\begin{example}
\label{ex: dP4}
Figure~\ref{fig:higher_dp4} shows the dual intersection complex $(B,\P)$ of a
toric degeneration of $dP_4$ from Construction~\ref{con: nontoric}. The
additional focus-focus point changes the structure $\scrS$ and allows broken
lines to scatter with the wall in direction $(1,1)$ in the central cell
$\sigma_0$ which subdivides $\sigma_0$ into two chambers $\fou,\fou'$ and
yields new root tangent directions. A broken line coming from infinity in
direction $(\pm 1,0)$ produces root tangent vectors $(-1\pm 1,-1)$, whereas
one with direction $(0,\pm 1)$ takes directions $(-1,-1\pm 1)$. By
construction, every broken line reaching the interior cell $\sigma_0$ has
$t$-order at least $2$.

Note also that by \cite[Expl.\,4.3]{Invitation} only the wall indicated by a
dotted line in Figure~\ref{fig:higher_dp4} enters $\sigma_0$. Thus $\sigma_0$ is
subdivided into two chambers $\fou,\fou'$, in all orders. A broken line can have
at most one break point within $\sigma_0$, in which case the $t$-order increases
by one. So let us compute $W_{dP_4}^3$. We get two new root tangent directions,
namely $(-1,-2)$ and $(-2,-1)$, and possibly more contributions from directions
$(0,-1)$ and $(-1,0)$. The two leftmost pictures in Figure~\ref{fig:higher_dp4}
show all new broken lines for different choices of root vertex, apart from the
six toric ones we have already encountered in Example~\ref{ex: dP_3}. Depending
on this choice, the superpotential to order three is therefore either given by 
\begin{align*}
W_{dP_4}(\fou) &=  \Big( x +  y + xy + \frac{1}{x} +
\frac{1}{y} + \frac{1}{xy}\Big) \cdot t^2 +
\Big(\frac{1}{x} + \frac{1}{x^2y}\Big) \cdot t^3 \quad\text{or } \\
W_{dP_4}(\fou') &= \Big(x + y + xy + \frac{1}{x} +
\frac{1}{y} + \frac{1}{xy}\Big) \cdot t^2  +
\Big(\frac{1}{y} + \frac{1}{xy^2}\Big) \cdot t^3 ,
\end{align*}
both of which have seven critical points, as expected. These superpotentials are
not only related by interchanging $x$ and $y$, for symmetry reasons, but also by
wall crossing along the wall separating $\sigma_0$ into two chambers. This is
the first non-trivial example of an ostensibly algebraic superpotential, as
defined at the end of \S\ref{Ch: Broken lines}.

In the rightmost picture in Figure~\ref{fig:higher_dp4} we indicated the
behaviour of a single broken line of root tangent direction $(-1,-2)$ under
change of root vertex. If the root vertex changes a chamber by passing one of
the dotted lines drawn, the broken lines change accordingly. 
\end{example}

\begin{figure}
\input{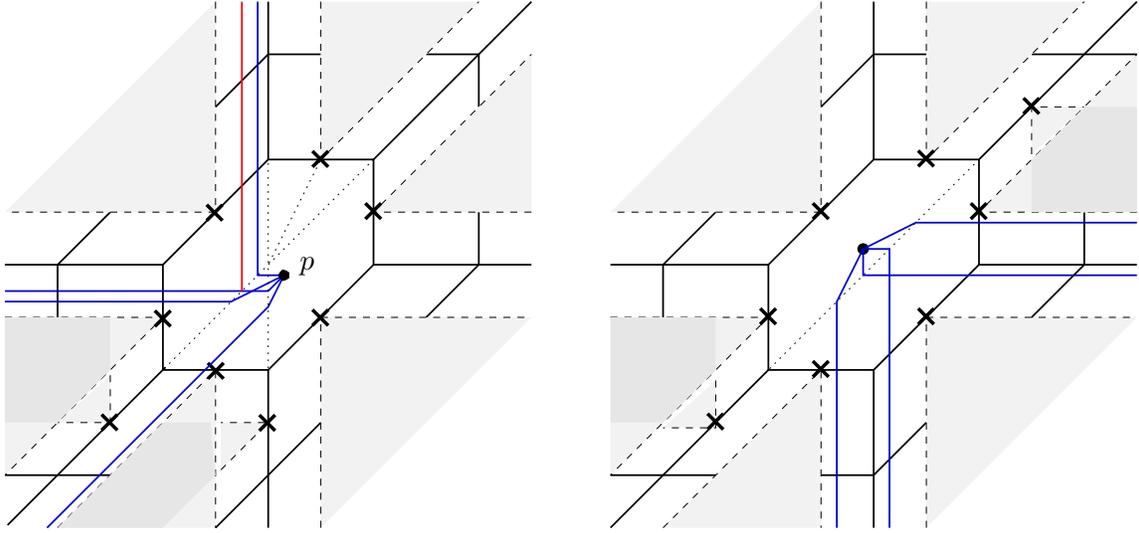}
\caption{Mirror bases to $dP_5$ showing all new broken lines.}
\label{fig:higher_dp5}
\end{figure}

\begin{example} \label{ex: dp5}
Attaching another singular point on the unbounded ray in direction $(0,-1)$ as
in Figure~\ref{fig:higher_dp5} on the left we arrive at a degeneration of
$dP_5$. For a structure consistent to all orders there are three walls in the
bounded maximal cell necessary, indicated by dotted lines in the figure. They
are the extensions of the slabs with tangent directions $(1,1)$ and $(0,1)$
caused by additional singular points, and the result of scattering of these, the
wall with tangent direction $(1,2)$. Because $(1,1)$ and $(0,1)$ form a lattice
basis, the scattering procedure at the origin does not produce any additional
walls. In any case, any broken line coming in from direction $(1,1)$ and with
endpoint $p$ as indicated in Figure~\ref{fig:higher_dp5} can not interact with
any of the scattering products. Tracing any possible broken lines starting from
$t=-\infty$ one arrives at only five broken lines with endpoint $p$, with only
one, drawn in red, having more than one breakpoint. We therefore obtain the
following superpotential on the chamber $\fou$ containing $p$:
\[
 W_{dP_5}(\fou) = \Big(x +  y +  xy + \frac{1}{x} + \frac{1}{y} +
\frac{1}{xy}\Big) \cdot t^2  +  \Big( \frac{1}{y} +  \frac{1}{xy} +
\frac{1}{x^2y} + \frac{1}{xy^2}\Big)
\cdot t^3 +  \frac{1}{xy}t^4.
\]
\end{example}

\begin{figure}
\input{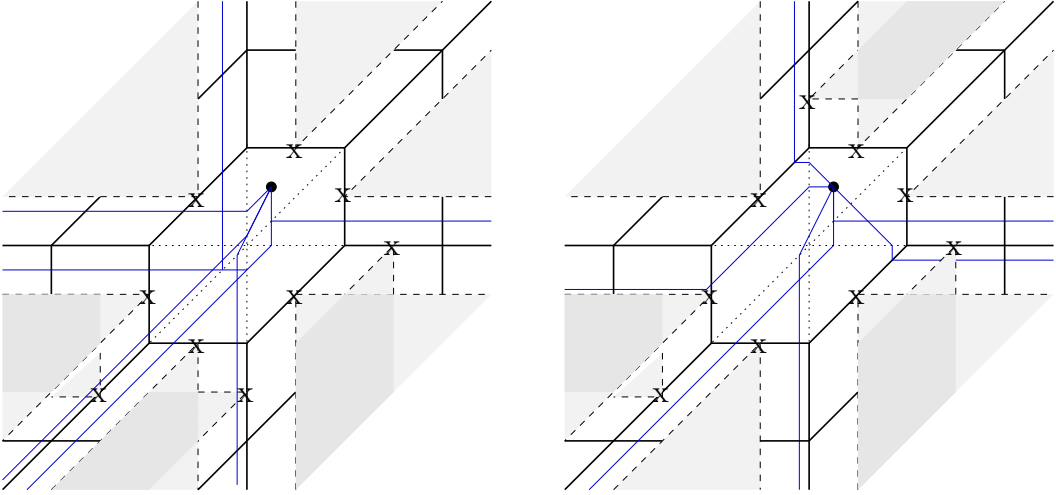}
\caption{Two mirror bases to $dP_6$.}
\label{fig:higher_dp6}
\end{figure}
\begin{example}

We study another model of the mirror to $dP_5$, which differs from
the last one in the position of the second focus-focus point.
Instead of placing it on the ray with generator $(0,-1)$ we move it
to the ray generated by $(1,1)$, as shown in
Figure~\ref{fig:higher_dp5} on the right. This alternative
choice yields the superpotential
\[
W_{dP_5}(\fou) =  \Big(x + y +  xy + \frac{1}{x} +
\frac{1}{y} +
\frac{1}{xy}\Big) \cdot t^2 +
\Big( \frac{1}{y} + x + x^2y  + \frac{1}{xy^2}\Big) \cdot t^3
\]
on the selected chamber $\fou$.
It is an interesting question to understand in detail the effect of
particular choices of singular points and the corresponding
degenerations.
\end{example}

\begin{example}
As a last example, we study broken lines in the mirror of a
distinguished model of $dP_6$, depicted on the left in
Figure~\ref{fig:higher_dp6}. This time we obtain the superpotential
\[
W_{dP_6}(\fou) = \Big(x +  y +  xy + \frac{1}{x} + \frac{1}{y}
+ \frac{1}{xy}\Big) \cdot t^2 + \Big(2 \cdot \frac{1}{xy^2} + \frac{1}{xy} +
2\cdot \frac{1}{y}\Big) \cdot t^3 + \frac{1}{y} \cdot t^4 +
\frac{1}{y} \cdot t^5,
\]
with nine critical points. Again, this potential comes from a special choice of
positions of critical points and root vertex among many others. This
superpotential is ostensibly algebraic although the three walls meeting at the
origin produce an infinite wall structure on the bounded cell $\sigma_0$.

In contrast, Figure~\ref{fig:higher_dp6} on the right shows the mirror base of an
alternative $dP_6$-degeneration featuring a finite wall structure in the bounded
cell $\sigma_0$ by \cite[Expl.\,4.3]{Invitation}. We again obtain an ostensibly
algebraic superpotential
\[
W_{dP_6}(\fou) = \Big(x +  y +  xy +
\frac{1}{x} + \frac{1}{y} + \frac{1}{xy}\Big) \cdot t^2 + \Big(\frac{y}{x} +
\frac{1}{x} + \frac{1}{xy^2} + 2\cdot \frac{1}{y} + \frac{x}{y}\Big)
\cdot t^3.
\]
\end{example}

These examples illustrate that if we leave the realm of toric
geometry, Landau-Ginzburg potentials for del Pezzo surfaces can, at
least locally, still be described by Laurent polynomials, as in the
toric setting.


\section{Singular Fano and smooth non-Fano surfaces}

Having studied smooth Fano surfaces, we now show that our approach
also works if we admit Gorenstein singularities or drop the Fano
condition.


\subsection{Singular Fano surfaces}
\label{Subsect: Singular Fano surfaces}
We now classify the remaining distinguished toric degenerations of del Pezzo
surfaces (Definition~\ref{Def: toric degen dP_k}) from Theorem~\ref{Thm: unique}
with non-simple singularities. In this theorem we have already seen that the
intersection complex $(\check B,\check\P)$ is obtained from a star subdivision
of a reflexive polygon $\Xi$ with a singular point on each of the interior
edges, and then posssibly a further subdivision adding more edges without a
singular point connecting the interior integral point to a non-vertex point on
$\partial\check B$. There are 16~well-known isomorphism classes of reflexive
polygons, among which five give rise to smooth varieties, studied in the last
section. The remaining eleven polygons are characterized by the property that
the dual polygon has at least one integral non-vertex boundary point. These
11~polygons are in one-to-one correspondence with isomorphism classes of
singular toric del Pezzo surfaces.

To obtain a smooth boundary model $(\check B=\Xi,\check \P)$, we now
need to add a non-simple singular point on some of the added edges to
straighten the boundary. Non-simple means that the affine monodromy along a
counterclockwise loop about such a singular point is conjugate to
$\left(\begin{smallmatrix}1&-k\\0&1\end{smallmatrix}\right)$ for some $k>1$.
We refer to $k$ as the \emph{order} of the singular point. Legendre-duality then
yields an affine singularity of the same order~$k$ on the dual edge. As in the
smooth case, we polarize $(\check B,\check\P)$ by the minimal polarizing
function $\check\varphi$ changing slope by one along each interior edge.

Figure~\ref{fig:singular-del-pezzo} depicts the discrete Legendre duals $(B,\P)$
thus obtained from star subdivision of the remaining $11$~reflexive polygons.
Note that the affine monodromy of the singular point is reflected in the shaded
regions. The order~$k$ of a singular point now equals the number of integral
points on the edge of $\P$. We obtain the following addition to
Theorem~\ref{Thm: unique}.

\begin{theorem}
In the situation of Theorem~\ref{Thm: unique} assume that $(\pi:\check\X\to
T,\check\D)$ does not have simple singularities. Then $(\check B,\check\P)$ is
Legendre-dual to one of the cases listed in Figure~\ref{fig:singular-del-pezzo}.
\qed
\end{theorem}

With non-simple singularities on some edges in $(B,\P)$, the slab functions are
not uniquely determined by the gluing data. If $v$ is a vertex on an edge $\rho$
containing an order~$k$ singularity, the slab function $f_{\rho,v}$ has the form
$1+a_1 x+\ldots a_{k-1} x^{k-1}+a_k x^k$ for $x$ the generating monomial for the
$\rho$-stratum of $X_0$. The coefficient $a_k$ is determined by the gluing data,
with $a_k=1$ for trivial gluing data. Thus in any case, there are $k-1$ free
coefficients for each singular point of order $k$. These coefficients reflect
classes of exceptional curves on the resolution of the Fano side. Ignoring these
classes, as suggested by working with the unresolved del Pezzo surface, leads to
the slab function $(1+x)^k$.

\begin{theorem}\footnote{The closely related superpotential of the corresponding
$11$~semi-Fano surfaces obtained by MPCP resolution has independently been
computed in \cite{ChanLau} by other methods. See \cite{MaxThesis} for the
reproduction and comparison of their results with our method.
Add reference to MPCP
}
\label{singular del pezzo}
Let $(\X \to \Spec \kk\lfor t\rfor,W)$ be mirror to a distinguished toric
degeneration $(\check\X\to T,\check\D)$ of del Pezzo surfaces
(Definition~\ref{Def: toric degen dP_k}), and $(B,\P,\varphi)$ the associated
intersection complex for the anticanonical polarization on $\check\X$.

Denote by $\sigma_0$ the bounded cell of the associated intersection complex
$(B,\P,\varphi)$. For an integral point $m$ on an edge
$\omega\subset\partial\sigma_0$ of integral length $k$ define
$N_m=\binom{l}{k}$, where $l$ is the integral distance between $m$ and one of
the vertices of $\omega$. Then it holds
\[
W(\sigma_0) = t\cdot \sum_{m \in \partial \sigma_0\cap\Lambda_{\sigma_0}}
N_m z^m.
\]
\end{theorem}

\begin{figure}[h!]
\input{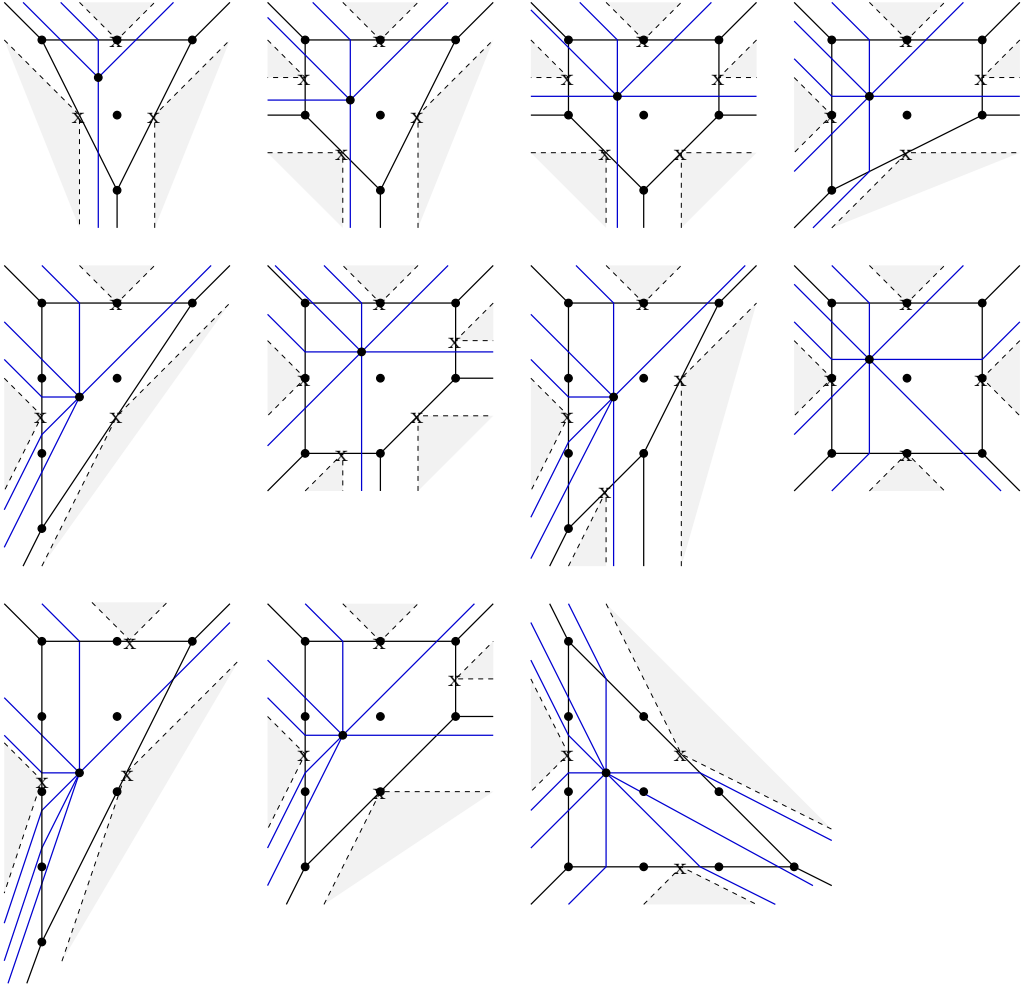}
\caption{The broken lines contributing to the proper superpotential
of the eleven singular toric del Pezzo surfaces.}
\label{fig:singular-del-pezzo}
\end{figure}

\begin{proof}
Corollary~\ref{Cor: reproduce Hori-Vafa} already treats the case with simple
singularities. The remaining 11~cases are most easily done by inspection of the
set of broken lines ending at a specified point $p$ as given in
Figure~\ref{fig:singular-del-pezzo}. Note that there are no walls in
$\Int\sigma_0$, so the result of this computation is independent of the choice
of $p$. It is instructive to check this continuity property explicitly in some
cases.

Alternatively one can argue more generally as follows. If a broken line ends at
$p\in\Int(\sigma_0)$ then the next to last vertex of $\beta$ maps to the
intersection of the ray $p+\RR_{\ge0} m_\beta$ with $\partial\sigma_0$. Denote by
$\omega\subset \partial\sigma_0$ the edge containing this point of intersection.
Then by Lemma~\ref{Lem: Broken lines are rays} there are no more bending points
of $\beta$, and hence the remaining part of $\beta$ has to be parallel to the
unbounded edges of the unbounded cell containing $\omega$. This determines the
kink of $\beta$ when crossing $\omega$. Note that this argument also limits the
exponents $m$ appearing in $W(\sigma_0)$ to be contained in
$\partial \sigma_0\cap\Lambda_{\sigma_0}$. Moreover, each such exponent belongs
to at most one broken line ending at $p$.

Now choose $p$ very close to the unique interior integral point of $\sigma_0$
and let $m\in \partial \sigma_0\cap\Lambda_{\sigma_0}$. Then $p+\RR_{\ge0}$
intersects $ \partial\sigma_0$ very close to $m$. A local computation now shows
that depending on the choice of a vertex $v\in\omega$ the coefficient of $z^m$
comes from either the $l$-th or the $(k-l)$-th coefficient of the associated
slab function $(1+x)^k$. In either case we obtain the stated binomial
coefficient $\binom{l}{k}$.
\end{proof}


\subsection{Hirzebruch surfaces}
As a last application featuring some non-Fano cases, we will study proper
superpotentials for Hirzebruch surfaces $\mathbb{F}_m$. We fix the fan $\Sigma$
in $N \cong \ZZ^2$ of $\mathbb{F}_m$ to be the fan with rays
$\rho_0,\ldots,\rho_3$ generated by the four primitive vectors
\[
v_0 = (0,1),\ v_1 = (-1,0),\ v_2 = (0,-1),\ v_3 = (1,m).
\]
Since $\mathbb{F}_m$ is only
Fano in the cases $\mathbb{F}_0 = \PP^1 \times \PP^1$ and $\mathbb{F}_1 = dP_1$,
for $m\geq2$ the normal fan of the anticanonical polytope is not the fan of a
Hirzeburch surface.

We fix $m$ in the following and denote by $D_{\rho_i}$ the torus-invariant
divisor associated to $\rho_i$. Instead of the anticanonical divisor, which is
not ample for $m\geq 2$, we now consider a smooth divisor $D$ on $\FF_m$ in the
ample class 
\[
D_{\rho_0} + D_{\rho_1} + D_{\rho_2} + m \cdot D_{\rho_3}.
\]
Define the tropical manifold $(\check B, \check \P, \check \varphi)$
with straight boundary as follows. $\check B$ is obtained from the Newton
polytope
\[
\Xi_{D}=\conv\big\{(-1,-1),(2m,-1),(0,1),(-1,1)\big\}
\]
of $D$ by joining each vertex with one of the endpoints of the line segment
$[0,m-1]\times\{0\}\subset \Int\Xi_D$ as shown in~Figure~\ref{fig:Hirzebruch} on
the top, and then introducing a single focus-focus singularity on each of the
four joining one-cells. It is again elementary to check that this makes
$\partial \check B$ totally geodesic. Moreover, setting $\check \varphi(v) = 1$
for all vertices $v$ of $\Xi_D$ and 
\[
\check\varphi(0,0) =  \check\varphi(m-1,0) = 0
\]
defines a strictly convex, integral PL-function $\check\varphi$ on $(\check
B,\check \P)$. 
\begin{figure}
\input{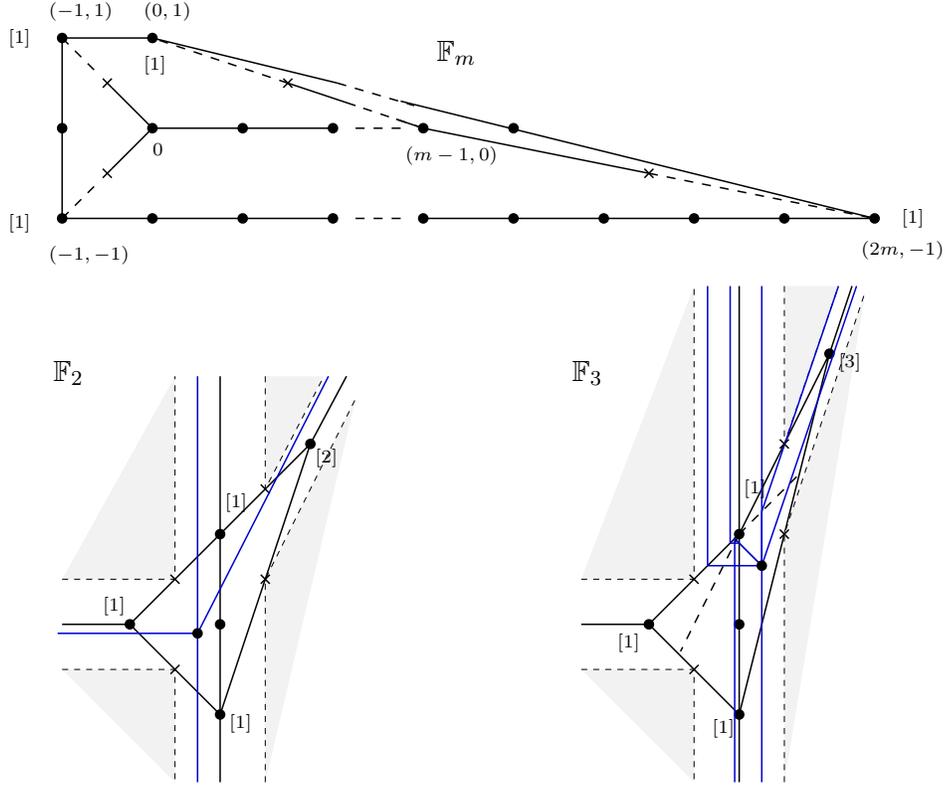}
\caption{A straight boundary model for the Hirzebuch surfaces
$\mathbb{F}_m$ with mirrors for $m=2,3$.}
\label{fig:Hirzebruch}
\end{figure}

The Legendre dual $(B, \P,\varphi)$ has the four vertices $v_0,\ldots,v_3$ we started with, two bounded cells 
\[
\sigma_0 = \conv(v_0,v_1,v_2), \quad \sigma_1 = \conv(v_0,v_2,v_3), 
\]
and four unbounded one-cells contained in $\RR_{\ge0} v_i$, $i=0,\ldots,3$.
These one-cells are indeed parallel in any chart with domain an open set in the
complement of the union of bounded cells. Moreover, $\partial (\sigma_0 \cup
\sigma_1)$ has precisely four integral points, namely $v_0$, $v_1$, $v_2$ and
$v_3$. Note also that by the definition of the Legendre-dual, $\varphi$ is
uniquely determined by
\[
\varphi(v_0) = \varphi(v_1) = \varphi(v_2)= 1,\ \varphi(v_3) = m,
\]
and by the requirement to change slope by one along $\partial(\sigma_0 \cup
\sigma_1)$.

We see that for $m\geq 3$ the union $\sigma_0 \cup \sigma_1$ is non-convex.
In this case the scattering of the two walls emanating from the singular points
on the two edges with vertices $v_0,v_1,v_3$ produce walls entering
$\Int(\sigma_0\cup\sigma_1)$.

For $m>3$ there is even an infinite number of walls to be added to make the wall
structure consistent to all orders. While all these walls are added in the
half-plane below the line $\RR\cdot(1,m)$ and hence there is an open set
contained in $\sigma_0$ not containing any wall, we have not carried out the
necessary analysis to decide if $W(\sigma_0)$ is given to all orders by an
algebraic expression.

We now restrict to the cases $m=2,3$ and explicitly compute the full
superpotentials. 

First, for the case $m=2$ there are no walls in $\Int(\sigma_0\cup\sigma_2)$. The broken lines ending at a specified point $p\in\Int(\sigma_0)$ are depicted on the lower left in Figure~\ref{fig:Hirzebruch}. The Landau-Ginzburg superpotential can then be read off as
 \[
W(\sigma_0) = \Bigl(\frac{1}{x} + \frac{1}{y} + y\Bigr) \cdot t +
xy^2 \cdot t^2.
\]
This is indeed the full potential, as there is no scattering in $\sigma_0\cup
\sigma_1$, so we can apply Lemma~\ref{Lem: Broken lines are rays}.

For $m=3$ let us first compute the walls entering $\Int(\sigma_0\cup\sigma_1)$.
These come from scattering at the point $(0,1)$ of the adjacent edges with
focus-focus singular points in directions $(1,1)$ and $(-1,-2)$. Locally this
scattering situation is equivalent to the scattering of incoming walls from
directions $(-1,0)$ and $(0,-1)$ meeting at the origin. An explicit computation
carried out in~\cite[\S4.1]{Invitation} shows that this scattering diagram can
be made consistent to all orders by introducing outgoing walls in directions
$(1,0)$, $(0,1)$ and $(1,1)$. These translate to walls in directions $(1,1)$,
$(-1,-2)$ and $(0,-1)$ in our situation, as indicated by the dashed lines in the
lower right of Figure~\ref{fig:Hirzebruch}. Of course it will be necessary to
insert more walls \emph{outside} of the bounded part, but this is unessential
for the computation of the superpotential. Hence scattering on the bounded part
$\sigma_0 \cup \sigma_1$ is finite and we can once more apply Lemma~\ref{Lem:
Broken lines are rays}. For the choice of root vertex $p'\in\sigma_1$ as
indicated in the lower right of Figure~\ref{fig:Hirzebruch} we get the
following full superpotential for $(B_3, \P_3)$
 \[
W(\fou) = \Bigl(\frac{1}{x} + \frac{1}{y} + y + \frac{y}{x}\Bigr) \cdot t  + \frac{y}{x} \cdot t^2 +
\Bigl(xy^3  + y^2\Bigr)\cdot t^3.
\]
Thus, neglecting $t$-orders for a moment, we have three new contributions that
differ from the Hori-Vafa potential $\frac{1}{x} + \frac{1}{y} + y + xy^3$,
namely the monomial $y^2$ and twice the term $\frac{y}{x}$. These come from
broken lines that have a break point at the new walls emanating from $(0,1)$ in
direction $(1,1)$, $(0,-1)$ and $(-1,-2)$, respectively. Note that these are
precisely the terms Auroux found in~\cite[Prop.\,3.2]{auroux2}, when we make the
coordinate change $x \mapsto \frac{1}{x}$ and $y \mapsto \frac{1}{y}$.

The computation in~\cite[Prop.\,3.2]{auroux2} is very explicit and rather long
when compared with our derivation. Of course all the hard work is hidden in the
scattering process of \cite{affinecomplex} and the propagation of monomials via
broken lines, but still it is remarkably easy to compute Landau-Ginzburg models
with this approach, once everything is set up.


\section{Three-dimensional examples}

So far we restricted ourselves to $\dim B =2$. We now turn to a few simple
examples illustrating some features of higher dimensional cases.

\begin{example}
\label{Expl: PP3}
Starting from the momentum polytope $\Xi$ for $\PP^3$ with its anticanonical
polarization, Construction~\ref{Const: Reflexive polytope degeneration} provides
a model with a distinguished polarized tropical manifold $(\check B,\check
\P,\check \varphi)$ with $\check B=\Xi$ and with flat boundary. In dimension
three this is done by trading corners \emph{and edges} with a one-dimensional
singular locus of the affine structure. Explicitly, $\check \P$ is the
star-subdivision of
\[
\Xi=\conv\big\{ (2,-1,-1),(-1,2,-1),(-1,-1,2),(-1,-1,-1)\big\}
\]
that introduces six two-faces spanned by the origin and two distinct vertices of
$\Xi$. The discriminant locus $\check\Delta$ is the subcomplex of the first
barycentric subdivision of these six affine triangles shown in
Figure~\ref{fig:three-dimensional}. The affine structure is fixed by the
embedding of $\Xi$ into $\RR^3$ at the origin, and by the affine charts at the
vertices of $\check B$ making $\partial \Xi$ flat and inducing the given affine
chart on the maximal cells of $\check \P$. Finally, $\check\varphi$ is
determined by $\check \varphi(v)=1$ for every vertex $v$ of $\Xi$ and $\check
\varphi(0)=0$.

The discrete Legendre dual $(B,\P,\varphi))$, also drawn in
Figure~\ref{fig:three-dimensional}, has four parallel unbounded rays and a
discriminant locus $\Delta$ with six unbounded rays. Note that unlike in the
case of closed tropical manifolds or in dimension two, $\check\Delta$ and
$\Delta$ are not homeomorphic, but $\check \Delta$ is homeomorphic to the
compactification of $\Delta$ that adds a point at infinity to each unbounded
one-cell of $\Delta$. Every bounded two-face is subdivided into three $4$-gons
by $\Delta$ and at every vertex of $B$ three of these $4$-gons
meet. Denote the bounded three-cell by $\sigma_0$. As in the proof of
Theorem~\ref{singular del pezzo} it now follows that any broken line ending at
$p\in\Int(\sigma_0)$ is straight. This shows
\[
W_{\PP^3}(\sigma_0) = \Big(x + y + z + \frac{1}{xyz}\Big) \cdot t.
\]
\begin{figure}
\input{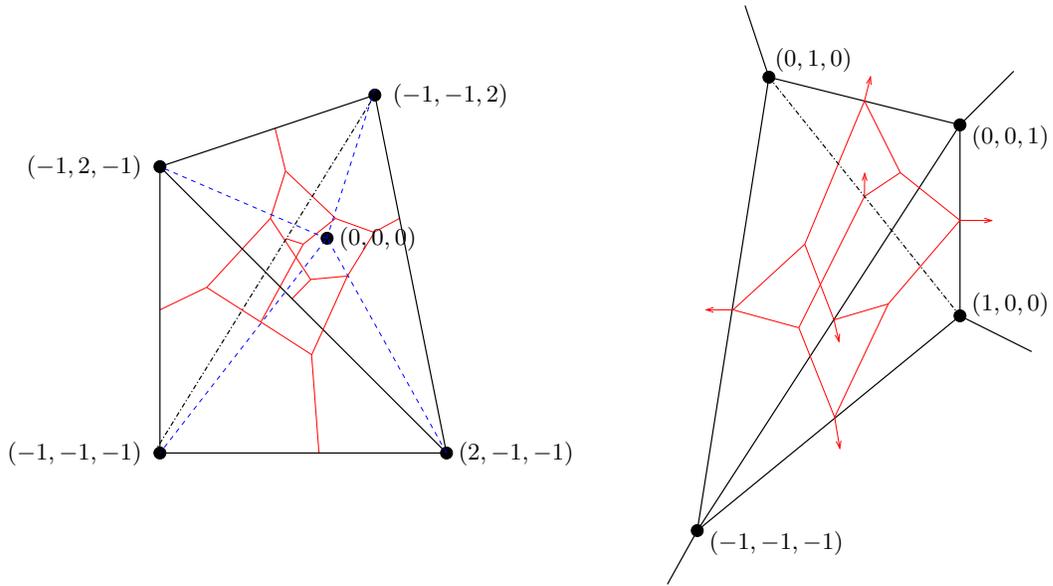}
\caption{The distinguished model of $\PP^3$ and its affine Legendre dual. The
trivalent graphs indicate the discriminant loci, the dashed lines on the left
the one-cells added in the star-subdivision. The little arrows indicate parts of
unbounded cells of the discriminant locus.}
\label{fig:three-dimensional}
\end{figure}
\end{example}

\begin{example}
\label{ex: corti}
Consider the reflexive polytope $\Xi$ depicted on the left in
Figure~\ref{fig:corti}, a truncated tetrahedron with parallel top and bottom
facets that is symmetric under cyclic permutation of the coordinates. The polar
dual is the bounded polyhedron $\sigma_0$ on the right of the same figure.

Note that each edge of $\sigma_0$ has integral length two. This means that the
toric Fano variety $\check X$ with anticanonical Newton polyhedron $\Xi$ is
singular along each one-dimensional toric stratum; each such stratum has a
neighborhood isomorphic to a product of a two-dimensonal $A_1$-singularity with
$\GG_m$.
\begin{figure}
\input{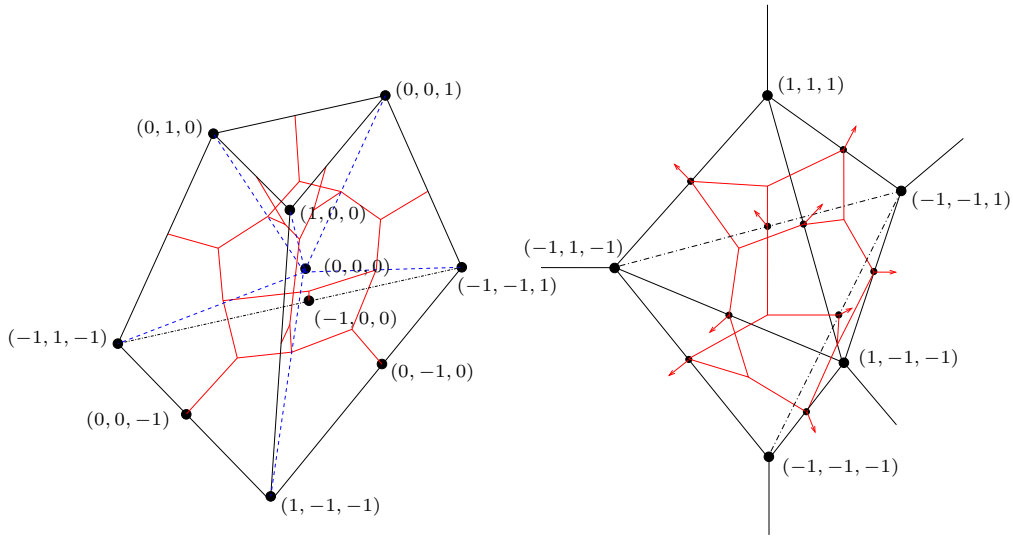}
\caption{A reflexive polytope with edges and corners pushed in and its
Legendre-dual, following the same conventions as in
Figure~\ref{fig:three-dimensional}.}
\label{fig:corti}
\end{figure}

In exactly the same way as in Example~\ref{Expl: PP3}, we arrive at a tropical
manifold $(\check B,\check \P,\check \varphi)$ with flat boundary and $\check
\P$ given by the star subdivision of $\Xi$. The Legendre-dual $(B,\P,\varphi)$
has a double tetrahedron as the unique bounded maximal cell. Both tropical
manifolds are depicted in a chart at the origin in Figure~\ref{fig:corti}. The
discriminant locus $\check\Delta$ of $(\check B,\check \P)$ now is contained in
the union of triangles added in the star-subdivision, with the intersection of
$\check\Delta$ with one such triangle $\tau$ three edges meeting in the
barycenter of $\tau$. The discriminant locus $\Delta$ of $(B,\P)$ has nine rays
emanating from the barycenters of the edges of $\sigma_0$, and then for each
facet $\tau\subset\sigma_0$ again a union of three edges intersecting in the
barycenter of $\tau$.

Now neither side has simple singularities. For example, the affine structure of
$(B,\P)$ at an interior point of an edge of $\Delta$ on $\partial\sigma_0$ is
conjugate to $\left(\begin{smallmatrix} 1&-2&0\\0&1&0\\0&0&1
\end{smallmatrix}\right)$, the square of the monodromy of a focus-focus
singularity times an interval. Thus although it is not hard to see that $(B,\P)$
is compactifiable (Definition~\ref{Def: compactifiable (B,P)}), the assumptions
of \cite[Def.\,1.26]{affinecomplex} can not be fulfilled for this
compactification, and hence the existence of a consistent wall structure is not
immediately clear.
\begin{figure}
\input{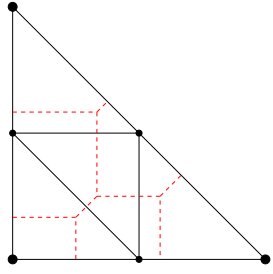}
\caption{Refinement of the discriminant locus on a bounded two-cell of $\P$.}
\label{fig:refinement Fano3}
\end{figure}

In the present case we can proceed as follows. Each of the bounded two-cells of
$(B,\P)$ is integral affine isomorphic to the planar triangle with vertices
$(0,0),(2,0),(0,2)$. Subdivide each such two-cell as in
Figure~\ref{fig:refinement Fano3} and refine $\P$ by intersection with the fan
over the faces of this subdivision. The discriminant locus can then be refined
as indicated by the dashed graph in Figure~\ref{fig:refinement Fano3}, along
with replacing each unbounded cell of $\Delta$ by two copies joined with the
rest of the graph at disjoint trivalent vertices lying on the edges of
$\sigma_0$. Now we can indeed run the algorithm\footnote{There is a technical
problem to assure that the process is finite at each step. This can be done by
working with $d\varphi+\psi$ for $d>0$ and an appropriate $\psi$ or by
inspection of the local scattering situations in the case at hand.} to construct
a compatible sequence $\tilde\scrS_k$ of consistent wall structures. In a second
step undo the refinement process to show that the algorithm indeed works
starting with the non-rigid data $(B,\P,\varphi)$.

As in Theorem~\ref{singular del pezzo} we also have a choice for the initial
slab function, with a distinguished choice a square $f_{\rho,v}=(1+x+y)^2$ for
each bounded $2$-cell $\rho$ and vertex $v\in \rho$.

Now the details of the construction of the wall structure are completely
irrelevant for the computation of the superpotential in the bounded cell
$\sigma_0\in\P$: By the same arguments as in Theorem~\ref{singular del pezzo},
it is again just the sum over broken lines with at most one bend when crossing
$\partial\sigma_0$. Moreover, the set of such broken lines with endpoint at any
$p\in\Int\sigma_0$ are in bijection with the integral points of
$\partial\sigma_0$. For the distinguished slab functions the coefficient carried
by the broken line equals $1$ for the ones without bend and $2$
for the others. The superpotential therefore equals
\[
W(\sigma_0) =  \Big(xyz + \frac{1}{xyz} + \frac{x}{yz}  + \frac{y}{xz}  +
\frac{z}{xy} +2(x^2+y^2+z^2)+\frac{2}{x^2}+\frac{2}{y^2}+\frac{2}{z^2}+\frac{2}{xy}+\frac{2}{yz}+\frac{2}{xz}\Big) \cdot t .
\]

A similar challenge concerns the existence of a consistent wall structure on
$(\check B,\check \P,\check\varphi)$, but analogous arguments apply. The
resulting toric degeneration $(\check\foX\to T,\check\foD)$ then fits into an
algebraizable two-parameter family with the singular toric Fano manifold $\check
X$ with its toric anticanonical divisor $\check D$ as another fiber. We have not
performed a more detailed analysis to identify this family. Possibly it is just
isomorphic to a deformation of $\check D$ inside $\check X$, as in
Example~\ref{Expl: PP2} for $\PP^2$ and $\check D$ a family of elliptic
curves.
\end{example}

\section*{Concluding remarks.}
It would be interesting to more systematically analyze Landau-Ginzburg models
for non-toric Fano threefolds with our method. In~\cite{Pr1, Pr2} so called
\emph{very weak Landau-Ginzburg potentials} are found. The terms and
coefficients of these Laurent polynomials have to be chosen very carefully. As
the potentials presented there do not come from a specific algorithm, but rather
are written down in an ad hoc way, one would like to have an interpretation of
the terms occurring. One might ask whether there are toric degenerations
reproducing the potentials in~\cite{Pr1, Pr2} via tropical disk counting, as in
the examples here.


\end{document}